\documentclass[reqno, 12pt]{amsart}

\usepackage[utf8]{inputenc}
\usepackage[T1]{fontenc}
\usepackage[english]{babel}
\usepackage{libertine} 
\usepackage[libertine]{newtxmath}
\usepackage{geometry}
\usepackage{fullpage}
\usepackage{xcolor}
\usepackage[inline]{enumitem}
\usepackage[linkcolor=blue, citecolor=green, colorlinks=true]{hyperref}
\usepackage{cleveref}
\usepackage{amsmath}
\usepackage{array,longtable}
\newcolumntype{C}{>{$}c<{$}}
\newcolumntype{L}{>{$}l<{$}}    

\title{An effective version of the Kuznetsov trace formula for $\GSp(4)$}
\date{\today}

\author[Comtat]{Félicien Comtat}
\address{Mathematisches Institut, Endenicher Allee 60, 53115 Bonn, Germany}
\email{comtat@math.uni-bonn.de}

\author[Lesesvre]{Didier Lesesvre}
\address{Université de Lille -- Laboratoire Paul Painlevé, UMR 8524,
 59000 Lille, France}
\email{didier.lesesvre@univ-lille.fr}

\author[Man]{Siu Hang Man}
\address{Charles University, Faculty of Mathematics and Physics, Department of Algebra, Sokolovská 49/83, 186 75 Praha 8, Czech Republic}
\email{shman@karlin.mff.cuni.cz}

\newtheorem{thm}{Theorem}[section]
\newtheorem{prop}[thm]{Proposition}
\newtheorem{lem}[thm]{Lemma}

\newtheorem{conj}[thm]{Conjecture}
\theoremstyle{remark}
\newtheorem{rk}{Remark}
\crefname{thm}{theorem}{theorems}
\crefname{lem}{lemma}{lemmata}
\crefname{prop}{proposition}{propositions}

\allowdisplaybreaks

\numberwithin{equation}{section}

\newcommand{\R}{\mathbf{R}}
\newcommand{\C}{\mathbf{C}}
\newcommand{\Z}{\mathbf{Z}}
\newcommand{\N}{\mathbf{N}}
\newcommand{\Q}{\mathbf{Q}}
\newcommand{\A}{\mathbf{A}}
\newcommand{\Ad}{\operatorname{Ad}}
\newcommand{\bs}{\backslash}
\newcommand{\GL}{\operatorname{GL}}	
\newcommand{\GSp}{\operatorname{GSp}}
\newcommand{\Kl}{\operatorname{Kl}}
\newcommand{\Sp}{\operatorname{Sp}}	
\newcommand{\SO}{\operatorname{SO}}
\newcommand{\Sym}{\operatorname{Sym}}
\newcommand{\id}{\operatorname{id}}
\newcommand{\diag}{\operatorname{diag}}

\newcommand{\pb}[1]{\left\langle#1\right\rangle}
\newcommand{\rb}[1]{\left(#1\right)}
\newcommand{\cb}[1]{\left\{#1\right\}}
\newcommand{\cbc}[2]{\left\{#1\ :\ #2\right\}}
\newcommand{\vb}[1]{\left|#1\right|}
\newcommand{\ol}[1]{\overline{#1}}
\newcommand{\y}{\operatorname{y}}
\newcommand{\vol}{\mathrm{vol}}
\newcommand{\Spin}{\mathrm{Spin}}
\newcommand{\Std}{\mathrm{Std}}
\renewcommand{\Re}{\operatorname{Re}}
\renewcommand{\Im}{\operatorname{Im}}
\renewcommand{\le}{\leqslant}

\renewcommand{\ge}{\geqslant}
\renewcommand{\geq}{\geqslant}

\newcommand{\bp}{\begin{pmatrix}}
\newcommand{\ep}{\end{pmatrix}}
\newcommand{\bsm}{\begin{smallmatrix}}
\newcommand{\esm}{\end{smallmatrix}}
\newcommand{\bv}{\begin{vmatrix}}
\newcommand{\ev}{\end{vmatrix}}

\begin{document}

\begin{abstract}
We develop an explicit version of the Kuznetsov trace formula for $\GSp(4)$, relating sums of Fourier coefficients to Kloosterman sums.
We study the precise analytic behaviour of both the spectral and the arithmetic transforms arising in the Kuznetsov trace formula for $\GSp(4)$. We use these results to provide an effective version of the trace formula, and establish various results on the family of Maaß automorphic forms on $\GSp(4)$ in the spectral aspect: the Weyl law, a density result on the non-tempered spectrum, large sieve inequalities, bounds on the second moment of the spinor and standard $L$-functions, as well as a statement on the distribution of the low-lying zeros of these $L$-functions, determining the associated types of symmetry.
\end{abstract}

\maketitle

\setcounter{tocdepth}{1}
\tableofcontents

\thispagestyle{empty}

\section{Introduction}

\subsection{Automorphic forms and averages}

Number theory is at the junction of many mathematical fields. Among its important topics are elliptic curves, modular forms, Maaß waveforms, and Galois representations. Despite their diversity, these objects are different faces of a single one: \textit{automorphic representations}.

Isolated automorphic forms however remain elusive to study, even in the case of $\GL(2)$. A leading philosophy, originating with Sarnak, is to consider automorphic forms in \textit{families} and to seek results on average, called \textit{arithmetic statistics}. These results are often restricted to families of automorphic forms for $\GL(2)$ or $\GL(3)$ with  varying aspects (level, weight, eigenvalue, etc.), called \textit{harmonic subfamilies}. It remains crucial and challenging to address the case of more general reductive groups, and the aim of this paper is to establish such arithmetic statistics in the case of the symplectic group $\GSp(4)$.

Trace formulas are the central tools in this approach, as they relate in various settings averages on families of automorphic forms and more explicit arithmetic or geometric quantities. Such trace formulas led successfully to various results on families of automorphic forms on $\GL(2)$ and~$\GL(3)$ in the recent decades, e.g. Weyl laws \cite{muller}, global Plancherel equidistribution \cite{shin}, Sato-Tate conjectures \cite{st}, subconvexity results \cite{MV}, etc.  The relative trace formula of Jacquet~\cite{jacquet} is a variation of the trace formulas that can be instantiated in the guise of the Kuznetsov trace formulas, which from very far apart look like distributional equalities of the form
\begin{equation}
\label{eq:ktf-high-level}
    \sum_\varpi \tilde{F}(\varpi) = \sum_w \sum_{c \in \mathbf{N}^r} \mathcal{J}_{w, F}(c)
\end{equation}
where $F$ is a test function satisfying certain smoothness and decay properties, and $r$ is the rank of the underlying group. Here, the sum over $\varpi$ is a sum over the generic spectrum of automorphic forms, the sum over $w$ is a sum over the Weyl elements of the underlying group, and $\mathcal{J}_{w, F}(c)$ is made of arithmetic and analytic quantities, typically involving Kloosterman sums and oscillating integrals. 
Even with such formulas, it remains challenging to derive results for two reasons: it is necessary to precisely understand the transform~$\tilde{F}(\varpi)$ on the spectral side to select the desired family and the relevant statistics one wants to study, and also to understand the properties of the transform on the arithmetic side $\mathcal{J}_{w,F}$.

Although there are very general settings for trace formulas, both classical \cite{arthur} and relative~\cite{jacquet}, allowing us to write down such equalities for almost any reductive group, precise and exploitable statements are only developed for very few and specific cases, namely $\GL(2)$ \cite{kuznetsov}, $\GL(3)$ \cite{blomer,bbm,buttcane}, and recently also for $\GSp(4)$ \cite{assing,comtat,man}. The aim of this paper is to address the two challenges outlined above and determine finely enough the properties of the transforms to obtain results on families of automorphic forms for GSp(4) in the spectral aspect.

\subsection{Main technical results}

In order to be able to use the Kuznetsov trace formula \eqref{eq:ktf-high-level} to obtain arithmetic statistics on the automorphic forms of $\GSp(4)$, it is necessary to have a precise understanding of the behaviour of the transforms $\tilde{F}$ and $\mathcal{J}_{w, F}$. The main tools of this paper give such knowledge, providing:
\begin{itemize}
    \item an explicit version of the Kuznetsov trace formula (see \Cref{thm:ktf});
    \item explaining the behaviour of the spectral transform $\tilde{F}$ (see \Cref{thm:sl});
    \item explaining the behaviour of the arithmetic transform~$\mathcal{J}_{w, F}$ (see \Cref{prop:Jw_ls}).
\end{itemize}
We explain from a higher level point of view these two last results in this section. 

\subsubsection{Localisation properties of the spectral transform}

To an automorphic form $\varpi$ of $\GSp(4)$ one associates an (archimedean) spectral parameter $\mu = (\mu_1, \mu_2)$, as explained in \Cref{subsec:gsp4-automorphic-forms}, and we will write $\varpi = \varpi_\mu$ to emphasise this parameter. \Cref{thm:sl} states that for $\tau = (\tau_1, \tau_2) \in \R_+^2$ a ``target'' parameter, there is a suitable test function $F_\tau$ such that  $\tilde{F}_\tau(\varpi_\mu)$ is vanishingly small when~$\tau$ and $\mu$ are far apart, and has controllable size when~$\tau$ and $\mu$ are close, so that $\tilde{F}_\tau$  behaves like a spectral bump function localising around $\tau$. More precisely the statement (A) of \Cref{thm:sl} is of the form
\begin{equation}
    \vb{\tilde{F}_\tau(\varpi_\mu)} \ll \rb{1 + |\mu_1-\tau_1| + |\mu_2-\tau_2|}^{-A}
\end{equation}
for all $ A \geqslant 1$ and $\tau_1, \tau_2 \gg_A 1$. Moreover, for $|\mu_i - \tau_i| \leqslant c$, where $c$ is a certain absolute constant, we understand the exact size of $|\tilde{F}_\tau(\varpi_\mu)|$ as explained in the statement (B) of \Cref{thm:sl}.

The temperedness of the automorphic representation $\pi$ is quantified by $\Re(\mu_i)$: explicitly, $\pi$ is tempered if and only if $\Re(\mu_i) = 0$. Result (C) of \Cref{thm:sl} gives a tool to amplify the non-tempered part of the spectrum. More precisely, it states that there exists a test-function~$F_X$ such that $\tilde{F}_X(\varpi_\mu)$ has size about $X^{\max\{|\Re \mu_i|\}}$, allowing one to select more significantly the contribution from the non-tempered spectrum. 

In summary, these statements allow one to precisely select families of spectral parameters in a certain range.

Here we give some comments on the proof of \Cref{thm:sl}. The standard proofs of analogous statements for $\GL(2)$ \cite{di} and $\GL(3)$ \cite{blomer,gk} make use of the inverse Mellin transform of the Whittaker function, which consists of $\Gamma$-quotients, whose analytic behaviour is well understood via Stirling's approximation. Meanwhile, by Ishii \cite{ishii05} we know that the inverse Mellin transform of the $\GSp(4)$ Whittaker function consists of not only $\Gamma$-quotients but also generalised hypergeometric functions~${}_3F_2$ evaluated at the edge of convergence, whose analytic behaviour is much more difficult to analyse. To bypass this obstacle, we made use of an alternative integral representation of the $\GSp(4)$ Whittaker function as a four-fold Mellin-Barnes integral \eqref{eq:IshiiMoriyama} by Ishii and Moriyama \cite{im}. While this integral representation consists only of $\Gamma$-quotients, the presence of extra contour integrals means that extra residues are picked up when moving contours, and these obscure the true analytic behaviour of the spherical transform. A careful analysis of these extra residues shows that they cancel in pairs, revealing the true analytic behaviour of the spherical transform, establishing \Cref{thm:sl}.

\subsubsection{Localisation properties of the arithmetic transforms}

The arithmetic side of the Kuznetsov trace formula has a term corresponding to the identity Weyl element $w = \mathrm{id}$, which has a straightforward expression in terms of the test function $F$. As a rule of thumb, this identity term is often expected to be the most important contribution to the arithmetic side, giving the leading term in various asymptotic results on families of automorphic forms. It is therefore needed to control sufficiently the other terms in order prove that they are negligible compared to the identity contribution.

\Cref{prop:Jw_trivial} provides bounds on the integral transforms $\mathcal{J}_{w, F}(c)$ in a coupled fashion, depending on $c,w$ as well as the ``target'' parameter $\tau$ of the test function $F=F_\tau$. In particular, this proves that the contributions of $\mathcal{J}_{w, F}(c)$ are negligible for $c$ outside a certain range, and effectively truncates the arithmetic sum over $c \in \mathbf{N}^2$.

Finer estimates on $\mathcal{J}_{w, F}(c)$ need to be used in the remaining range, and are given in \Cref{prop:Jw_ls}. They are obtained from estimating the volumes of the essential ranges on which the corresponding integrals concentrate, which are quasi-algebraic sets. Exploiting the extra structure of these quasi-algebraic sets allows us to give bounds beyond the one obtained by trivially bounding each variable; and in certain cases exploiting the oscillations displayed in $\mathcal{J}_{w, F}(c)$ by the stationnary phase method leads to extra savings in \Cref{prop:Jw_ls}.

These statements allow us to control sufficiently the arithmetic side and the contributions beyond the identity term in order to derive various results.

\subsection{Consequences}
\label{subsec:consequences}

Many fine arithmetic statistics on the automorphic spectrum of $\GSp(4)$ in the spectral aspect can be obtained building from the precise knowledge on the transforms given in the above theorems. To state the results, we pick an orthogonal basis $\{\varpi\}$ of arithmetically normalised Hecke--Maaß cusp forms, which spans the spherical spectrum of $\GSp(4)$ (see \Cref{subsec:gsp4-automorphic-forms}).

\subsubsection{Weyl law}

Weyl laws are important in the realm of automorphic forms \cite{muller}, as they provide the asymptotic size of families and conjecturally \cite{lv} relate to the geometry of the underlying group. The following theorem gives a weighted count of Hecke--Maaß cusp forms with spectral parameters lying in a small box.

\begin{thm}[Weyl law]\label{thm:Weyl_law}
There are absolute constants $T_0, K\ge 1$ such that for all $T_0 \le T_1, T_2-T_1$ we have
\begin{equation}
\sum_{\substack{|\mu_1(\varpi) - iT_1| \le K\\ |\mu_2(\varpi) - iT_2| \le K}} \|\varpi\|^{-2} \asymp \mathcal{T},
\end{equation}
where {$\mathcal T := T_1T_2(T_2+T_1)(T_2-T_1)$}. 
\end{thm}

\begin{rk}
The Weyl law is a description of the asymptotic size of a family of automorphic forms.  One standard example is the family $\mathscr F_T$ of automorphic forms with spectral parameters $\vb{\mu_i}\le T$. The norm $\|\varpi\|^2$ is given by the adjoint $L$-value $L(1,\varpi,\Ad)$, see \Cref{lem:adlv}. It is standard to obtain an upper bound $\|\varpi\|^2 \ll T^\varepsilon$, which implies an upper bound $|\mathscr F_T| \ll T^{6+\varepsilon}$ after integrating over all the parameters $T_i \leqslant T$. However, a good lower bound of the form~$\|\varpi\|^2 \gg T^{-\varepsilon}$ is only known conditionally on the non-existence of Siegel zeros for the adjoint $L$-function, which remains unproven. If we assume this lower bound, then \Cref{thm:Weyl_law} agrees with the general statement of Lindenstrauss and Venkatesh \cite{lv}, claiming that the family has asymptotic  size~$|\mathscr F_T| \asymp T^d$ where $d=6$ is the dimension of the associated Riemannian symmetric space, in this case $\Sp(4, \mathbf{Z})\backslash \Sp(4,\R)/K$.
\end{rk}

\begin{rk}
    The Plancherel measure on the set of spectral parameters $(\mu_1, \mu_2)$ is the absolutely continuous measure $|c(\mu_1, \mu_2)|^{-2}d\mu_1 d\mu_2$, where the density $c(\mu_1, \mu_2)$ is explained in \cite[(3.2)]{blomer-pohl}, and is asymptotically
    \begin{equation}
    \label{eq:plancherel-density}
        |c(\mu_1, \mu_2)|^{-2} \asymp \mu_1 \mu_2 (\mu_2+\mu_1) (\mu_2-\mu_1)
    \end{equation}
    when $\mu_2 - \mu_1 \gg 1$. Therefore, the quantity $\mathcal{T}$ arising in \Cref{thm:Weyl_law} genuinely is the Plancherel measure of the spectral ball of radius $K$ centered at $(iT_1, iT_2)$. \Cref{thm:Weyl_law} therefore confirms what can be expected from the Plancherel inversion formula for the size of the family. It is worth noting from \eqref{eq:plancherel-density} that the Plancherel density drops when the spectral parameters $\mu_1$ and $\mu_2$ are not well-spaced, explaining why statements and proofs may appeal to case-splitting to treat these density-dropping zones (which are exactly the walls of the Weyl chambers). 
\end{rk}

\subsubsection{Exceptional spectrum}

Among the $\GSp(4)$ automorphic forms, unlike in the $\GL(n)$ setting, there are genuine non-tempered forms, as given for instance by the Saito-Kurokawa lifts \cite{schmidt}.
These are however also non-generic, and the generalized Ramanujan conjecture states that there are no generic and non-tempered automorphic representations. Even though such a statement seems out of reach, the following result quantifies that, even if they exist, there are very few generic non-tempered automorphic representations.

\begin{thm}[Non-tempered spectrum]\label{thm:nts}
For any $K > 0$ and any $\varepsilon > 0$ we have
\begin{equation}
\sum_{\substack{\varpi \textup{ non-tempered}\\ |\Im \mu_1(\varpi) - T| \le K \\ |\Im \mu_2(\varpi) - T| \le K}} T^{\frac{34}{15} R(\varpi)} \ll_{K, \varepsilon} T^{3+\varepsilon},
\end{equation}
where $R(\varpi) := \max\{|\Re \mu_1(\varpi)|, |\Re \mu_2(\varpi)|\}$. 
\end{thm}

Such a density theorem is a result towards Sarnak's density hypothesis \cite{sarnak-xue}; it is often sufficient in practice to rule out the non-tempered case, guaranteeing that they do not impact results on average. See also \cite{assing,man} for a result in the level aspect.

\subsubsection{Large sieve inequalities}

A standard way to understand an automorphic form $\varpi$ is to study its Fourier coefficients $A_\varpi(m,n)$, defined by certain integral periods. In general, these are not easily computable but can be managed by orthogonality-on-average statements, usually known as \textit{large sieve inequalities}. Of particular interest are the coefficients $A_\varpi(1,n)$ and $A_\varpi(n,1)$, which are the coefficients of the spinor L-function and the standard L-function, respectively (once factorized by a zeta factor for the latter, see Section \ref{subsec:$L$-functions}).
 
\begin{thm}[Large sieve]\label{thm:ls} 
Let $N\ge 1$, $T_1, T_2\ge 3T_1$ be sufficiently large, and $(\alpha(n))_{n\in\N}$ a sequence of complex numbers. Then we have
\begin{align}\label{eq:ls1}
\sum_{\substack{T_1 \le |\mu_1(\varpi)| \le 2T_1\\ T_2 \le |\mu_2(\varpi)| \le 2T_2}} \vb{\sum_{n\le N} \alpha(n) A_{\varpi}(1,n)}^2 &\ll \mathcal{T} T_1T_2 N \|\alpha\|_2^2 \rb{N^{-1}+N^{2/3}T_1^{-1}T_2^{-7/3}+N^2T_1^{-1}T_2^{-9/4}} (NT_2)^\varepsilon  ,\\
\sum_{\substack{T_1 \le |\mu_1(\varpi)| \le 2T_1\\ T_2 \le |\mu_2(\varpi)| \le 2T_2}} \vb{\sum_{n\le N} \alpha(n) A_{\varpi}(n,1)}^2 &\ll \mathcal{T} T_1T_2 N \|\alpha\|_2^2 \rb{N^{-1} + N^{5/2}T_1^{-1}T_2^{-9/4}} (NT_2)^\varepsilon  .\label{eq:ls2}
\end{align}
\end{thm}
The trivial bounds of the sum on the left hand side above, assuming that the Fourier coefficients are essentially of constant size, are $\mathcal{T} T_1T_2 N \|\alpha\|_2^2 \asymp N T_1^{2+\varepsilon}T_2^{4+\varepsilon} \|\alpha\|_2^2$, and the statements therefore provide non-trivial cancellations in certain ranges.

\subsubsection{Second moments of $L$-functions}

The $L$-functions attached to an automorphic form $\varpi$ encapsulate a lot of information, in particular their central values are related to various arithmetic and geometric quantities. For an automorphic form $\varpi$ of $\GSp(4)$, two $L$-functions naturally arise, namely the spinor and standard $L$-functions (see Section \ref{subsec:$L$-functions}), whose Dirichlet series coefficients are associated to the Fourier coefficients $A_\varpi(1,n)$ and $A_\varpi(n,1)$ respectively, and those converge on a right half-plane. The above large sieve inequalities allow to study the second moments of the central values of these $L$-functions.

\begin{thm}[Second moment for the spinor $L$-functions] \label{thm:2mspin} 
We have, for $T_1, T_2\ge 3T_1$ sufficiently large and $M\in\N$,
\begin{equation}
\label{eq:ls-spinor}
\sum_{\substack{T_1 \le |\mu_1(\varpi)| \le 2T_1\\ T_2 \le |\mu_2(\varpi)| \le 2T_2}} \vb{L(\tfrac12,\varpi,\Spin)}^2 \ll \mathcal{T}(T_1T_2)^{1+\varepsilon} \rb{T_1^2 T_2^{3/4} + T_1^{-M} T_2} .
\end{equation}
\end{thm}

\begin{rk}
Provided the weights in the Weyl law are well-controlled, viz. $T^{-\varepsilon} \ll \|\varpi\| \ll T^\varepsilon$, the length of the above sums are of size $\mathcal{T}T_1T_2 \asymp T_1^2T_2^2(1+T_1+T_2)(1+|T_1-T_2|)$ and we can wonder how far are these statements from the Lindelöf conjecture, which claims $L(\tfrac12, \varpi, \bullet) \ll \mathcal{T}^\varepsilon$, on average. As explained in Section \ref{subsec:$L$-functions}, the analytic conductor of $L(s,\varpi,\Spin)$ has asymptotic size $\asymp (T_1T_2)^2$. Therefore, the convexity bound for the spinor $L$-function would be $|L(\tfrac12, \varpi, \Spin)|^2 \ll (T_1T_2)^{1+\varepsilon}$, while the statement \eqref{eq:ls-spinor} indicates that the bound is $T_1^2 T_2^{3/4}$ on average, which is beyond convexity in the case $T_2 \gg T_1^{4+\delta}$.
\end{rk}

\begin{thm}[Second moment for the standard $L$-functions]\label{thm:2mstd}  
We have, for $T_1, T_2\ge 3T_1$ sufficiently large,
\begin{equation*}
\sum_{\substack{T_1 \le |\mu_1(\varpi)| \le 2T_1\\ T_2 \le |\mu_2(\varpi)| \le 2T_2}} \vb{L(\tfrac12,\varpi,\Std)}^2 \ll \mathcal{T}(T_1T_2)^{1+\varepsilon} \begin{cases} (T_1^{-1}T_2^{19/4}+T_2^2) & \text{assuming Langlands conjecture},\\ T_1T_2^{58/11} & \text{otherwise.}\end{cases}
\end{equation*}
\end{thm}

\begin{rk}
The Langlands functoriality conjecture postulates the automorphy of $L(s,\varpi, \Std)$, i.e. the fact that this L-function can be described as an automorphic L-function on a certain general linear group. It is known in the case of the spinor $L$-function $L(s,\varpi,\Spin)$ by Gan and Takeda~\cite{gt} but not for the standard $L$-function $L(s,\varpi,\Std)$. Automorphicity in particular implies strong bounds on the non-temperedness of the spectral parameters, i.e. on $|\Re \mu_i|$, sparkling the two versions of the above result.
\end{rk}

\subsubsection{Quantitative quasi-orthogonality statements}

An easily usable version of the Kuznetsov trace formula, which has already incorportated the behaviour and the bounds for the arithmetic and spectral transforms, is given by the following.

\begin{thm}[Quantitative quasi-orthogonality]\label{thm:quasi_orthogonality}
Let $T_1, T_2\ge 3T_1$ be sufficiently large parameters, and $M,N\in\N^2$. Then we have 
\begin{equation}
\sum_{\varpi} A_{\varpi}(n_1, n_2) \overline{A_{\varpi}(m_1, m_2)} \frac{h_{T_1,T_2}(\mu(\varpi))}{\|\varpi\|^2} = \delta_{m_i=n_i} \sum_{\varpi}  \frac{h_{T_1,T_2}(\mu(\varpi))}{\|\varpi\|^2} + O\rb{S},
\end{equation}
where
\begin{equation}
    S = \mathcal{T}T_1T_2 \Big(T_1^{-1}T_2^{-2}(m_1m_2n_1n_2)^\theta + T_1^{-1}T_2^{-9/4}(m_1n_1)^{5/4}(m_2n_2)\Big) (T_1T_2m_1m_2n_1n_2)^\varepsilon.
\end{equation}
Here $\theta = 7/64$ is a bound towards Ramanujan conjecture for $\GL(2)$, the function $h_{T_1,T_2}$ is non-negative and uniformly bounded on the strip $\cb{|\Re(\mu_1)|\le 1/2} \times \cb{|\Re(\mu_2)|\le 1/2}$, such that $h_{T_1,T_2}\asymp 1$ in the box
\[
\cbc{(\mu_1,\mu_2)}{T_1\le \Im(\mu_1) \le 2T_1,\ T_2\le \Im(\mu_2) \le 2T_2,\ |\Re\mu_1|, |\Re\mu_2|\le 1/2},
\]
and everywhere we have the bound, for all $A \geqslant 1$, 
\begin{equation}
    h_{T_1,T_2}(\mu_1,\mu_2) \ll_A \rb{(1+|\mu_1|/T_1)(1+|\mu_2|/T_2)}^{-A}.
\end{equation}
\end{thm}

\subsubsection{Low-lying zeros and the density conjecture}

The spacings of zeros of families of $L$-functions are well understood: they are distributed following a universal law, independent of the exact family under consideration, as proven by Rudnick and Sarnak \cite{RudSar94}. This recovers the behaviour of spacings between eigenangles of the classical groups of random matrices.
However, the distribution of \textit{low-lying} zeros attached to reasonable families of $L$-functions depends upon the specific setting under consideration; see \cite{sst} for a discussion in a general setting.

More precisely, let $L(s, f)$ be an $L$-function attached to an arithmetic object $f$. Consider its non-trivial zeros written in the form $\rho_f = \frac{1}{2}+i\gamma_f$ where $\gamma_f$ is a priori a complex number. There is a notion of analytic conductor $c(f)$ of $f$ quantifying the number of zeros of $L(s,f)$ in a given zone; concretely, if $C$ is a sufficiently large constant, and let $N(f,C)$ denote the number of zeros~$\rho_f$ such that $0 \leqslant |\gamma_f| \leqslant C$, then we have $N(f,C) \sim C\log (c(f))/ 2\pi$.
We renormalise the mean spacing of the zeros to $1$ by setting~$\tilde{\gamma}_f = {\log (c(f))} \gamma_f / {2\pi}$. Let $\phi$ be an even Schwartz function on $\mathbf{R}$ whose Fourier transform is compactly supported, in particular it admits an analytic continuation to all of $\mathbf{C}$. The one-level density attached to $f$ is defined by
\begin{equation}
\label{def:one-level-density}
D(f,\phi) := \sum_{\gamma_f} \phi\left(\tilde{\gamma}_f\right).
\end{equation}
The analogy with the behaviour of small eigenangles of random matrices led Katz and Sarnak~\cite{katz_zeroes_1999} to formulate the so-called \emph{density conjecture}, claiming the same universality for the types of symmetry of families (understood in a reasonable sense, see Sarnak--Shin--Templier~\cite{sst}) of $L$-functions as those arising for classical groups of random matrices.

\begin{conj}[Katz-Sarnak] 
\label{conj:katz-sarnak}
Let $\mathcal{F}$ be a family of $L$-functions in the sense of Sarnak, and $\mathcal{F}_X$ a finite truncation increasing to $\mathcal{F}$ when $X$ grows. Then there is one classical group $G$ among $\mathrm{U}$, $\mathrm{SO(even)}$, $\mathrm{SO(odd)}$, $\mathrm{O}$ or $\mathrm{Sp}$ such that for all even Schwartz function $\phi(x)$ on $\mathbf{R}$ with compactly supported Fourier transform, 
\begin{equation}
\label{density-conjecture}
\frac{1}{|\mathcal{F}_X|} \sum_{f \in \mathcal{F}_X} D(f,\phi) \xrightarrow[X\to\infty]{} \int_{\mathbf{R}} \phi(x) W_G(x)dx, 
\end{equation}
where $W_G(x)$ is the explicit distribution function modeling the distribution of the eigenangles of the corresponding group of random matrices, explicitly: 
\begin{equation}
    \begin{array}{rclcrcl}
         W_{\mathrm{O}}(x) & =  &\displaystyle 1 + \frac{1}{2}\delta_0(x), & \qquad &
W_{\mathrm{SO(even)}}(x) & =  &\displaystyle 1+ \frac{\sin 2\pi x}{2 \pi x}, \\[1em]
W_{\mathrm{SO(odd)}}(x) & = &\displaystyle  1 - \frac{\sin 2\pi x}{2\pi x} + \delta_0(x), & \qquad &
W_{\mathrm{Sp}}(x) & = &\displaystyle   1 - \frac{\sin 2\pi x}{2\pi x}.
    \end{array}
\end{equation}

The family $\mathcal{F}$ is then said to have the \emph{type of symmetry} of~$G$.
\end{conj}

In this direction, we obtain the following restricted statement towards this conjecture, for the families of spinor and standard $L$-functions attached to automorphic forms of $\GSp(4)$. 

\begin{thm}[Type of symmetry]
\label{thm:llz}
Let $T_1$ and $T_2\ge 3T_1$ be sufficiently large parameters, and the test function $h_{T_1,T_2}$ be as in \Cref{thm:quasi_orthogonality}. Suppose $T_2 \asymp T_1^t$ for some $t\ge 1$. Let $\delta>0$, and $\phi$ a Schwartz function whose Fourier transform $\widehat\phi$ is supported in $(-\delta,\delta)$. Let $c_{\Spin} = c_{\Spin,T_1,T_2} \asymp T_1^2T_2^2$ and let~$c_{\Std} = c_{\Std,T_1,T_2} \asymp T_2^4$. For a Hecke--Maaß cusp form $\varpi$ of $\GSp(4)$ and $\bullet\in\{\Spin,\Std\}$, denote~$\rho_{\varpi,\bullet} = \frac 12+i\gamma_{\varpi,\bullet}$  the nontrivial zeros of $L(s,\varpi,\bullet)$. Define $\widetilde \gamma_{\varpi,\bullet} = \log(c_\bullet)\gamma_{\varpi,\bullet}/2\pi$, and
\begin{equation}
D_\bullet(\varpi,\phi) := \sum_{\gamma_{\varpi,\bullet}} \phi\rb{\widetilde\gamma_{\varpi,\bullet}}.
\end{equation}
Then we have
\begin{align}
H^{-1} \sum_{\varpi} D_{\Spin}(\varpi,\phi) \frac{h_{T_1,T_2}(\mu(\varpi))}{\|\varpi\|^2} &\xrightarrow[T_1 \to \infty]{} \int_\R W_{\rm O} \phi = \widehat\phi(0) + \frac 12 \phi(0), & &\forall \delta < \frac{4+9t}{12(1+t)},\\
H^{-1} \sum_{\varpi} D_{\Std}(\varpi,\phi) \frac{h_{T_1,T_2}(\mu(\varpi))}{\|\varpi\|^2} &\xrightarrow[T_1 \to \infty]{} \int_\R W_{\Sp} \phi = \widehat\phi(0) - \frac 12 \phi(0), & &\forall \delta < \frac{4+9t}{28t},
\end{align}
where $H = H_{T_1,T_2} := \sum_{\varpi} h_{T_1,T_2}(\mu(\varpi))/\|\varpi\|^2$. In particular, the type of symmetry for the family of spinor $L$-functions for $\GSp(4)$ is orthogonal, and the type of symmetry for the family of standard $L$-functions for $\GSp(4)$ is symplectic.
\end{thm}

\begin{rk}
    The parameters $c_{\bullet,T_1,T_2}$ are the average size of the analytic conductor of $L(s,\varpi,\bullet)$ for the family of forms $\varpi$ picked up by the test function $h_{T_1,T_2}$.
\end{rk}

\subsection{Structure of the paper} 

We recall the necessary definitions and properties of the automorphic forms and $L$-functions for $\GSp(4)$ in \Cref{sec:setting}. An explicit version of the Kuznetsov trace formula is proven in \Cref{sec:ktf}. \Cref{sec:spectral-localisation,sec:arithmetic-transform} constitute the technical core of the paper. In \Cref{sec:spectral-localisation} we write explicitly the spectral transform in terms of an explicit integral transform of Whittaker functions; the analytic behaviour of the spectral transform (\Cref{thm:sl}) then follows from a precise analysis of the residues obtained by moving contours, which happen to mostly cancel themselves except for a few ones that contribute with an explicit size in a specific region. In \Cref{sec:arithmetic-transform}, explicit bounds for the arithmetic transforms are obtained; the explicit oscillations thus displayed lead to the localisation statement (\Cref{prop:Jw_trivial}), and  finer bounds (\Cref{prop:Jw_ls}) are obtained by controlling the volume of the corresponding supports using stationary phase. \Cref{sec:consequences} is then dedicated to prove the various results as consequences of these fine controls on both sides of the Kuznetsov formula.
\section{Automorphic forms on GSp(4)}
\label{sec:setting}

\subsection{The symplectic group of genus two}
\label{subsec:gsp4}

For any commutative ring $R$, the symplectic similitude group of genus $2$ is given by
\[
    \GSp(4,R) = \{g \in M_4(R) \ : \ \exists \lambda = \lambda(g) \in R^\times, {}^\top{}gJg = \lambda J \}, \quad \text{where} \quad J =
    \begin{pmatrix}
        & I_2 \\
        -I_2 &
    \end{pmatrix}.
\]
Throughout, we shall denote $G = \GSp(4)$. The symplectic group is a subgroup of $G(R)$, defined as
\[
    \mathrm{Sp}(4,R) = \cbc{g \in G(R)}{\lambda(g) = 1}.
\]
We shall write $\Gamma = \Sp(4,\Z)$ for the standard arithmetic subgroup of $\GSp(4,\R)$. Define $U$ to be the standard unipotent subgroup
\[
    U(R) = \cbc{x = \begin{pmatrix} 1&x_1&x_2&x_3\\&1&x_4&x_5\\&&1\\&&-x_1&1\end{pmatrix}}{x_i \in R, \; x_3 = x_4 + x_1x_5}.
\]

Let $W$ be the Weyl group of $\GSp(4)$, generated by the matrices
\begin{equation}
    \alpha = \begin{pmatrix}
        & 1 & & \\
        -1 & & & \\
         & & & 1  \\
       & & 1 & 
    \end{pmatrix}
    \quad \text{and} \quad 
    \beta = 
    \begin{pmatrix}
        1 & & & \\
        & & & 1 \\
        & & 1 & \\
        & -1 & & 
    \end{pmatrix}.
\end{equation}

We denote the long Weyl element by $w_0 = \alpha\beta\alpha\beta$. For $M=(m_1,m_2)\in\R^2$ we define an additive character $\psi_M: U(\R) \to S^1$ of $U(\R)$ by
\begin{equation}\label{eq:theta_char} 
    \psi_M(x) = e(m_1x_1 + m_2x_5),
\end{equation}
where $e(x) := e^{2\pi ix}$ is the standard additive character. When $M \in \Z^2$, this defines a character of $U(\Z)\bs U(\R)$. We say $\psi_M$ is \emph{non-degenerate} if $m_1m_2\ne 0$. If $M=(1,1)$, we often drop it from the notation of the character.

Let $T$ be the diagonal torus of  $\Sp(4)$. Embed $y = (y_1, y_2) \in \R^2$ into $T(\R)$ via the map $\iota(y) := (y_1y_2^{1/2}, y_2^{1/2}, 1/y_1y_2^{1/2}, 1/y_2^{1/2})$, and denote the image of $\R_+^2$ into $T(\R)$ by $T(\R_+)$. The standard minimal parabolic subgroup of $\Sp(4)$ is $P_0=TU$. Denote by $g = x\tilde{y}k$ the Iwasawa decomposition of $g \in \Sp(4,\R)$, where $x \in U(\R)$, $\tilde{y}\in T(\R_+)$ and $k \in K = \SO(4,\R) \cap \Sp(4, \R)$ is the maximal compact subgroup of $\Sp(4, \R)$.  Let $\y(\tilde{y}) = \iota^{-1}(\tilde{y}) \in \R^2$ the Iwasawa $y$-coordinate of $g$. 

For $\alpha=(\alpha_1,\alpha_2)\in\C^2$ and $y \in \R_+^2$, we write $y^\alpha := y_1^{\alpha_1} y_2^{\alpha_2}$. Let $\eta := (2,\tfrac 32)$. We define the measures
\begin{align*}
dx &= dx_1 dx_2 dx_4 dx_5, & d^\times y = y^{-2\eta} \frac{dy_1}{y_1} \frac{dy_2}{y_2}
\end{align*}
on $U(\R)$ and $\R_+^2$ respectively. By an abuse of notation we also denote by $d^\times y$ the pushforward of $d^\times y$ to $T(\R_+)$ by $\iota$. Then $dx$ is the Haar measure on $U(\R)$, and $dxd^\times y$ is a left $\Sp(4,\R)$-invariant measure on $\Sp(4,\R)/K$.

For our purposes, we define another embedding of $\R_+^2$ into $T(\R_+)$ by
\[
c = (c_1, c_2) \mapsto c^\star  := \diag(1/c_1, c_1/c_2, c_1, c_2/c_1).
\]
It is straightforward to verify that $\y(c^\star )^\eta = (c_1c_2)^{-1}$.

\subsection{Automorphic forms and spectral parameters}
\label{subsec:gsp4-automorphic-forms}

We introduce automorphic representations and forms in the setting of the symplectic group $\GSp(4, \R)$, and give the precise normalisation we use for their spectral parameters.

\subsubsection{Automorphic forms, representations and spectral parameters}

For $\mu = (\mu_1, \mu_2)\in\C^2$, the Jacquet--Whittaker function $W_\mu:\R_+^2 \to \C$ over $\GSp(4)$ is defined by (the analytic continuation of) the integral
\[
W_\mu(y) = \int_{U(\R)} I_\mu(w_0u\iota(y)) \ol{\psi(u)} du,
\]
where $\psi = \psi_{(1,1)}$ is a nontrivial character of $U$, and the function $I_\mu:  \Sp(4,\R) \to \C$ is given by
\[
I_\mu(u\iota(y)k) = y_1^{2+\mu_1+\mu_2} y_2^{3/2+\mu_1}.
\]

Given an automorphic form $\varpi$ on $\GSp(4)$, the Jacquet period of $\varpi$ is defined to be the function
\begin{equation}
\label{eq:unipotent-period}
\mathcal{W}_\varpi(g) := \int_{U(\mathbf{Z}) \backslash U(\mathbf{R})} \varpi(ug) \ol{\psi(u)} du.
\end{equation}
We say that $\varpi$ is \textit{generic} if $\mathcal W_\varpi$ is nonzero. In that case, the Jacquet period $\mathcal W_\varpi$ satisfies 
\begin{equation*}
\mathcal{W}_\varpi(u \iota(y) k) = c W_\mu(y) \psi(u)
\end{equation*}
for certain parameters $\mu_1(\varpi),\mu_2(\varpi) \in \C$, for all $g = u \iota(y) k \in  \Sp(4,\R) = UT(\R_+)K$ and some constant~$c\in~\C^\times$. To each automorphic form $\varpi$ we therefore associate by this procedure an (archimedean) spectral parameter $\mu(\varpi) = \{\mu_1(\varpi),\mu_2(\varpi)\} \in \C^2$. 
Throughout the article, we consider only the spherical cuspidal spectrum, denoted $\mathcal A(\GSp(4))$, that is the space of $K$-invariant cuspidal automorphic forms $\varpi$. These are necessarily generic by Proposition \ref{prop:spherical-implies-generic} below.

The Weyl action leaves the set~$\{\pm\mu_1,\pm\mu_2\}$ invariant, so we may assume
\begin{equation}
    0 \le \Im(\mu_1) \le \Im(\mu_2).
\end{equation}
We always assume unitaricity:
\begin{equation}
    \{\pm\mu_1, \pm\mu_2\} = \{\pm\ol{\mu_1},\pm\ol{\mu_2}\}.
\end{equation}
Through the global generic lift \cite{gt} and  known bounds towards Ramanujan conjecture for~$\GL(4)$ automorphic forms~\cite{bb}, we have
\begin{equation}
\label{eq:ma}
    \vb{\Re\mu_1}, \vb{\Re\mu_2} \le \tfrac 9{22}.
\end{equation}
This implies that the spectral parameter $\mu(\varpi)$ belongs to one of the following two cases: 
\begin{align}\label{eq:mat} 
&\mu_1,\mu_2 \in i\R, & &\text{(tempered case)}
\intertext{or}
&\mu_2 = -\ol\mu_1, \quad 0 < \vb{\Re \mu_1} \le \frac 9{22}. & &\text{(non-tempered case)} \label{eq:man} 
\end{align}

\subsubsection{Whittaker functions} 
\label{subsec:whittaker-functions}

For $\mu \in \C^2$, the (completed) Whittaker function $W_\mu^\star(y)$ is defined from the Jacquet--Whittaker function by the formula
\[
W_\mu^\star(y) = W_\mu(y) \pi^{-\mu_1-2\mu_2-2} \Gamma\rb{\tfrac{1+2\mu_1}2}\Gamma\rb{\tfrac{1+2\mu_2}2}\Gamma\rb{\tfrac{1+\mu_1+\mu_2}2} \Gamma\rb{\tfrac{1-\mu_1+\mu_2}2}.
\]
The completed Whittaker function is an entire function in $\mu$, and is invariant under Weyl action. By Ishii--Moriyama \cite{im}, we may rewrite $W_\mu^\star$ as a four-fold Mellin-Barnes integral:
\begin{multline}\label{eq:IshiiMoriyama}
W_\mu^\star(y) = C y^\eta \int_{(\sigma_1)} \frac{ds_1}{2\pi i} \int_{(\sigma_2)} \frac{ds_2}{2\pi i} \int_{(\sigma_3)} \frac{ds_3}{2\pi i} \int_{(\sigma_4)} \frac{ds_4}{2\pi i} (\pi y_1)^{-s_3} (\pi y_2)^{-s_4} \Gamma\rb{\frac{s_1+\mu_1}2} \Gamma\rb{\frac{s_1-\mu_1}2}\\
\times \Gamma\rb{\frac{s_2+\mu_2}2} \Gamma\rb{\frac{s_2-\mu_2}2} \Gamma\rb{\frac{s_3}2} \Gamma\rb{\frac{s_3-s_1-s_2}2} \Gamma\rb{\frac{s_4-s_1}2} \Gamma\rb{\frac{s_4-s_2}2},
\end{multline}
with $\sigma_1 > |\Re\mu_1|$, $\sigma_2 > |\Re\mu_2|$, $\sigma_3 > \sigma_1+\sigma_2$, and $\sigma_4 > \max\{\sigma_1,\sigma_2\}$, where $C$ is an unspecified absolute constant. This integral representation and Stirling's formula will be our main tool in the proof of \Cref{thm:sl} and are crucial for understanding the spectral transform.

\medskip

In this article, we shall consider the following normalisation of the Whittaker function:
\begin{align*}
\widetilde W_\mu(y) &:= W_\mu(y) \pi^{-\mu_1-2\mu_2-2} \frac{\Gamma(\frac{1+2\mu_1}2) \Gamma(\frac{1+2\mu_2}2) \Gamma(\frac{1+\mu_1+\mu_2}2) \Gamma(\frac{1-\mu_1+\mu_2}2)}{\vb{\Gamma(\frac{1+i\Im\mu_1+i\Im\mu_2}2) \Gamma(\frac{1+2i\Im\mu_1}2) \Gamma(\frac{1+2i\Im\mu_2}2) \Gamma(\frac{1-i\Im\mu_1+i\Im\mu_2}2)}}\\
&= \frac{W_\mu^\star(y)}{\vb{\Gamma(\frac{1+2i\Im\mu_1}2) \Gamma(\frac{1+2i\Im\mu_2}2) \Gamma(\frac{1+i\Im\mu_1+i\Im\mu_2}2) \Gamma(\frac{1-i\Im\mu_1+i\Im\mu_2}2)}}.
\end{align*}
We note that in the strip $|\Re\mu_1|, |\Re\mu_2| \le 9/22$, the function $\widetilde W_\mu$ differs from the Jacquet--Whittaker function $W_\mu$ by a bounded factor, and is everywhere continuous (but not holomorphic). For convenience, we shall write
\[
c_\mu := \pi^{-\mu_1-2\mu_2-2} \frac{\Gamma(\frac{1+2\mu_1}2) \Gamma(\frac{1+2\mu_2}2) \Gamma(\frac{1+\mu_1+\mu_2}2) \Gamma(\frac{1-\mu_1+\mu_2}2)}{\vb{\Gamma(\frac{1+i\Im\mu_1+i\Im\mu_2}2) \Gamma(\frac{1+2i\Im\mu_1}2) \Gamma(\frac{1+2i\Im\mu_2}2) \Gamma(\frac{1-i\Im\mu_1+i\Im\mu_2}2)}}
\]
for the normalisation factor, so that $\widetilde W_{\mu} = c_{\mu} W_\mu$.

\subsubsection{Spherical representations and genericity}\label{subsec:spherical-implies-generic}

In general, only the generic spectrum naturally appears in the setting of Kuznetsov trace formulas. Concretely, they can be phrased as a relative trace formula for the direct product $U\times U$ of unipotents, and involve unipotent periods as in \eqref{eq:unipotent-period}, hence erasing non-generic contributions. All the statements should therefore be understood as statistics on the \textit{generic} spectrum of $\GSp(4)$; this misses an a priori less trivial part as in the~$\GL(n)$ case, where every  automorphic representation in the discrete spectrum is either trivial or generic~\cite{cogdell}. 

However, in the specific setting of $\GSp(4)$, and when we consider spherical representations only, this is not a serious restriction as we have the following result.

\begin{prop}
\label{prop:spherical-implies-generic}
    The automorphic representations occurring in the discrete spectrum of the space $L^2(\GSp(4,\Q) \backslash \GSp(4,\A) / K)$ are either the trivial representation, or generic. 
\end{prop}

In other words, the spherical discrete spectrum is, except for the trivial representation, included in the generic spectrum. In particular, the spherical cuspidal spectrum is included in the generic spectrum. Therefore, all the statistics displayed in the various consequences of the Kuznetsov formula stated in Section \ref{subsec:consequences} are genuinely statistics over the whole cuspidal spectrum.

\begin{rk}
       We moreover know that spherical generic representations are principal series, so this proves also that spherical representations are automatically principal series. This strikingly parallels the $\GL(2)$ case, where holomorphic modular forms are also not spherical.
\end{rk}

\begin{proof}

Arthur \cite{arthur_classification} classified the automorphic representations of $\GSp(4)$ with trivial central character  --- more precisely, these correspond to automorphic representations of $\mathrm{PGSp}(4)$, which is isomorphic to $\mathrm{SO}(5)$, and Arthur classified representations of orthogonal groups.  The classification describes the discrete automorphic spectrum of $\GSp(4)$ by sorting them into Arthur packets (called A-packets), which encompass all the automorphic representations with well-described multiplicities. The classification is summarized by Schmidt \cite{schmidt_packet_2017}: there are six types, viz. $(\mathbf{F})$ --- characters; $(\mathbf{Q})$, $(\mathbf{P})$, $(\mathbf{B})$ --- called the CAP representations, that are non-tempered; and $(\mathbf{G})$, $(\mathbf{Y})$ --- which are conjecturally tempered. The corresponding spectral parameters and L-factors are described in~\cite{schmidt_packet_2017}.

Let $\pi$ be an automorphic spherical representation of $\Sp(4)$. By the Flath tensor product theorem \cite[Theorem 5.7.1]{gh},  $\pi$ factorises as a restricted tensor product $\otimes_v \pi_v $ over the places $v$ of $\Q$, where $\pi_v$ is a smooth admissible representation of $\Sp(4,\Q_v)$. Since $\pi$ is spherical, each $\pi_v$ is also spherical; indeed, $\pi$ has nonzero fixed vectors by the maximal compact subgroup $K = \prod_v K_v$, hence each $\pi_v$ has nonzero fixed vectors by $K_v$.
Global parameters give rise to local parameters by projection onto the components of the restricted tensor product, and hence to local A-packets (see \cite{schmidt_packet_2017} and \cite[Theorem 1.5.1]{arthur_book} to the precise construction  of the local A-packets). These local A-packets are explained in \cite{schmidt_packet_2017} and we use this explicit description to understand the spherical automorphic representation $\pi$ and rule out the possibility for them to be non-generic.

A representation of type $(\mathbf{F})$ is a character. The only character that is spherical at all places is the trivial representation.

A representation of type $(\mathbf{Q})$ or $(\mathbf{B})$ cannot be paramodular at every place \cite[Proposition~5.1]{schmidt_paramodular_2020}, hence $\pi$ cannot occur in packets of these types. A representation in a packet of type $(\mathbf{P})$ cannot be spherical at infinity by the explicit description of its possible archimedean parameters \cite[Diagram (22)]{schmidt_paramodular_2020}, none of which are spherical; hence $\pi$ cannot occur in packets of type $(\mathbf{P})$ either. Thus we conclude that the CAP representations cannot be spherical, and the only possibilities left are that $\pi$ belongs to packets of type $(\mathbf{G})$ or $(\mathbf{Y})$.

Each local packet coming from a global A-packet of type $(\mathbf{G})$ or $(\mathbf{Y})$ contains a unique generic representation \cite[Theorem 1.1(i)]{schmidt_packet_2017}. For the non-archimedean places, the local A-packet contains a unique paramodular representation, which coincides with the above-mentioned generic one~\cite[Theorem 1.1(ii)]{schmidt_packet_2017}. This implies that $\pi_v$ is generic for each non-archimedean place $v$. At the Archimedean place, we can use the explicit parametrisation of the representations through the archimedean Langlands classification for $\Sp(4,\R)$ --- more precisely it is possible to examine the purely local representation theory of $\Sp(4,\R)$ as obtained by Muić \cite{muic_intertwining_2009} and the corresponding L-parameters described therein --- to conclude that the L-packet of a spherical representation is a singleton\footnote{We are deeply indebted to Ralf Schmidt for his explanations and arguments on this question.}. The A-packets that are not already L-packets are described by Schmidt \cite{schmidt_paramodular_2020}:  there is a case where a spherical representation shares an A-packet with another representation. This is the type ``$(0,0)$, D'' listed in Table 1 and also Table 3 therein; but neither representation in the two-elements A-packet is generic. On the other hand, by \cite[Theorem 1.1]{schmidt_packet_2017}, the local A-packet of type $(\mathbf{G})$ or $(\mathbf{Y})$ of $\pi_\infty$ must contain a generic representation. Therefore, such spherical representations do not occur as local component of an automorphic representation. In conclusion, not only the L-packet but also the A-packet of $\pi_\infty$ is a singleton, and since it must contain a generic representation \cite[Theorem 1.1]{schmidt_packet_2017}, the representation $\pi_\infty$ itself is generic. 

Hence any spherical representation of A-type $(\mathbf{G})$ or $(\mathbf{Y})$ is generic at all places (also known as \emph{abstractly generic}). Since A-packets of type $(\mathbf{G})$ or $(\mathbf{Y})$ are generic packets, the representations they contain are nearly equivalent to a globally generic representation by definition. The Shahidi principle, proven by Jiang--Soudry \cite{jiang_multiplicity-one_2007}, says that if $\pi$ is an abstractly generic representation that is nearly equivalent to a globally generic representation~$\pi'$, then we have $\pi\simeq\pi'$, and in particular~$\pi$ is actually globally generic. (This can alternatively be shown using \cite[Proposition~8.3.2]{arthur_book}.) This finishes the case-by-case examination of the spherical automorphic representations of $\GSp(4)$, and finishes the proof.
\end{proof}

\subsubsection{Fourier coefficients and Hecke operators}\label{subsec:Hecke}

For $M = (m_1,m_2)\in\Z^2$, the $M$-th Fourier coefficient of $\varpi$ is given by
\begin{equation}\label{eq:Fourier_def}
\varpi_M(g) := \int_{U(\mathbf{Z})\backslash U(\mathbf{R})} \varpi(ug) \ol{\psi_M(u)} du = A_\varpi(M)M^{-\eta} \psi_M(x) \widetilde W_{\mu(\varpi)}(\iota(M) y),
\end{equation}
for some $A_\varpi(M) \in \C$. If $\varpi$ is generic, we say $\varpi$ is \emph{arithmetically normalised} if $A_\varpi(1,1)=1$. Throughout the article, we shall always assume that an automorphic form is aritmetically normalised.

From Man \cite{man}, there is a nice characterisation of the Fourier coefficients $A_\varpi(M)$, when $\varpi$ is in addition an eigenfunction of the Hecke algebra of $\GSp(4)$. To state the results, we recall some facts about Hecke operators on~$\GSp(4)$ from \cite{man}. 

Let $\mathcal M$ be a set of matrices in $\GSp(4,\Q)^+$ that is left and right invariant under $\Gamma = \Sp(4,\Z)$; it is a finite union $\bigcup_j \Gamma \gamma_j$ of left cosets. Then $\mathcal M$ defines the Hecke operator $T_\mathcal{M}$ on the space~$\mathcal A(\GSp(4))$ of automorphic forms by
\[
T_{\mathcal M}\varpi(g) = \sum_j \varpi(\gamma_j g).
\]
For a matrix $\gamma\in\GSp(4,\Q)^+$, we denote by $T_\gamma$ the Hecke operator $T_{\Gamma \gamma\Gamma}$. For $m\in\N$, we define the set of matrices
\[
S(m) := \cb{\gamma \in \GSp(4,\Z)^+ \ : \ \lambda(\gamma)=m}.
\]
The $m$-th standard Hecke operator is then given by $T(m) = T_{S(m)}$. From \cite{spence}, a complete set of coset representatives of $\Gamma\bs S(m)$ is given by the set of matrices
\begin{equation}\label{eq:Hecke_rep}
\mathcal H(m) := \cb{\bp A & m^{-1}BD\\ & D\ep\ :\ A = \bp a_1 & a\\ & a_2\ep,\ B = \bp b_1 & b_2\\ b_2 & b_3\ep,\ \begin{array}{l} a_1,a_2>0,\ 0\le a<a_2,\\ 0\le b_i<m, A {}^\top D = mI_2,\\ BD\equiv 0\pmod{m}\end{array}}.
\end{equation}
For $r\in\N_0$, two parameters $0\le a \le b \le \frac r2$, and $p$ prime, we define also the Hecke operator 
\[
T_{a,b}^{(r)}(p) := T_{ \diag(p^a,p^b,p^{r-a},p^{r-b}) }.
\]
Then we have a decomposition
\[
T(p^r) = \sum_{0\le a\le b\le r/2} T_{a,b}^{(r)}(p).
\]
The Hecke algebra $\mathcal H(\GSp(4))$ is generated by $T(p) = T_{0,0}^{(1)}(p)$ and $T_{0,1}^{(2)}(p)$ for primes $p$, along with the identity. In particular, we have the Hecke relation
\begin{equation}\label{eq:Hecke_p2}
T(p^2) = T(p)^2 - p T_{0,1}^{(2)}(p)-(p^3+p^2+p)\id.
\end{equation}

We say $\varpi \in \mathcal A(\GSp(4))$ is a \emph{Hecke--Maaß form} if $\varpi$ is an eigenfunction of the Hecke algebra~$\mathcal H(\GSp(4))$. For such $\varpi$, we write $\lambda(m,\varpi)$ (resp. $\lambda_{a,b}^{(r)}(p,\varpi)$) to denote the eigenvalue of $\varpi$ with respect to the operator $T(m)$ (resp. $T_{a,b}^{(r)}(p)$). It is well-known that the Hecke algebra~$\mathcal H(\GSp(4))$ is commutative; as a consequence the spherical cuspidal spectrum admits an orthogonal basis of Hecke--Maaß forms.

Let $\mathcal X = \{X_p\}_p$ and $\mathcal Y = \{Y_p\}_p$ be sequences of complex numbers indexed by primes $p$, and $M\in\N^2$. We define a function $\mathcal B_{\mathcal X,\mathcal Y}(M)$, multiplicative in $M$, by the generating function
\begin{align}
& \sum_{i,j\ge 0} \mathcal B_{\mathcal X,\mathcal Y}(p^i,p^j)u^iv^j\\
& = \frac{(1-u)(1+u)(1+uv^2)-X_p(1-u)uv}{(1-Y_pu+(X_p^2-Y_p-1)u^2-(X_p^2-Y_p-1)u^3+Y_pu^4-u^5)(1-X_pv+(Y_p+1)v^2-X_pv^3+v^4)}. \notag
\end{align}

Let $\varpi \in \mathcal A(\GSp(4))$ be an arithmetically normalised Hecke--Maaß form. Let $\mathcal X(\varpi) := \{X_{\varpi,p}\}_p$ and~$\mathcal Y(\varpi) := \{Y_{\varpi,p}\}_p$, where $X_{\varpi,p} := p^{-3/2}\lambda(p,\varpi)$ and $Y_{\varpi,p} := p^{-2}(\lambda_{0,1}^{(2)}(p,\varpi) + 1)$. Then the Fourier coefficients $A_\varpi(M)$ of $\varpi$ are given by, for all $M \in \N^2$, 
\begin{equation}\label{eq:Fourier_coefficient}
A_\varpi(M) = \mathcal B_{\mathcal X(\varpi), \mathcal Y(\varpi)}(M).
\end{equation}

\subsection{$L$-functions}
\label{subsec:$L$-functions}

We define in this section the $L$-functions attached to automorphic forms of $\GSp(4)$. Even though the present work is phrased in the setting of automorphic \textit{forms} on $\GSp(4,\R)$, the relation with automorphic \textit{representations} will be needed, in particular to define the associated $L$-functions.

\subsubsection{Theoretical definition}

An automorphic form as defined above generates an automorphic representation, by considering the $\GSp(4, \mathbf{A}_f) \times (\mathfrak{gsp}_4(\R),K_\infty)$-module it generates, see \cite{bump}.  There is a procedure to attach to every generic automorphic representation of $\GL(n)$ an $L$-function, which is fundamental in the theory of automorphic forms, see \cite{cogdell} or \cite{shahidi}. The Langlands functoriality conjectures postulate an arithmetic parametrization of the automorphic forms of a group~$G$, which naturally embeds in the $\GL(n)$ setting,  and this can be used to define the $L$-function for general groups.

More precisely, the local Langlands conjectures postulate a surjection with finite fibers from the admissible dual $\hat{G}_k$ of $G(k)$, for a certain local field $k$, i.e. the set of equivalence classes of irreducible admissible complex representations of $G(k)$, to a set of "parameters" which are representations $\phi_\pi : W_k \to {}^LG$, where~$W_k$ is the Weil-Deligne group of $k$ and ${}^LG$ is the Langlands dual group of $G$. A representation of the Langlands dual group of $G$, say~$r : {}^LG \to \GL(n, \C) = {}^L\GL(n)$,  therefore gives by composition a representation $r \circ \phi_\pi : W_k \to {}^L\GL(n)$, i.e. a Langlands parameter of $\GL(n)$.  Since $L$-functions are already defined on $\GL(n)$, one can pull back the definition from this latter setting and let 
\begin{equation}
    L(s, \pi, r) := L(s, r \circ \phi_\pi)
\end{equation}
where the right-hand side is already defined by the $\GL(n)$ theory of $L$-functions developed by Jacquet, Piatetski-Shapiro and Shalika \cite{jpss}. This defines the local $L$-functions, and for an irreducible automorphic representation of $\GSp(4, \mathbf{A})$ and a representation $r : {}^L\GSp(4) \to \GL(n, \C)$ of its L-group, we set 
\begin{equation}
    L(s, \pi, r) := \prod_v L(s, \pi_v, r_v),
\end{equation}
where each local L-factor is defined as in the above procedure, and $v$ runs through the places of~$\Q$. See \cite{cogdell} for more details and the relations with the Langlands local and global functoriality conjectures.

The $L$-functions we will consider are those arising via this procedure when choosing the representations $\mathrm{Spin} : {}^L\GSp(4) = \GSp(4, \C) \to \GL(4, \C)$ and $\mathrm{Std} : \GSp(4, \C) \to \GL(5, \C)$; we describe them explicitly in this section. The $L$-functions constructed by these means in the case of the $\Spin$ representation can be obtained by a more classical approach, implemented by Novodvorsky~\cite{novodvorsky}, by using a family of zeta integrals  built from the Whittaker model of the underlying representation, mimicking the classical construction by Gelbart and Jacquet for $\GL(n)$, see \cite{rs}. The detailed construction and properties of the spinor $L$-function are described in \cite{takloo}.

\subsubsection{Explicit description of $L$-functions for $\GSp(4)$}

Each Hecke--Maaß cusp form $\varpi \in \mathcal A(\GSp(4))$ gives rise to an irreducible automorphic representation $\pi$, so we may define $L(s,\varpi,r)$ as the $L$-function $L(s,\pi,r)$ attached to the associated automorphic representation $\pi$. 

Let $\alpha_p = \alpha_{\varpi,p}$ and $\beta_p = \beta_{\varpi,p}$ be the local Satake parameters for $\varpi$. The local $L$-factors for the spinor and standard representations are given by
\begin{align}
L_p(s,\varpi,\Spin) &= \prod_{i=1}^4 \rb{1-u_{\varpi,\Spin,p,i}p^{-s}}^{-1}, & \{u_{\varpi,\Spin,p,i}\}_{1\le i \le 4} &= \{\alpha_p, \alpha_p^{-1}, \beta_p, \beta_p^{-1}\},\label{eq:up_spin}\\
L_p(s,\varpi,\Std) &= \prod_{i=1}^5 \rb{1-u_{\varpi,\Std,p,i}p^{-s}}^{-1}, & \{u_{\varpi,\Std,p,i}\}_{1\le i \le 5} &= \{\alpha_p\beta_p, \alpha_p\beta_p^{-1}, \alpha_p^{-1}\beta_p, \alpha_p^{-1}\beta_p^{-1},1\}. \label{eq:up_std}
\end{align}
If we write
\[
X_p = X_{\varpi,p} = \alpha_p + \alpha_p^{-1} + \beta_p + \beta_p^{-1}, \quad \text{ and } \quad Y_p = Y_{\varpi,p} = (\alpha_p + \alpha_p^{-1})(\beta_p+\beta_p^{-1}) + 1, 
\]
then we have
\begin{align*}
L_p(s,\varpi,\Spin) &= \rb{1-X_pp^{-s}+(Y_p+1)p^{-2s}-X_pp^{-3s}+p^{-4s}}^{-1},\\
L_p(s,\varpi,\Std) &= \rb{1-Y_pp^{-s}+(X_p^2-Y_p-1)p^{-2s}-(X_p^2-Y_p-1)p^{-3s}+Y_pp^{-4s}-p^{-5s}}^{-1}.
\end{align*}

We have the following relations between the Satake parameters and Fourier coefficients:
\[
X_{\varpi,p} = A_\varpi(1,p) \quad \text{ and } \quad Y_{\varpi,p} = A_\varpi(p,1).
\]
Since the Fourier coefficients of $\varpi$ satisfy the relation \eqref{eq:Fourier_coefficient}, we obtain expressions of $L$-functions in terms of Fourier coefficients as Dirichlet series:
\begin{equation}\label{eq:Fourier_L}
\sum_{n\ge 1} A(n,1) n^{-s} = \frac{L(s,\varpi,\Std)}{\zeta(2s)}, \quad \text{ and } \quad \sum_{n\ge 1} A(1,n) n^{-s} = L(s,\varpi,\Spin).
\end{equation}

We can define the archimedean L-factors from the archimedean Langlands parameters of $\pi_\infty$, as in \cite{ishii}; under the functoriality conjectures these should correspond to the archimidean factors of the global automorphic $L$-functions $L(s,\pi)$. Following \cite{ishii}, for a Hecke--Maaß cusp form $\varpi$ with associated representation $\pi$ and spectral parameters $(\mu_1, \mu_2)\in \C^2$, we set
\begin{align}
    L_\infty(s,\varpi,\Spin) &= \Gamma_\R(s+\mu_1) \Gamma_\R(s-\mu_1) \Gamma_\R(s+\mu_2) \Gamma_\R(s-\mu_2),\\
    L_\infty(s,\varpi,\Std) &= \Gamma_\R(s+1) \Gamma_\R(s+\mu_1+\mu_2) \Gamma_\R(s+\mu_1-\mu_2) \Gamma_\R(s-\mu_1+\mu_2) \Gamma_\R(s-\mu_1-\mu_2),
\end{align}
where $\Gamma_\R(s) := \pi^{-s/2}\Gamma(s/2)$. Following \cite{ik} for the definition of the analytic conductor, we find that the analytic conductor of $L(s,\varpi,\Spin)$ has size $\asymp (1+|\mu_1|)^2(1+|\mu_2|)^2$, and the analytic conductor of $L(s,\varpi,\Std)$ has size $\asymp (1+|\mu_1+\mu_2|)^2(1+|\mu_1-\mu_2|)^2$. 

\subsubsection{Bounds of $L$-values}
We recall known bounds for $L$-functions appearing on the spectral side of the Kuznetsov trace formula \eqref{eq:ktf}, which are needed for our applications.

\begin{lem}\label{lem:adlv}
Let $\varpi$ be an arithmetically normalised Hecke--Maaß cusp form on $\GSp(4)$ with spectral parameters $\mu(\varpi) = (\mu_1,\mu_2)$. Then, for any $\varepsilon>0$ we have
\[
\|\varpi\|^2 \ll (1+|\mu_1|+|\mu_2|)^\varepsilon.
\]
\end{lem}
\begin{proof}
By \cite{ci}, we have $\|\varpi\|^2 \asymp L(1, \varpi,\Ad)$. By \cite{gt}, the automorphic representation $\pi$ attached to $\varpi$ can be lifted to an automorphic representation $\Pi$ on $\GL(4)$ with the same $L$-parameters, and we have an equality of $L$-functions:
\[
L(s,\varpi,\Ad) = L(s,\Pi,\Sym^2).
\]
By \cite{bg}, $L(s,\Pi,\Sym^2)$ has analytic continuation and satisfies a functional equation. The bound then follows from the standard analytic arguments in Li \cite{li}.
\end{proof}

Now we consider the $L$-functions appearing in the denominators of the contribution of the continuous spectrum. The result
\begin{equation}\label{eq:zlb} 
    |\zeta(1+it)| \gg (1+|t|)^{-\varepsilon}
\end{equation}
is classical, and follows from the standard zero-free region of $\zeta(s)$. In \cite{ghl,hl,hr}, it is established that for a $\GL(2)$ cuspidal representation $\phi$, the $L$-functions $L(s,\phi)$ and $L(s,\Sym^2\phi)$ also admit a standard zero-free region (i.e. no Siegel zeros). This implies
\begin{equation}\label{eq:lflb} 
L(1+it,\phi), \ L(1+it,\Sym^2\phi) \gg (1+|\nu|+|t|)^{-\varepsilon},
\end{equation}
where $\nu$ denotes the spectral parameter of $\phi$.

\section{Explicit Kuznetsov trace formula}\label{sec:ktf}

In this section, we establish the Kuznetsov trace formula for $\GSp(4)$ in a completely explicit form.

\begin{prop}[Kuznetsov trace formula for $\GSp(4)$]
\label{thm:ktf}
Let $F:\R_+^2 \to \C$ be a test function with compact support, and $M = (m_1,m_2),N = (n_1,n_2)\in\N^2$. Then we have
\begin{equation}\label{eq:ktf} 
\mathcal S_{\operatorname{cusp}} + \mathcal S_{0} + \mathcal S_{S} + \mathcal S_{K} = \mathcal K_{\id} + \mathcal K_{\alpha\beta\alpha} + \mathcal K_{\beta\alpha\beta} + \mathcal K_{w_0}.
\end{equation}
The terms $\mathcal{S}_{\operatorname{cusp}}, \mathcal{S}_{0}, \mathcal{S}_{S}, \mathcal{S}_{K}$ on the spectral side are given as follows:
\begin{align*}
\mathcal S_{\operatorname{cusp}} &= \sum_\varpi \frac{\ol{A_\varpi(M)} A_\varpi(N)}{\|\varpi\|^2} \vb{\langle \widetilde W_{\mu_1,\mu_2}, F\rangle}^2,\\
\mathcal S_{0} &= C_{0} \int_{(0)}\int_{(0)} \frac{|c_{s_2,s_1-s_2}|^{-2}\ol{B_{0,s}(M)} B_{0,s}(N)}{\vb{\zeta(1+2s_1-2s_2)\zeta(1-s_1+2s_2)\zeta(1+2s_2)\zeta(1+s_1)}^2} \vb{\langle \widetilde W_{s_2,s_1-s_2}, F\rangle}^2 ds_1 ds_2,\\
\mathcal S_{S} &= C_{S} \sum_\phi \int_{(0)}  \frac{|c_{s,\nu_\phi}|^{-2}\ol{B_{S,s,\phi}(M)} B_{S,s,\phi}(N)}{\vb{L(1+s,\phi)\zeta(1+2s)}^2 L(1,\phi,\Ad)} \vb{\langle \widetilde W_{s, \nu_\phi}, F\rangle}^2 ds,\\
\mathcal S_{K} &= C_{K} \sum_\phi \int_{(0)}  \frac{|c_{s/2+\nu_\phi,s/2-\nu_\phi}|^{-2}\ol{B_{K,s,\phi}(M)} B_{K,s,\phi}(N)}{\vb{L(1+s,\phi,\Sym^2)}^2 L(\phi, 1, \Ad)} \vb{\langle \widetilde W_{s/2+\nu_\phi,s/2-\nu_\phi}, F\rangle}^2 ds.
\end{align*}
Here, the sum over $\varpi$ runs over a basis of spherical Hecke--Maaß cusp forms of $\GSp_4$, the sum over $\phi$ runs over a basis of Hecke--Maaß cusp forms of $\GL_2$, and $\nu_\phi$ denotes the spectral parameter of $\phi$. The expressions $B_{0,s}, B_{S,s,\phi}, B_{K,s,\phi}$ are given in \eqref{eq:B_expression}, and $C_{0}, C_{S}, C_{K}$ are unspecified absolute constants. 

\smallskip

On the arithmetic side, the terms $\mathcal K_{\id}, \mathcal K_{\alpha\beta\alpha}, \mathcal K_{\beta\alpha\beta}, \mathcal K_{w_0}$ are given as follows:
\begin{align*}
\mathcal K_{\id} &= \delta_{M=N} \|F\|^2,\\
\mathcal K_{\alpha\beta\alpha} &= \sum_{\substack{c_2 \mid c_1^2\\ m_2c_1^2=n_2c_2^2}} \frac{\Kl_{\alpha\beta\alpha}(c,M,N)}{c_1c_2} \mathcal J_{\alpha\beta\alpha, F} \rb{\sqrt{\frac{m_1m_2n_1}{c_2}}},\\
\mathcal K_{\beta\alpha\beta} &= \sum_{\substack{c_1 \mid c_2\\ m_1c_2 = n_1c_1^2}} \frac{\Kl_{\beta\alpha\beta}(c,M,N)}{c_1c_2} \mathcal J_{\beta\alpha\beta, F} \rb{\frac{m_1\sqrt{m_2n_2}}{c_1}},\\
\mathcal K_{w_0} &= \sum_{c_1,c_2} \frac{\Kl_{w_0}(c,M,N)}{c_1c_2} \mathcal J_{w_0, F} \rb{\frac{\sqrt{m_1n_1c_2}}{c_1}, \frac{\sqrt{m_2n_2}c_1}{c_2}},
\end{align*}
where the sums run over $c = (c_1, c_2) \in \N^2$, the integral transforms $\mathcal J_{w,F}$ are given in \Cref{subsec:kuznetsov_integral} and the Kloosterman sums $\Kl_{w}(c,M,N)$ are defined in \eqref{eq:klw}.
\end{prop}

\begin{rk}
    In the Kuznetsov formula \eqref{eq:ktf}, the terms $\mathcal{S}_0$, $\mathcal{S}_S$, and $\mathcal{S}_K$ correspond to the contribution of the continuous spectrum, which is induced from the three parabolic subgroups of $\Sp(4)$, namely the minimal parabolic $P_0$, the Siegel parabolic $P_S$ and the Klingen parabolic $P_K$. The quantities $B_{0,s}$, $B_{S,s,\phi}$, and $B_{K,s,\phi}$ appearing there are essentially the Fourier coefficients of the associated lower-rank forms.
\end{rk}

\subsection{Strategy}

Here we give a quick account of the genesis of the Kuznetsov formula as in \cite{man}. We define  the Poincaré series and obtain the arithmetic side of the Kuznetsov trace formula by expanding the inner product of two Poincaré series by unfolding. 

Starting from a test function $F:\R_+^2\to \C$ with compact support, we construct a right $K$-invariant function $\mathcal F:\Sp(4,\R)\to\C$ by
\begin{equation}
\mathcal F(xy) = \psi(x) F(\y(y)),
\end{equation}
where $\psi = \psi_{1,1}$ is the standard additive character on $U$. For $N\in\N^2$, the Poincaré series associated to $F$ and to the character $\psi_N$ is given by appealing to Iwasawa decomposition and setting
\[
P_N(xy) = \sum_{\gamma \in P_0 \cap \Gamma \backslash \Gamma} \mathcal F(\iota(N)\gamma xy),
\]
for $x \in U(\R)$ and $y \in T(\R_+)$, where $\Gamma = \Sp(4,\Z)$ and $P_0 = TU$ is the standard minimal parabolic.
Note that $\mathcal F(\iota(N)xy) = \psi_N(x) F(N\y(y))$. Similarly to the case of $\GL(2)$, the Kuznetsov trace formula follows from a double computation of the inner product $\langle P_N, P_M \rangle$ between Poincaré series.

\subsubsection{The pre-spectral side}
\label{subsubsec:ktf-spectral}

The spectral expansion of Poincaré series and Parseval identity gives
\begin{equation}
\label{eq:ss}
    \frac{\langle P_M, P_N \rangle}{(MN)^\eta} = \int_{\mathcal{A}(\GSp(4))} \frac{\overline{A_\varpi(M)} A_\varpi(N)}{\|\varpi\|^2} \vb{\pb{\widetilde W_\mu, F}}^2 d\varpi,
\end{equation}
where $\mathcal{A}(\GSp(4))$ denotes the generic spectrum of $\GSp(4)$, the integral over $\varpi \in \mathcal{A}(\GSp(4))$ is an integration over the complete generic spectrum, and $d\varpi$ is the Plancherel measure. The generic spectrum $\mathcal{A}(\GSp(4))$ consists of cuspidal and continuous. The continuous spectrum is spanned by Eisenstein series associated to the minimal, Siegel, and Klingen parabolic subgroups respectively. These Einsenstein series and the corresponding contributions will be described in Section \ref{subsec:33}.

\subsubsection{The pre-arithmetic side}
\label{subsubsec:ktf-arithmetic}

The Fourier coefficients of the Poincaré series are given by
\begin{align*}
\int_{U(\Z) \backslash U(\R)} P_M(xy) \overline{\psi_N(x)}dx  = \sum_{\gamma \in P_0 \cap \Gamma \backslash \Gamma} \int_{U(\Z) \backslash U(\R)} \mathcal{F}(\iota(M)\gamma xy) \overline{\psi_N(x)}dx.
\end{align*}
Recall from Bruhat decomposition that $\Sp(4,\R) = \coprod_{w \in W} G_w$
where $G_w = UwTU$. Therefore, we have the decomposition $\Gamma = \coprod_{w \in W} \Gamma \cap G_w$ and we obtain
\begin{align*}
& \int_{U(\Z) \backslash U(\R)} P_M(xy) \overline{\psi_N(x)}dx  = \sum_{w \in W} \sum_{\gamma \in P\cap \Gamma \backslash \Gamma \cap G_w} \int_{U(\Z) \backslash U(\R)} \mathcal{F}(\iota(M)\gamma  xy) \overline{\psi_N(x)}dx \\
& \qquad = \sum_{w \in W} \sum_{\gamma \in P\cap \Gamma \backslash \Gamma \cap G_w / \Gamma_w} \sum_{\ell \in \Gamma_w} \int_{U(\Z) \backslash U(\R)} \mathcal{F}(\iota(M)\gamma \ell xy) \overline{\psi_N(x)}dx \\
& \qquad = \sum_{w \in W} \sum_{c \in \N^2} \mathrm{Kl}_w(c, M, N) \int_{U_w(\R) } \mathcal{F}(\iota(M)\gamma c^\star  w xy) \overline{\psi_N(x)}dx,
\end{align*}
letting $\Gamma_w = U(\Z) \cap w^{-1}  U(\Z) {}^\top w$ and $U_w = U \cap w^{-1} U {}^\top w$. We define the Kloosterman sum as in~\cite{man2} by setting
\begin{equation}
\label{eq:klw}
\mathrm{Kl}_w(c, M, N) = \sum_{xwc^\star  x' \in U(\Z) \backslash G_w(\Q)  \cap \Gamma / U(\Z)} \psi_M(x) \psi_N(x').
\end{equation}

To obtain the arithmetic side, we unfold the Petersson inner product of two Poincaré series $P_N$ and $P_M$, replacing $P_N$ by its very definition and using the above calculation, to get
\begin{align*}
\langle P_M, P_N\rangle & = \int_{\Gamma\backslash\Sp(4,\R)/K} P_M(xy) \overline{P_N(xy)} dx d^\times y\\
&= \int_{T(\R_+)} \int_{U(\Z)\backslash U(\R)} P_M(xy) \ol{\psi_N(x) F(N\y(y))}  dx d^\times y\\
&= \sum_{w\in W} \sum_{c\in\N^2} \Kl_w(c,M,N) \int_{T(\R_+)} \int_{U_w(\R)} \mathcal F(\iota(M) c^\star wxy) \ol{F(N\y(y))} \psi_{-N}(x) dx d^\times y.
\end{align*}
From considerations on the spectral side, it is more convenient to divide both sides by $M^\eta N^\eta$, and write the pre-trace formula thus obtained, as derived in \cite{man}
\begin{align}
\frac{\langle P_M, P_N\rangle}{M^\eta N^\eta} &= \sum_{w\in W} \sum_{c\in\N^2} \frac{\Kl_w(c,M,N)}{M^\eta N^\eta} \int_{T(\R_+)} \int_{U_w(\R)} \mathcal F(\iota(M) c^\star wxy) \ol{F(N\y(y))} \psi_{-N}(x) dx d^\times y. \label{eq:as} 
\end{align}

To write $\mathcal F(\iota(N) c^\star wxy)$ in terms of $F$, we need to know the Iwasawa decomposition of the matrix~$c^\star wxy$. Indeed, writing $c^\star wxy = nak$ where $n \in U(\R)$, $a \in T(\R_+)$ and $k \in K$, we have by definition~$\mathcal{F}(nak) = \psi(n) F(\y(a))$. The right-hand side of \eqref{eq:as} is a pre-arithmetic part, and the integrals appearing there will be explained in Section \ref{subsec:arithmetic-transform} for each $w \in W$.

\subsubsection{Compatibility and bounds of Kloosterman sums}

The Kloosterman sum $\Kl_w(c,M,N)$ is subject to compatibility conditions depending on $w$.  Precisely, we require that the summation in \eqref{eq:klw} is well-defined (i.e. independent of the Bruhat decomposition); otherwise we let $\Kl_w(c,M,N) = 0$. We recall these compatibility conditions from \cite{man2}:
\begin{equation}\label{eq:Kloosterman_compatibility}
\begin{array}{c|c}
w & \text{conditions}\\\hline
\id & c=(1,1),\  M=N\\
\alpha & c_2=1,\  m_2=n_2=0\\
\beta & c_1=1,\  m_1=n_1=0\\
\alpha\beta & c_2\mid c_1,\ m_2=n_1=0
\end{array}
\quad
\begin{array}{c|c}
w & \text{conditions}\\\hline
\beta\alpha & c_1^2\mid c_2,\ m_1=n_2=0\\
\alpha\beta\alpha & c_2\mid c_1^2,\ m_2c_1^2=n_2c_2^2\\
\beta\alpha\beta & c_1\mid c_2,\ m_1c_2=n_1c_1^2\\
w_0 & \varnothing
\end{array}
\end{equation}
In particular, if $\psi_M, \psi_N$ are non-degenerate characters, then $\Kl_w(c,M,N)$ does not vanish only when $w=\id,\alpha\beta\alpha,\beta\alpha\beta,w_0$. We also recall from \cite{man2} some non-trivial bounds for $\Kl_w(c,M,N)$. For any $\varepsilon > 0$, we have
\begin{align}
\vb{\Kl_{\alpha\beta\alpha}(c,M,N)} &\ll_\varepsilon (m_1,n_1,c_1)(m_2,c_2)(c_1,c_2)(c_1c_2)^{1/3+\varepsilon},\label{eq:Kaba_bound}\\
\vb{\Kl_{\beta\alpha\beta}(c,M,N)} &\ll_\varepsilon (m_1,c_1)(m_2,n_2,c_2)(c_1^2,c_2)c_1^{-1/2+\varepsilon}c_2^{1/2+\varepsilon},\label{eq:Kbab_bound}\\
\vb{\Kl_{w_0}(c,M,N)} &\ll_\varepsilon (m_1m_2,n_1n_2,c_1c_2)^{1/2}(c_1,c_2)^{1/2} c_1^{1/2+\varepsilon} c_2^{3/4+\varepsilon}.\label{eq:Kw0_bound}
\end{align}

\subsection{The arithmetic transform}
\label{subsec:arithmetic-transform}

We explain here the integral transforms arising in \eqref{eq:as} by giving an explicit parametrization.

\subsubsection{Explicit parametrization of Fourier integrals}
\label{subsec:explicit-integral-transform}

Using Iwasawa decomposition, we are able to write down the Fourier integrals arising in the arithmetic expansion \eqref{eq:as}, i.e.
\begin{equation}\label{eq:UwF}
\mathcal U_{w,F} = \mathcal U_{w,F}(c,M,N,y) = \int_{U_w(\R)} \mathcal F(\iota(M) c^\star wxy) \psi_{-N}(x) dx, 
\end{equation}
in an explicit way. The computation is completely straightforward, so we just summarise the results in the following proposition.

\begin{prop}\label{prop:Fourier_integral}
For $w=\id,\alpha\beta\alpha,\beta\alpha\beta,w_0$, and $c,M,N\in\N^2$ satisfying the compatibility conditions \eqref{eq:Kloosterman_compatibility}, the integral transforms $\mathcal U_{w,F}$ are given as follows. For $w=\id$, we have
\begin{equation*}
    \mathcal U_{\id,F} = F(m_1y_1,m_2y_2).
\end{equation*}
For $w=\alpha\beta\alpha$ we have
\begin{multline*}
    \mathcal U_{\alpha\beta\alpha,F} = y_1^4y_2^2 \int_{\R^3} e\rb{\frac{m_1c_2\zeta_{\alpha\beta\alpha,1}}{c_1^2\xi_{\alpha\beta\alpha,2}y_1y_2}} e\rb{\frac{m_2c_1^2\zeta_{\alpha\beta\alpha,2}y_2}{c_2^2\xi_{\alpha\beta\alpha,1}}} F\rb{\frac{m_1c_2\sqrt{\xi_{\alpha\beta\alpha,1}}}{c_1^2\xi_{\alpha\beta\alpha,2}y_1y_2}, \frac{m_2c_1^2\xi_{\alpha\beta\alpha,2}y_2}{c_2^2\xi_{\alpha\beta\alpha,1}}}\\
    \times e\rb{-n_1x_1y_1} dx_1dx_2dx_4,
\end{multline*}
where
\begin{align*}
    \xi_{\alpha\beta\alpha,1} &= (x_1^2+1)^2 + (x_1x_4+x_2)^2,\\
    \xi_{\alpha\beta\alpha,2} &= x_1^2+x_2^2+x_4^2+1,\\
    \zeta_{\alpha\beta\alpha,1} &= x_1x_2-x_4,\\
    \zeta_{\alpha\beta\alpha,2} &= x_1^2x_2-x_1^3x_4-x_1x_4^3-x_2x_4^2-2x_1x_4.
\end{align*}
For $w=\beta\alpha\beta$ we have
\begin{multline*}
    \mathcal U_{\beta\alpha\beta,F} = y_1^3y_2^3 \int_{\R^3} e\rb{\frac{m_1c_2\zeta_{\beta\alpha\beta,1}y_1}{c_1^2\xi_{\beta\alpha\beta,1}}} e\rb{\frac{m_2c_1^2\zeta_{\beta\alpha\beta,2}}{c_2^2\xi_{\beta\alpha\beta,2}y_1^2y_2}} F\rb{\frac{m_1c_2\sqrt{\xi_{\beta\alpha\beta,2}}y_1}{c_1^2\xi_{\beta\alpha\beta,1}}, \frac{m_2c_1^2\xi_{\beta\alpha\beta,1}}{c_2^2\xi_{\beta\alpha\beta,2}y_1^2y_2}}\\
    \times e\rb{-n_2x_5y_2} dx_2dx_4dx_5,
\end{multline*}
where
\begin{align*}
    \xi_{\beta\alpha\beta,1} &= x_4^2+x_5^2+1,\\
    \xi_{\beta\alpha\beta,2} &= (x_4^2-x_2x_5)^2+x_2^2+2x_4^2+x_5^2+1,\\
    \zeta_{\beta\alpha\beta,1} &= x_4(x_2+x_5),\\
    \zeta_{\beta\alpha\beta,2} &= x_4^2x_5-x_2x_5^2-x_2.
\end{align*}
For $w=w_0$ we have
\begin{multline*}
    \mathcal U_{w_0,F} = y_1^4y_2^3 \int_{\R^4} e\rb{\frac{m_1c_2\zeta_{w_0,1}}{c_1^2\xi_{w_0,2}y_1}} e\rb{\frac{m_2c_1^2\zeta_{w_0,2}}{c_2^2\xi_{w_0,1}y_2}} F\rb{\frac{m_1c_2\sqrt{\xi_{w_0,1}}}{c_1^2\xi_{w_0,2}y_1},\frac{m_2c_1^2\xi_{w_0,2}}{c_2^2\xi_{w_0,1}y_2}}\\
    \times e(-n_1x_1y_1-n_2x_5y_2) dx_1 dx_2 dx_4 dx_5
\end{multline*}
where
\begin{align*}
\xi_{w_0,1} &= (x_4^2 + x_1x_4x_5 - x_2x_5)^2 + (x_1x_4 - x_2)^2 + 2x_4^2 + x_5^2 + 1,\\
\xi_{w_0,2} &= (x_1x_5 + x_4)^2 + x_1^2 + x_2^2 + 1,\\
\zeta_{w_0,1} &= -(x_1x_5^2 + x_2x_4 + x_4x_5 + x_1),\\
\zeta_{w_0,2} &= x_1^3x_4x_5^2 + 2x_1^2x_4^2x_5 - x_1^2x_2x_5^2 + x_1^3x_4 + x_1x_4^3 - x_1^2x_2 + x_2x_4^2 - x_2^2x_5 + 2x_1x_4 - x_5.
\end{align*}
\end{prop}

\subsubsection{The Kuznetsov integrals}\label{subsec:kuznetsov_integral}

With \Cref{prop:Fourier_integral} in place, we are ready to prove explicit formulas for the integral transforms on the arithmetic side. In terms of $\mathcal U_{w,F}$ from \eqref{eq:UwF}, the arithmetic side of \eqref{eq:ktf} is given by
\[
\sum_{w\in W} \sum_{c\in\N^2} \frac{\Kl_w(c,M,N)}{M^\eta N^\eta} \int_{T(\R_+)} \mathcal U_{w,F}(c,M,N,y) \ol{F(N\y(y))} d^\times y.
\]
\begin{prop}\label{prop:kuznetsov_integral}
For $w\in\cb{\alpha\beta\alpha,\beta\alpha\beta,w_0}$, and $(c,M,N)$ satisfying the corresponding compatibility conditions \eqref{eq:Kloosterman_compatibility}, we have
\begin{align}
    M^{-\eta} N^{-\eta} \int_{T(\R_+)} \mathcal U_{\alpha\beta\alpha,F}(c,M,N,y) \ol{F(N\y(y))} d^\times y &= (c_1c_2)^{-1} \mathcal J_{\alpha\beta\alpha,F} \rb{\sqrt{\frac{m_1m_2n_1}{c_2}}},\label{eq:UJ_aba}\\
    M^{-\eta} N^{-\eta} \int_{T(\R_+)} \mathcal U_{\beta\alpha\beta,F}(c,M,N,y) \ol{F(N\y(y))} d^\times y &= (c_1c_2)^{-1} \mathcal J_{\beta\alpha\beta,F} \rb{\frac{m_1\sqrt{m_2n_2}}{c_1}},\label{eq:UJ_bab}\\
    M^{-\eta} N^{-\eta} \int_{T(\R_+)} \mathcal U_{w_0,F}(c,M,N,y) \ol{F(N\y(y))} d^\times y &= (c_1c_2)^{-1} \mathcal J_{w_0,F} \rb{\frac{\sqrt{m_1n_1c_2}}{c_1}, \frac{\sqrt{m_2n_2}c_1}{c_2}},\label{eq:UJ_w0}
\end{align}
where the integral transforms $\mathcal J_{w,F}$ are given as follows:
\begin{align*}
    \mathcal J_{\alpha\beta\alpha,F}(A) &:= A^{-4} \int_{\R_+^2} \int_{\R^3} e\rb{\frac{A\zeta_{\alpha\beta\alpha,1}}{{\xi_{\alpha\beta\alpha,2}y_1y_2}}} e\rb{\frac{\zeta_{\alpha\beta\alpha,2}y_2}{\xi_{\alpha\beta\alpha,1}}} e\rb{-Ax_1y_1}\\
    &\hspace{2cm}\times F\rb{\frac{A\sqrt{\xi_{\alpha\beta\alpha,1}}}{\xi_{\alpha\beta\alpha,2}y_1y_2}, \frac{\xi_{\alpha\beta\alpha,2}y_2}{\xi_{\alpha\beta\alpha,1}}} \ol{F(Ay_1,y_2)} dx_1dx_2dx_4 \frac{dy_1dy_2}{y_1y_2^2},\\
    \mathcal J_{\beta\alpha\beta,F}(A) &:= A^{-3} \int_{\R_+^2} \int_{\R^3} e\rb{\frac{\zeta_{\beta\alpha\beta,1}y_1}{\xi_{\beta\alpha\beta,1}}} e\rb{\frac{\zeta_{\beta\alpha\beta,2}}{\xi_{\beta\alpha\beta,2}y_1^2y_2}} e\rb{-Ax_5y_2}\\
    &\hspace{2cm}\times F\rb{\frac{\sqrt{\xi_{\beta\alpha\beta,2}}y_1}{\xi_{\beta\alpha\beta,1}}, \frac{\xi_{\beta\alpha\beta,1}}{\xi_{\beta\alpha\beta,2}y_1^2y_2}} \ol{F(y_1,Ay_2)} dx_2dx_4dx_5 \frac{dy_1dy_2}{y_1^2y_2},\\
    \mathcal J_{w_0,F}(A_1,A_2) &:= A_1^{-4} A_2^{-3} \int_{\R_+^2} \int_{\R^4} e\rb{\frac{A_1\zeta_{w_0,1}}{\xi_{w_0,2}y_1}} e\rb{\frac{A_2\zeta_{w_0,2}}{\xi_{w_0,1}y_2}} e\rb{-A_1x_1y_1-A_2x_5y_2}\\
    &\hspace{2cm}\times F\rb{\frac{A_1\sqrt{\xi_{w_0,1}}}{\xi_{w_0,2}y_1},\frac{A_2\xi_{w_0,2}}{\xi_{w_0,1}y_2}} \ol{F(A_1y_1,A_2y_2)} dx_1dx_2dx_4dx_5 \frac{dy_1dy_2}{y_1y_2}.
\end{align*}
\end{prop}
\begin{proof}
The left hand side of \eqref{eq:UJ_aba}, through a change of variables $y_1\mapsto Ay_1/n_1$, $y_2\mapsto y_2/n_2$, can be rewritten as
\begin{multline*}
m_1^{-2}m_2^{-3/2}n_1^{-2}n_2^{-1/2}\int_{\R_+^2} \int_{\R^3} e\rb{\frac{m_1n_1n_2c_2\zeta_{\alpha\beta\alpha,1}}{c_1^2A\xi_{\alpha\beta\alpha,2}y_1y_2}} e\rb{\frac{m_2c_1^2\zeta_{\alpha\beta\alpha,2}y_2}{n_2c_2^2\xi_{\beta\alpha\beta,1}}} e\rb{-Ax_1y_1}\\
\times F\rb{\frac{m_1n_1n_2c_2\sqrt{\xi_{\alpha\beta\alpha,1}}}{{c_1^2A\xi_{\alpha\beta\alpha,2}y_1y_2}}, \frac{m_2c_1^2\xi_{\alpha\beta\alpha,2}y_2}{n_2c_2^2\xi_{\alpha\beta\alpha,1}}} \ol{F(Ay_1,y_2)} dx_1 dx_2 dx_4 \frac{dy_1dy_2}{y_1y_2^2}.
\end{multline*}
Putting $A = \sqrt{m_1m_2n_1/c_2}$ and invoking the compatibility condition $m_2c_1^2=n_2c_2^2$ yields the right hand side of \eqref{eq:UJ_aba}.

The left hand side of \eqref{eq:UJ_bab}, through a change of variables $y_1\mapsto y_1/n_1$, $y_2\mapsto Ay_2/n_2$, can be rewritten as
\begin{multline*}
m_1^{-2}m_2^{-3/2}n_1^{-1}n_2^{-3/2}\int_{\R_+^2} \int_{\R^3} e\rb{\frac{m_1c_2\zeta_{\beta\alpha\beta,1}y_1}{n_1c_1^2\xi_{\beta\alpha\beta,1}}} e\rb{\frac{m_2n_1^2n_2c_1^2\zeta_{\beta\alpha\beta,2}}{c_2^2A^2\xi_{\beta\alpha\beta,2}y_1^2y_2}} e\rb{-Ax_5y_2}\\
\times F\rb{\frac{m_1c_2\sqrt{\xi_{\beta\alpha\beta,2}}y_1}{n_1c_1^2\xi_{\beta\alpha\beta,1}}, \frac{m_2n_1^2n_2c_1^2\xi_{\beta\alpha\beta,1}}{c_2^2A^2\xi_{\beta\alpha\beta,2}y_1^2y_2}} \ol{F(y_1,Ay_2)} dx_2dx_4dx_5 \frac{dy_1dy_2}{y_1^2y_2}.
\end{multline*}
Putting $A = m_1\sqrt{m_2n_2}/c_1$ and invoking the compatibility condition $m_1c_2=n_1c_1^2$ yields the right hand side of \eqref{eq:UJ_bab}.

Finally, through a change of variables $y_1 \mapsto \frac{A_1y_1}{n_1}$, $y_2 \mapsto \frac{A_2y_2}{n_2}$, the left hand side of \eqref{eq:UJ_w0} can be rewritten as
\begin{multline*}
m_1^{-2}m_2^{-3/2}n_1^{-2}n_2^{-3/2}\int_{\R_+^2} \int_{\R^4} e\rb{\frac{m_1n_1c_2\zeta_{w_0,1}}{c_1^2A_1\xi_{w_0,2}y_1}} e\rb{\frac{m_2n_2c_1^2\zeta_{w_0,2}}{c_2^2A_2\xi_{w_0,1}y_2}} e\rb{-A_1x_1y_1-A_2x_5y_2}\\
\times \rb{\frac{m_1n_1c_2\sqrt{\xi_{w_0,1}}}{c_1^2A_1\xi_{w_0,2}y_1}, \frac{m_2n_2c_1^2\xi_{w_0,2}}{c_2^2A_2\xi_{w_0,1}y_2}} \ol{F(A_1y_1,A_2y_2)} dx_1dx_2dx_4dx_5 \frac{dy_1dy_2}{y_1y_2}.
\end{multline*}
Putting $A_1 = \sqrt{m_1n_1c_2}/c_1$ and $A_2 = \sqrt{m_2n_2}c_1/c_2$ yields the right hand side of \eqref{eq:UJ_w0}.
\end{proof}

\subsubsection{Explicit arithmetic transforms}

The Kuznetsov trace formula recalled in \Cref{subsubsec:ktf-spectral,subsubsec:ktf-arithmetic} reads
\begin{multline}
\label{eq:ktf-pre-man}
    \int_{\mathcal{A}(\GSp(4))} \frac{\overline{A_\varpi(M)} A_\varpi(N)}{\|\varpi\|^2} \vb{\pb{\widetilde W_\mu, F}}^2 d\varpi\\
    = \sum_{w\in W} \sum_{c\in\N^2} \frac{\Kl_w(c,M,N)}{M^\eta N^\eta} \int_{T(\R_+)} \int_{U_w(\R)} \mathcal F(\iota(M) c^\star wxy) \ol{F(N\y(y))} \psi_{-N}(x) dx d^\times y.
\end{multline}

By \eqref{eq:UwF} and \Cref{prop:kuznetsov_integral}, the arithmetic side now reads
\begin{equation}
    \sum_{w\in W} \sum_{c\in\N^2} \frac{\Kl_w(c,M,N)}{c_1c_2} \mathcal{J}_{w, F}(A_w(M, N, c))
\end{equation}
where
\begin{equation}
    A_w(M, N, c) := \left\{
\begin{array}{cl}
    \sqrt{\frac{m_1m_2n_1}{c_2}} & \text{for } w = \alpha \beta \alpha,  \\
    \frac{m_1\sqrt{m_2n_2}}{c_1} & \text{for } w = \beta\alpha\beta,  \\
    \rb{\frac{\sqrt{m_1n_1c_2}}{c_1}, \frac{\sqrt{m_2n_2}c_1}{c_2}} & \text{for } w = w_0.
\end{array}
    \right.
\end{equation}

This achieves the explicit description of the arithmetic side of the Kuznetsov trace formula.

\subsection{The spectral transform}
\label{subsec:33}

On the spectral side, we collect in \eqref{eq:ss} the Fourier coefficients from the generic spectrum $\mathcal A(\GSp(4))$. For the cuspidal spectrum, we obtain the Fourier coefficients directly from \eqref{eq:Fourier_def}. On the other hand, for the continuous spectrum, we need to compute the Fourier coefficients of the three families of Eisenstein series. Following the parametrisation in \cite{man3}, the Eisenstein series associated to the minimal parabolic $P_0$, the Siegel parabolic $P_S$, and the Klingen parabolic~$P_K$ are given respectively by (the analytic continuation of)
\begin{align*}
    E_0(g,s) &= \sum_{P_0 \cap \Gamma \bs \Gamma} I_0(\gamma g,s), & I_0(u\iota(y)k,s) &= y_1^{s_1+2} y_2^{s_2+3/2}, & s=(s_1,s_2)\in\C^2,\\
    E_S(g,s,\phi) &= \sum_{P_S\cap \Gamma \bs \Gamma} I_S(\gamma g,s,\phi), & I_S(u\iota(y)k,s) &= (y_1y_2)^{s+3/2} \phi(\pi_S(u\iota(y)k)), & s\in\C,\\
    E_K(g,s,\phi) &= \sum_{P_K\cap \Gamma \bs \Gamma} I_K(\gamma g,s,\phi), & I_K(u\iota(y)k,s) &= y_1^{s+2} y_2^{s/2+1} \phi(\pi_K(u\iota(y)k)), & s\in\C,\\
\end{align*}
where $\phi$ is a cusp form of $\GL(2)$, and the map $\pi_\bullet$ is the projection $\pi_\bullet : \Sp(4,\R)/K \to M_\bullet / (M_\bullet \cap K)$ with respect to the Levi decomposition $\Sp(4,\R) = P_\bullet K = N_\bullet M_\bullet K$. Throughout, we may assume~$\phi$ is a Hecke--Maaß cusp form, with Hecke eigenvalues $\lambda_m(\phi)$ for every $m\in\N$.

By unfolding, the Fourier coefficients of the Eisenstein series are given by
\begin{align*}
    \int_{U(\Z)\bs U(\R)} E_0(ug,s) \ol{\psi_M(u)} du &= A_{E_0(-,s)}(M) M^{-\eta} \psi_M(x) \widetilde W_{s_2,s_1-s_2}(\iota(M)y),\\
    \int_{U(\Z)\bs U(\R)} E_S(ug,s,\phi) \ol{\psi_M(u)} du &= A_{E_S(-,s,\phi)}(M) M^{-\eta} \psi_M(x) \widetilde W_{s,\nu}(\iota(M)y),\\
    \int_{U(\Z)\bs U(\R)} E_K(ug,s,\phi) \ol{\psi_M(u)} du &= A_{E_K(-,s,\phi)}(M) M^{-\eta} \psi_M(x) \widetilde W_{s/2+\nu,s/2-\nu}(\iota(M)y),\\
\end{align*}
where $\nu\in\C$ denote the spectral parameter of $\phi$. The first Fourier coefficients can be read off from \cite[Theorem 1.1]{man} and \cite[Theorem 7.1.2]{shahidi}:
\begin{align*}
    A_{E_0(-,s)}(1,1) &= \frac{C_{0} c_{s_2,s_1-s_2}^{-1}}{\zeta(1+2s_1-2s_2)\zeta(1-s_1+2s_2)\zeta(1+2s_2)\zeta(1+s_1)},\\
    A_{E_S(-,s,\phi)}(1,1) &= \frac{C_{S} c_{s,\nu}^{-1}}{L(1+s,\phi)\zeta(1+2s)L(1,\phi,\Ad)^{1/2}},\\
    A_{E_K(-,s,\phi)}(1,1) &= \frac{C_{K} c_{s/2+\nu,s/2-\nu}^{-1}}{L(\phi,1+s,\Sym^2)L(1,\phi,\Ad)^{1/2}},
\end{align*}
where $C_{0}$, $C_{S}$, $C_{K}$ are some nonzero absolute constants. We note that the factor $L(1, \phi,\Ad)^{1/2}$ in the denominator arises because an arithmetically normalised $\GL(2)$ cusp form $\phi$ has norm 
\begin{equation}
\|\phi\|^2 = 2 L(1, \phi,\Ad),
\end{equation}
as proven in e.g. \cite{blomer}. 

To compute the other Fourier coefficients, we use \eqref{eq:Fourier_coefficient}, which gives the ratio of Fourier coefficients in terms of Hecke eigenvalues. It remains to compute the Hecke eigenvalues of the Eisenstein series; and to do so we compute explicitly the Hecke action on the Eisenstein series. By the construction of the Eisenstein series, it suffices to compute the Hecke action on the Eisenstein summands $I_0(g,s)$, $I_S(g,s,\phi)$, $I_K(g,s,\phi)$.

First we compute the action of $T(p)$. We classify the coset representatives $\gamma\in\mathcal H(p)$ by its top left $2\times 2$ block $A$. In the following table, we follow the notations in \eqref{eq:Hecke_rep} and list the matrices $A$, the conditions on the entries of~$B$, as well as the expressions for $I_0(\gamma g,s)$, $I_S(\gamma g,s,\phi)$, $I_K(\gamma g,s,\phi)$, for $g = u\iota(y)k$.
\begin{longtable}{CCCL}
A & B &\\\hline
\rb{\bsm 1\\&1 \esm} & 0\le b_1,b_2,b_3 < p & \begin{aligned} I_0(\gamma g,s) &= p^{-s_2-3/2}y_1^{s_1+2}y_2^{s_2+3/2}\\I_S(\gamma g,s,\phi) &= p^{-s-3/2} (y_1y_2)^{s+3/2} \phi(\pi_S(\iota(y))),\\ I_K (\gamma g,s,\phi) &= p^{-s/2-1} y_1^{s+2}y_2^{s/2+1} \phi\rb{\rb{\bsm 1&b_3\\ &p\esm} \pi_K(\iota(y))}\end{aligned}\\\hline
\rb{\bsm 1&a\\&p \esm} & b_2,b_3=0, 0\le b_1 < p & \begin{aligned} I_0(\gamma g,s) &= p^{s_2-s_1-1/2}y_1^{s_1+2}y_2^{s_2+3/2}\\ I_S(\gamma g,s,\phi) &= (y_1y_2)^{s+3/2} \phi\rb{\rb{\bsm 1&a\\&p \esm}\pi_S(\iota(y))}\\ I_K (\gamma g,s,\phi) &= p^{-s/2-1} y_1^{s+2}y_2^{s/2+1} \phi\rb{\rb{\bsm p&\\ &1\esm} \pi_K(\iota(y))}\end{aligned}\\\hline
\rb{\bsm p\\&1 \esm} & b_1=b_2=0, 0\le b_3 < p & \begin{aligned} I_0(\gamma g,s) &= p^{s_1-s_2+1/2}y_1^{s_1+2}y_2^{s_2+3/2}\\ I_S(\gamma g,s,\phi) &= (y_1y_2)^{s+3/2} \phi\rb{\rb{\bsm p\\&1 \esm}\pi_S(\iota(y))}\\ I_K (\gamma g,s,\phi) &= p^{s/2+1} y_1^{s+2}y_2^{s/2+1} \phi\rb{\rb{\bsm 1&b_3\\ &p\esm} \pi_K(\iota(y))}\end{aligned}\\\hline
\rb{\bsm p\\&p \esm} & b_1=b_2=b_3=0 & \begin{aligned} I_0(\gamma g,s) &= p^{s_2+3/2}y_1^{s_1+2}y_2^{s_2+3/2}\\ I_S(\gamma g,s,\phi) &= p^{s+3/2} (y_1y_2)^{s+3/2} \phi\rb{\pi_S(\iota(y))}\\ I_K (\gamma g,s,\phi) &= p^{s/2+1} y_1^{s+2}y_2^{s/2+1} \phi\rb{\rb{\bsm p\\ &1\esm} \pi_K(\iota(y))}\end{aligned}\\\hline
\end{longtable}
From this we conclude
\begin{align*}
    \lambda(p,E_0(-,s)) &= p^{3/2} \rb{p^{s_2}+p^{-s_2}+p^{s_1-s_2}+p^{s_2-s_1}},\\
    \lambda(p,E_S(-,s,\phi)) &= p^{3/2} \rb{p^s+p^{-s}+\lambda_p(\phi)},\\\lambda(p,E_K(-,s,\phi)) &= p^{3/2} \rb{p^{s/2}+p^{-s/2}}\lambda_p(\phi).
\end{align*}

Next we compute the action of $T(p^2)$. Again we classify the coset representatives $\gamma\in\mathcal H(p^2)$ by its top left $2\times 2$ block.
\begin{longtable}{CCCL}
A & B &\\\hline
\rb{\bsm 1\\&1 \esm} & 0\le b_1,b_2,b_3 < p^2 & \begin{aligned}I_0(\gamma g,s) &= p^{-2s_2-3}y_1^{s_1+2}y_2^{s_2+3/2}\\ I_S(\gamma g,s,\phi) &= p^{-2s-3} (y_1y_2)^{s+3/2} \phi(\pi_S(\iota(y)))\\ I_K (\gamma g,s,\phi) &= p^{-s-2} y_1^{s+2}y_2^{s/2+1} \phi\rb{\rb{\bsm 1&b_3\\ &p^2\esm} \pi_K(\iota(y))}\end{aligned}\\\hline
\rb{\bsm 1&a\\&p \esm} & \begin{array}{c} 0\le b_1,b_2,b_3 < p^2\\ b_2,b_3\equiv 0 \pmod{p}\end{array} & \begin{aligned}I_0(\gamma g,s) &= p^{-s_1-2}y_1^{s_1+2}y_2^{s_2+3/2}\\ I_S(\gamma g,s,\phi) &= p^{-s-3/2} (y_1y_2)^{s+3/2} \phi\rb{\rb{\bsm 1&a\\&p \esm}\pi_S(\iota(y))}\\ I_K (\gamma g,s,\phi) &= p^{-s-2} y_1^{s+2}y_2^{s/2+1} \phi\rb{\rb{\bsm p&b_3/p\\ &p\esm} \pi_K(\iota(y))}\end{aligned}\\\hline
\rb{\bsm 1&a\\&p^2 \esm} & 0\le b_1 < p^2,\ b_2=b_3=0 & \begin{aligned}I_0(\gamma g,s) &= p^{2s_2-2s_1-1}y_1^{s_1+2}y_2^{s_2+3/2}\\ I_S(\gamma g,s,\phi) &= (y_1y_2)^{s+3/2} \phi\rb{\rb{\bsm 1&a\\&p^2 \esm}\pi_S(\iota(y))}\\ I_K (\gamma g,s,\phi) &= p^{-s-2} y_1^{s+2}y_2^{s/2+1} \phi\rb{\rb{\bsm p^2&\\ &1\esm} \pi_K(\iota(y))}\end{aligned}\\\hline
\rb{\bsm p\\&1 \esm} & \begin{array}{c} 0\le b_1,b_2,b_3 < p^2\\ b_1,b_2\equiv 0 \pmod{p}\end{array} & \begin{aligned}I_0(\gamma g,s) &= p^{s_1-2s_2-1}y_1^{s_1+2}y_2^{s_2+3/2}\\  I_S(\gamma g,s,\phi) &= p^{-s-3/2} (y_1y_2)^{s+3/2} \phi\rb{\rb{\bsm p\\&1 \esm}\pi_S(\iota(y))}\\ I_K (\gamma g,s,\phi) &= y_1^{s+2}y_2^{s/2+1} \phi\rb{\rb{\bsm 1&b_3\\ &p^2\esm} \pi_K(\iota(y))}\end{aligned}\\\hline
\rb{\bsm p\\&p \esm} & \begin{array}{c} 0\le b_1,b_2,b_3 < p^2,\\ b_1,b_2,b_3\equiv 0\pmod{p}\end{array} & \begin{aligned}I_0(\gamma g,s) &= y_1^{s_1+2}y_2^{s_2+3/2}\\ I_S(\gamma g,s,\phi) &= (y_1y_2)^{s+3/2} \phi(\pi_S(\iota(y)))\\ I_K (\gamma g,s,\phi) &= y_1^{s+2}y_2^{s/2+1} \phi\rb{\rb{\bsm p&b_3/p\\ &p\esm} \pi_K(\iota(y))}\end{aligned}\\\hline
\begin{array}{c}\rb{\bsm p&a\\&p \esm}\\ (a\ne 0)\end{array} & \begin{array}{c} 0\le b_1,b_2 < p^2,\ b_3=0\\ b_2\equiv 0\pmod{p}\\ b_1\equiv b_2a/p\pmod{p}\end{array} & \begin{aligned}I_0(\gamma g,s) &= y_1^{s_1+2}y_2^{s_2+3/2},\\  I_S(\gamma g,s,\phi) &= (y_1y_2)^{s+3/2} \phi\rb{\rb{\bsm p&a\\&p \esm}\pi_S(\iota(y))}\\ I_K (\gamma g,s,\phi) &= y_1^{s+2}y_2^{s/2+1} \phi(\pi_K(\iota(y)))\end{aligned}\\\hline
\rb{\bsm p&ap\\&p^2 \esm} & \begin{array}{c} 0\le b_1< p^2,\ b_2=b_3=0\\ b_1\equiv 0\pmod{p}\end{array} & \begin{aligned}I_0(\gamma g,s) &= p^{2s_2-s_1+1}y_1^{s_1+2}y_2^{s_2+3/2}\\  I_S(\gamma g,s,\phi) &= p^{s+3/2} (y_1y_2)^{s+3/2} \phi\rb{\rb{\bsm 1&a\\&p \esm}\pi_S(\iota(y))}\\ I_K (\gamma g,s,\phi) &= y_1^{s+2}y_2^{s/2+1} \phi\rb{\rb{\bsm p^2\\&1 \esm}\pi_K(\iota(y))}\end{aligned}\\\hline
\rb{\bsm p^2\\&1 \esm} & 0\le b_3< p^2,\ b_1=b_2=0 & \begin{aligned}I_0(\gamma g,s) &= p^{2s_1-2s_1+1}y_1^{s_1+2}y_2^{s_2+3/2}\\  I_S(\gamma g,s,\phi) &= (y_1y_2)^{s+3/2} \phi\rb{\rb{\bsm p^2\\&1 \esm}\pi_S(\iota(y))}\\ I_K (\gamma g,s,\phi) &= p^{s+2} y_1^{s+2}y_2^{s/2+1} \phi\rb{\rb{\bsm 1&b_3\\&p^2 \esm}\pi_K(\iota(y))}\end{aligned}\\\hline
\rb{\bsm p^2\\&p \esm} & \begin{array}{c} b_1=b_2=0,\ 0\le b_3< p^2,\\ b_3\equiv 0\pmod{p}\end{array} & \begin{aligned}I_0(\gamma g,s) &= p^{s_1+2}y_1^{s_1+2}y_2^{s_2+3/2}\\  I_S(\gamma g,s,\phi) &= p^{s+3/2} (y_1y_2)^{s+3/2} \phi\rb{\rb{\bsm p\\&1 \esm}\pi_S(\iota(y))}\\ I_K (\gamma g,s,\phi) &= p^{s+2} y_1^{s+2}y_2^{s/2+1} \phi\rb{\rb{\bsm p&b_3/p\\&p \esm}\pi_K(\iota(y))}\end{aligned}\\\hline
\rb{\bsm p^2\\&p^2 \esm} & b_1=b_2=b_3=0 & \begin{aligned}I_0(\gamma g,s) &= p^{2s_2+3}y_1^{s_1+2}y_2^{s_2+3/2}\\  I_S(\gamma g,s,\phi) &= p^{2s+3} (y_1y_2)^{s+3/2} \phi(\pi_S(\iota(y)))\\ I_K (\gamma g,s,\phi) &= p^{s+2} y_1^{s+2}y_2^{s/2+1} \phi\rb{\rb{\bsm p^2\\&1 \esm}\pi_K(\iota(y))}\end{aligned}\\\hline
\end{longtable}
From this we conclude
\begin{align*}
    \lambda(p^2,E_0(-,s)) &= p^3 \rb{p^{2s_2}+p^{-2s_2}+p^{s_1}+p^{-s_1}+p^{2s_2-2s_1}+p^{2s_1-2s_2}+p^{s_1-2s_2}+p^{2s_2-s_1}+\tfrac{2p-1}p},\\
    \lambda(p^2,E_S(-,s,\phi)) &= p^3 \rb{p^{2s}+\tfrac{p-1}p+p^{-2s} + \rb{p^s+p^{-s}}\lambda_p(\phi)+\lambda_{p^2}(\phi)},\\
    \lambda(p^2,E_K(-,s,\phi)) &= p^3 \rb{\tfrac{p-1}p + \rb{p^s+1+p^{-s}}\lambda_{p^2}(\phi)}.
\end{align*}
Using the Hecke relations \eqref{eq:Hecke_p2} and $\lambda_{p^2}(\phi) = \lambda_p(\phi)^2 - 1$, we find
\begin{align*}
    \lambda_{0,1}^{(2)}(p,E_0(-,s)) &= p^2 \rb{1+p^{s_1}+p^{-s_1}+p^{2s_2-s_1}+p^{s_1-2s_2}}-1,\\
    \lambda_{0,1}^{(2)}(p,E_S(-,s,\phi)) &= p^2 (1+(p^s+p^{-s})\lambda_p(\phi)))-1,\\ \lambda_{0,1}^{(2)}(p,E_K(-,s,\phi)) &= p^2(p^s+p^{-s}+\lambda_{p^2}(\phi))-1.
\end{align*}
Plugging this back into \eqref{eq:Fourier_coefficient}, we find
\begin{align*}
    A_{E_0(-,s)}(M) &= \frac{C_{0} c_{s_2,s_1-s_2}^{-1}B_{0,s}(M)}{\zeta(1+2s_1-2s_2)\zeta(1-s_1+2s_2)\zeta(1+2s_2)\zeta(1+s_1)},\\
    A_{E_S(-,s,\phi)}(M) &= \frac{C_{S} c_{s,\nu}^{-1} B_{S,s,\phi}(M)}{L(1+s,\phi)\zeta(1+2s)L(1,\phi,\Ad)^{1/2}},\\
    A_{E_K(-,s,\phi)}(M) &= \frac{C_{K} c_{s/2+\nu,s/2-\nu}^{-1} B_{K,s,\phi}(M)}{L(1+s,\phi,\Sym^2)L(1,\phi,\Ad)^{1/2}},
\end{align*}
where 
\begin{align}\label{eq:B_expression}
B_{0,s}(M) &= \mathcal B_{\mathcal X_{0,s}, \mathcal Y_{0,s}}(M), & B_{S,s,\phi}(M) &= \mathcal B_{\mathcal X_{S,s,\phi}, \mathcal Y_{S,s,\phi}}(M), & B_{K,s,\phi} (M) &= \mathcal B_{\mathcal X_{K,s,\phi}, \mathcal Y_{K,s,\phi}}(M),    
\end{align}
with $\mathcal{B}_{\mathcal{X}, \mathcal{Y}}$ defined by the formal series in Section \ref{subsec:Hecke}, and
\begin{align*}
\mathcal X_{0,s} &= \{X_{0,s,p}\}_p, & X_{0,s,p} &= p^{s_2} + p^{-s_2} + p^{s_1-s_2} + p^{s_2-s_1}\\
\mathcal Y_{0,s} &= \{Y_{0,s,p}\}_p, & Y_{0,s,p} &= 1+p^{s_1}+p^{-s_1}+p^{2s_2-s_1}+p^{s_1-2s_2},\\
\mathcal X_{S,s,\phi} &= \{X_{S,s,\phi,p}\}_p, & X_{S,s,\phi,p} &= p^s + p^{-s} + \lambda_p(\phi),\\
\mathcal Y_{S,s,\phi} &= \{Y_{S,s,\phi,p}\}_p, & Y_{S,s,\phi,p} &= 1 + (p^s+p^{-s})\lambda_p(\phi),\\
\mathcal X_{K,s,\phi} &= \{X_{K,s,\phi,p}\}_p, & X_{K,s,\phi,p} &= (p^{s/2} + p^{-s/2})\lambda_p(\phi),\\
\mathcal Y_{K,s,\phi} &= \{Y_{K,s,\phi,p}\}_p, & Y_{K,s,\phi,p} &= p^s + p^{-s} + \lambda_{p^2}(\phi). 
\end{align*}
This completes the explicit description of the spectral side of the Kuznetsov trace formula, ending the proof of \Cref{thm:ktf}. \qed

\begin{rk}
    The Fourier coefficients of the minimal Eisenstein series $E_0(-,s)$ were also computed in \cite{man3}, in terms of symplectic Schur functions. It is straightforward to verify that the expressions obtained there match the expressions here.
\end{rk}

\section{Properties of the spectral transform}
\label{sec:spectral-localisation}

\subsection{Statement of the result}

Let $f:\R_+ \to [0,1]$ be a smooth function with compact support, and $\tau_1,\tau_2\ge 0$, $X_1,X_2\ge 1$ be parameters. Consider the function
\begin{equation}\label{eq:tf_sl} 
F(y_1,y_2) = F_{\tau_1,\tau_2,X_1,X_2}(y_1,y_2) := f(X_1y_1) f(X_2y_2) y_1^{i(\tau_1+\tau_2)} y_2^{i\tau_2}. 
\end{equation}

In this section we study the transform 
\begin{equation}
I(\mu) = \pb{\widetilde W_\mu, F_{\tau_1,\tau_2,X_1,X_2}} = \int_0^\infty \int_0^\infty \widetilde W_\mu(y) f(X_1y_1) f(X_2y_2) y_1^{it_1} y_2^{it_2} y^{-2\eta} \frac{dy_1dy_2}{y_1y_2}
\end{equation}
appearing in the spectral side of \eqref{eq:ktf}, where we let $t_1 = \tau_1 + \tau_2$ and $t_2 = \tau_2$. 

\begin{thm}[Spectral localisation]\label{thm:sl} 
Let $A\ge 1$ be fixed, and $0\le \tau_1 \le \tau_2$ such that $\tau_2 \gg_A 1$. Set $t_1 := \tau_1 + \tau_2$, $t_2 := \tau_2$. Suppose $\mu \in\C^2$ satisfy \eqref{eq:ma} and either \eqref{eq:mat} or \eqref{eq:man}. Let $X_1,X_2\ge 1$ be parameters, and $F = F_{\tau_1,\tau_2,X_1,X_2}$ as in \eqref{eq:tf_sl}. Write
\begin{equation*}
\mathcal C(\mu) := (1+\vb{\mu_1})^{-1/2} (1+\vb{\mu_2})^{-1/2} (1+\vb{\mu_1+\mu_2})^{-1/2} (1+\vb{\mu_2-\mu_1})^{-1/2}
\end{equation*}
\begin{enumerate}[label=(\Alph*)]
\item \textbf{Decay.} Suppose $X_1 = X_2 = 1$. Then the integral $I(\mu)$ satisfies
\[
\vb{I(\mu)} \ll_A \Big(1+ \sum_{j=1}^2 \vb{\Im\mu_j-\tau_j}\Big)^{-A} \mathcal C(\mu).
\]
\item \textbf{Spectral localisation.} Suppose $X_1 = X_2 = 1$,  $\tau_1, \tau_2 - \tau_1 \gg_A 1$, and $|\Im\mu_j - \tau_j|\le c$ for some sufficiently small constant $c>0$ which depends only on $f$. Then the integral $I(\mu)$ satisfies
\[
\vb{I(\mu)} \asymp \mathcal C(\mu).
\]
\item \textbf{Non-tempered amplification.} Now suppose $X_2\gg 1$, $\tau_1 = \tau_2 = \tau \gg_A 1$ such that $\tau \gg~X_1^{1/A}$, and $R(\mu) := max\{|\Re\mu_1|, | \Re\mu_2|\} \ge \varepsilon$, with $|\Im\mu_j - \tau_j|\le c$ for some sufficiently small constant~$c>0$ which depends only on $f$. Then the integral $I(\mu)$ satisfies
\[
\vb{I(\mu)} \asymp X_1^2 X_2^{\frac 32+R(\mu)} \mathcal C(\mu).
\]
\end{enumerate}
\end{thm}

\begin{rk}
\Cref{thm:sl} may be interpreted as follows. Statements (A) and (B) say that under these assumptions, the integral $I(\mu)$ behaves as a bump at $\Im \mu_1 = \tau_1$ and $\Im \mu_2 = \tau_2$ with polynomial size in $\mu_1,\mu_2$, and decays rapidly away from this point. An important implication of this is that for the test function $F$, the spectral transform $\langle\tilde W_\mu,F\rangle$ arising in the Kuznetsov trace formula \eqref{eq:ktf} effectively picks the automorphic forms $\varpi$ with spectral parameters $\mu(\varpi)\approx(\tau_1,\tau_2)$. Meanwhile, Statement (C) says that for non-tempered spectral parameters $\mu$, the integral $I(\mu)$ is amplified by a factor depending on $R(\mu)$; such an amplification is useful in bounding the size of the non-tempered spectrum, in the form of a density theorem. 
\end{rk}

The remaining of this section is dedicated to the proof of \Cref{thm:sl}.

\subsection{Preparing the stage}

Using the Mellin-Barnes integral representation  \eqref{eq:IshiiMoriyama} of $\widetilde{W}_\mu$, and integrating over $y_1,y_2$, we recognize the Mellin transform of $f$ and rewrite
\begin{multline}\label{eq:igf} 
I(\mu) = \frac{C}{(2\pi i)^4} \int_{(\sigma_1)}\int_{(\sigma_2)}\int_{(\sigma_3)}\int_{(\sigma_4)} \frac{\hat f(-2+it_1-u_3) \hat f(-3/2+it_2-u_4)}{\pi^{u_3+u_4} X_1^{-2+it_1-u_3} X_2^{-\frac 32+it_2-u_4}}\\
\times \frac{\Gamma\rb{\frac{u_1+\mu_1}2}\Gamma\rb{\frac{u_1-\mu_1}2}\Gamma\rb{\frac{u_2+\mu_2}2}\Gamma\rb{\frac{u_2-\mu_2}2}\Gamma\rb{\frac{u_3}2}\Gamma\rb{\frac{u_3-u_1-u_2}2}\Gamma\rb{\frac{u_4-u_1}2}\Gamma\rb{\frac{u_4-u_2}2}}{\vb{\Gamma\rb{\frac{1+2i\Im\mu_1}2} \Gamma\rb{\frac{1+2i\Im\mu_2}2} \Gamma\rb{\frac{1+i\Im\mu_1+i\Im\mu_2}2} \Gamma\rb{\frac{1-i\Im\mu_1+i\Im\mu_2}2}}} du_4 du_3 du_2 du_1,
\end{multline}
where $\hat f$ is the Mellin transform of $f$. Since $f$ has compact support, it follows from the Paley--Wiener theorem that
\begin{equation}\label{eq:Mellin_decay}
\hat{f}(s) \ll_{A,f} |s|^{-A}  C^{\vb{\Re(s)} + A}
\end{equation}
for any $A\ge 0$, where $C\ge 1$ is a constant that depends only on the support of $f$. Moreover, $\hat f$ is non-vanishing in a horizontal strip $\vb{\Im s} \le c$ for sufficiently small $c>0$. Recall also Stirling's approximation for $\Gamma$-functions:
\[
\Gamma(x+iy) \asymp_x (1+|y|)^{x-\tfrac12} e^{-\frac{\pi}{2}|y|}, \quad |y|\to\infty.
\]

We split the problem into several cases, depending on the sizes and the relative positions of the spectral parameters $\mu_1,\mu_2$.

\subsection{Small parameters} Suppose $\vb{\mu_1}+\vb{\mu_2} \le 10$. In this case, it suffices to prove Statement (A); the other statements are void. We consider the $\Gamma$-quotient in \eqref{eq:igf}:
\[
\frac{\Gamma\rb{\frac{u_1+\mu_1}2}\Gamma\rb{\frac{u_1-\mu_1}2}\Gamma\rb{\frac{u_2+\mu_2}2}\Gamma\rb{\frac{u_2-\mu_2}2}\Gamma\rb{\frac{u_3}2}\Gamma\rb{\frac{u_3-u_1-u_2}2}\Gamma\rb{\frac{u_4-u_1}2}\Gamma\rb{\frac{u_4-u_2}2}}{\vb{\Gamma\rb{\frac{1+2i\Im\mu_1}2} \Gamma\rb{\frac{1+2i\Im\mu_2}2} \Gamma\rb{\frac{1+i\Im\mu_1+i\Im\mu_2}2} \Gamma\rb{\frac{1-i\Im\mu_1+i\Im\mu_2}2}}}.
\]
By Stirling's formula, this quotient has an absolutely bounded size, and has exponential decay in the $u_i$'s outside a bounded box, ensuring absolute convergence of all the integrals. This says the size of $I(\mu)$ is governed by the rapid decay of $\hat f$, and we conclude from \eqref{eq:Mellin_decay} that
\[
\vb{I(\mu)} \ll_{A,f} (\vb{\mu_1}+\vb{\mu_2})^{-A}. 
\]
In other words, since the spectral parameters $\mu_i$ are uniformly bounded, we indeed have Statement~(A) in this case.

\subsection{Big parameters, non-degenerate case} 
Suppose $\vb{\mu_1}+\vb{\mu_2} \ge 10$, and that the spectral parameters are non-degenerate, in the sense that
\[
\mu_i \ne 0 \quad \text{ and } \quad \mu_1 \ne \pm \mu_2.
\]
In this case, we estimate the integral by shifting contours and collecting residues. For this purpose, it is more convenient to work instead with variables $r_1 := u_1$, $r_2 := u_2$, $r_3 := u_3-u_1-u_2$, $r_4 := u_4-u_1-u_2$, and write
\begin{multline*}
I(\mu) = \frac{C}{(2\pi i)^4} \int_{(\rho_1)}\int_{(\rho_2)}\int_{(\rho_3)}\int_{(\rho_4)} \frac{\hat f(-2+it_1-r_1-r_2-r_3) \hat f(-\frac 32+it_2-r_1-r_2-r_4)}{\pi^{2r_1+2r_2+r_3+r_4} X_1^{-2+it_1-r_1-r_2-r_3} X_2^{-3/2+it_2-r_1-r_2-r_4}}\\
\times \frac{\Gamma\rb{\frac{r_1+\mu_1}2}\Gamma\rb{\frac{r_1-\mu_1}2}\Gamma\rb{\frac{r_2+\mu_2}2}\Gamma\rb{\frac{r_2-\mu_2}2}\Gamma\rb{\frac{r_1+r_2+r_3}2}\Gamma\rb{\frac{r_3}2}\Gamma\rb{\frac{r_1+r_4}2}\Gamma\rb{\frac{r_2+r_4}2}}{\vb{\Gamma\rb{\frac{1+i\Im\mu_1+i\Im\mu_2}2}\Gamma\rb{\frac{1+2i\Im\mu_1}2}\Gamma\rb{\frac{1+2i\Im\mu_2}2}\Gamma\rb{\frac{1-i\Im\mu_1+i\Im\mu_2}2}}} dr_4 dr_3 dr_2 dr_1,
\end{multline*}
where $\rho_1 > |\Re\mu_1|$, $\rho_2 > |\Re\mu_2|$, $\rho_3,\rho_4>0$. Without loss of generality, we may assume 
\[
\tfrac 12 < \rho_2 < \rho_1 < 1.
\]

\subsubsection{Residues and cancellations}

By shifting $\rho_3$, $\rho_4$ and then $\rho_2$ to $-\infty$, we collect $16$ series of residues of the form
\begin{multline*}
\frac{C}{2\pi i} \int_{(\rho_1)} dr_1 \sum_{k_4=0}^\infty \sum_{k_3=0}^\infty \sum_{k_2} \frac{8\epsilon(-1)^{k_2+k_3+k_4}}{k_2!k_3!k_4!} \frac{\hat f(-2+it_1-\alpha_1) \hat f(-\frac 32+it_2-\alpha_2)}{\pi^{\alpha_1+\alpha_2} X_1^{-2+it_1-\alpha_1} X_2^{-3/2+it_2-\alpha_2}}\\
\times \frac{\Gamma\rb{\psi_1} \Gamma\rb{\psi_2} \Gamma\rb{\psi_3} \Gamma\rb{\psi_4} \Gamma\rb{\psi_5}}{\vb{\Gamma\rb{\frac{1+i\Im\mu_1+i\Im\mu_2}2}\Gamma\rb{\frac{1+2i\Im\mu_1}2}\Gamma\rb{\frac{1+2i\Im\mu_2}2}\Gamma\rb{\frac{1-i\Im\mu_1+i\Im\mu_2}2}}},
\end{multline*}
where $\sum_{k_2}$ is some summation (not necessarily from $0$ to $\infty$), $\epsilon \in \cb{\pm 1}$ and the tuples $\alpha = (\alpha_1,\alpha_2)$ and $\psi = (\psi_1,\ldots,\psi_5)$ are obtained by substituting appropriate values for $r_3,r_4$, and $r_2$. We label the series of poles as follows:
\[
\begin{array}{|c|cc|}
\hline
& r_2 = \mu_2-k_2 & \text{(Aa)}\\
r_3 = -2k_3 & r_2 = -\mu_2-k_2 & \text{(Ab)}\\
r_4 = -r_1-2k_4 & r_2 = r_1-2k_2+2k_4 & \text{(Ac)}\\
& r_2 = -r_1-2k_2+2k_3 & \text{(Ad)}\\\hline
& r_2 = \mu_2-2k_2 & \text{(Ba)}\\
r_3 = -2k_3 & r_2 = -\mu_2-2k_2 & \text{(Bb)}\\
r_4 = -r_2-2k_4 & r_2 = r_1+2k_2-2k_4 & \text{(Bc)}\\
& r_2 = -r_1-2k_2+2k_3 & \text{(Bd)}\\\hline
\end{array}
\ 
\begin{array}{|c|cc|}
\hline
& r_2 = \mu_2-2k_2 & \text{(Ca)}\\
r_3 = -r_1-r_2-2k_3 & r_2 = -\mu_2-2k_2 & \text{(Cb)}\\
r_4 = -r_1-2k_4 & r_2 = r_1-2k_2+2k_4 & \text{(Cc)}\\
& r_2 = -r_1+2k_2-2k_3 & \text{(Cd)}\\\hline
& r_2 = \mu_2-2k_2 & \text{(Da)}\\
r_3 = -r_1-r_2-2k_3 & r_2 = -\mu_2-2k_2 & \text{(Db)}\\
r_4 = -r_2-2k_4 & r_2 = r_1+2k_2-2k_4 & \text{(Dc)}\\
& r_2 = -r_1+2k_2-2k_3 & \text{(Dd)}\\\hline
\end{array}
\]
By a suitable swap of variables $k_2,k_3,k_4$, we see that the following series of residues cancel in pairs:
\begin{center}\begin{tabular}{c|c}
Cancellations & Swap \\\hline
(Ac), (Bc) & $k_2\leftrightarrow k_4$\\
(Cc), (Dc) & $k_2\leftrightarrow k_4$
\end{tabular}
\quad
\begin{tabular}{c|c}
Cancellations & Swap \\\hline
(Ad), (Cd) & $k_2\leftrightarrow k_3$\\
(Bd), (Dd) & $k_2\leftrightarrow k_3$
\end{tabular}
\end{center}

For the remaining $8$ series of residues, we shift also $\rho_1$ to $-\infty$, obtaining $32$ series of residues, labelled as follows:
\[
\begin{array}{|cc|}
\hline
\text{(Aa)} &\\\hline
r_1 = \mu_1-2k_1 & \text{(Aa1)}\\
r_1 = -\mu_1-2k_1 & \text{(Aa2)}\\
r_1 = \mu_2+2k_1-2k_2-2k_4 & \text{(Aa3)}\\
r_1 = -\mu_2-2k_1+2k_2+2k_3 & \text{(Aa4)}\\\hline
\end{array}
\ 
\begin{array}{|cc|}
\hline
\text{(Ab)} &\\\hline
r_1 = \mu_1-2k_1 & \text{(Ab1)}\\
r_1 = -\mu_1-2k_1 & \text{(Ab2)}\\
r_1 = \mu_2-2k_1+2k_2+2k_3 & \text{(Ab3)}\\
r_1 = -\mu_2+2k_1-2k_2-2k_4 & \text{(Ab4)}\\\hline
\end{array}
\]
\[
\begin{array}{|cc|}
\hline
\text{(Ba)} &\\\hline
r_1 = \mu_1-2k_1 & \text{(Ba1)}\\
r_1 = -\mu_1-2k_1 & \text{(Ba2)}\\
r_1 = \mu_2-2k_1-2k_2+2k_4 & \text{(Ba3)}\\
r_1 = -\mu_2-2k_1+2k_2+2k_3 & \text{(Ba4)}\\\hline
\end{array}
\ 
\begin{array}{|cc|}
\hline
\text{(Bb)} &\\\hline
r_1 = \mu_1-2k_1 & \text{(Bb1)}\\
r_1 = -\mu_1-2k_1 & \text{(Bb2)}\\
r_1 = \mu_2-2k_1+2k_2+2k_3 & \text{(Bb3)}\\
r_1 = -\mu_2-2k_1-2k_2+2k_4 & \text{(Bb4)}\\\hline
\end{array}
\]
\[
\begin{array}{|cc|}
\hline
\text{(Ca)} &\\\hline
r_1 = \mu_1-2k_1 & \text{(Ca1)}\\
r_1 = -\mu_1-2k_1 & \text{(Ca2)}\\
r_1 = \mu_2+2k_1-2k_2-2k_4 & \text{(Ca3)}\\
r_1 = -\mu_2+2k_1+2k_2-2k_3 & \text{(Ca4)}\\\hline
\end{array}
\ 
\begin{array}{|cc|}
\hline
\text{(Cb)} &\\\hline
r_1 = \mu_1-2k_1 & \text{(Cb1)}\\
r_1 = -\mu_1-2k_1 & \text{(Cb2)}\\
r_1 = \mu_2+2k_1+2k_2-2k_3 & \text{(Cb3)}\\
r_1 = -\mu_2+2k_1-2k_2-2k_4 & \text{(Cb4)}\\\hline
\end{array}
\]
\[
\begin{array}{|cc|}
\hline
\text{(Da)} &\\\hline
r_1 = \mu_1-2k_1 & \text{(Da1)}\\
r_1 = -\mu_1-2k_1 & \text{(Da2)}\\
r_1 = \mu_2-2k_1-2k_2+2k_4 & \text{(Da3)}\\
r_1 = -\mu_2+2k_1+2k_2-2k_3 & \text{(Da4)}\\\hline
\end{array}
\ 
\begin{array}{|cc|}
\hline
\text{(Db)} &\\\hline
r_1 = \mu_1-2k_1 & \text{(Db1)}\\
r_1 = -\mu_1-2k_1 & \text{(Db2)}\\
r_1 = \mu_2+2k_1+2k_2-2k_3 & \text{(Db3)}\\
r_1 = -\mu_2-2k_1-2k_2+2k_4 & \text{(Db4)}\\\hline
\end{array}
\]
By a suitable swap of variables $k_1,k_3,k_4$, we see that the following series of residues cancel in pairs:
\begin{center}\begin{tabular}{c|c}
Cancellations & Swap \\\hline
(Aa3), (Ba3) & $k_1\leftrightarrow k_4$\\
(Aa4), (Ca4) & $k_1\leftrightarrow k_3$\\
(Ab3), (Cb3) & $k_1\leftrightarrow k_3$\\
(Ab4), (Cb4) & $k_1\leftrightarrow k_4$\\
\end{tabular}
\quad
\begin{tabular}{c|c}
Cancellations & Swap \\\hline
(Ba4), (Da4) & $k_1\leftrightarrow k_3$\\
(Bb3), (Db3) & $k_1\leftrightarrow k_3$\\
(Ca3), (Da3) & $k_1\leftrightarrow k_4$\\
(Cb4), (Db4) & $k_1\leftrightarrow k_4$\\
\end{tabular}
\end{center}

There are more cancellations among the remaining series of residues, as stated in the following lemma, but they are more subtle to prove.
\begin{lem}\label{lem:extra_cancellation}
The following series of residues cancel in pairs:
\begin{center}\textup{\begin{tabular}{cccc}
\{(Ca1), (Ca2)\}, & \{(Cb1), (Cb2)\}, & \{(Da1), (Db1)\}, & \{(Da2), (Db2)\}.
\end{tabular}}\end{center}
\end{lem}
\begin{proof}
We only work out in detail the cancellation of the pair (Ca1) and (Ca2); the proof of other cancellations is similar. The two series of residues have the form
\begin{multline}\label{eq:Ca1}\tag{Ca1}
    \sum_{k_1,k_2,k_3,k_4 = 0}^\infty \frac{16C(-1)^{k_1+k_2+k_3+k_4}}{k_1!k_2!k_3!k_4!} \frac{\hat f(-2+it_1+2k_3) \hat f(-\tfrac 32 + it_2-\mu_2+2k_2+2k_4)}{\pi^{\mu_2-2k_2-2k_3-2k_4} X_1^{-2+it_1+2k_3} X_2^{-3/2+it_2-\mu_2+2k_2+2k_4}}\\
    \times \frac{\Gamma\rb{\mu_1-k_1}\Gamma\rb{\mu_2-k_2}\Gamma\rb{\tfrac{-\mu_1-\mu_2}2+k_1+k_2-k_3}\Gamma\rb{\tfrac{-\mu_1+\mu_2}2+k_1-k_2-k_4}}{\vb{\Gamma\rb{\frac{1+i\Im\mu_1+i\Im\mu_2}2}\Gamma\rb{\frac{1+2i\Im\mu_1}2}\Gamma\rb{\frac{1+2i\Im\mu_2}2}\Gamma\rb{\frac{1-i\Im\mu_1+i\Im\mu_2}2}}},
\end{multline}
and
\begin{multline}\label{eq:Ca2}\tag{Ca2}
    \sum_{k_1,k_2,k_3,k_4 = 0}^\infty \frac{16C(-1)^{k_1+k_2+k_3+k_4}}{k_1!k_2!k_3!k_4!} \frac{\hat f(-2+it_1+2k_3) \hat f(-\tfrac 32 + it_2-\mu_2+2k_2+2k_4)}{\pi^{\mu_2-2k_2-2k_3-2k_4} X_1^{-2+it_1+2k_3} X_2^{-3/2+it_2-\mu_2+2k_2+2k_4}}\\
    \times \frac{\Gamma\rb{-\mu_1-k_1}\Gamma\rb{\mu_2-k_2}\Gamma\rb{\tfrac{\mu_1-\mu_2}2+k_1+k_2-k_3}\Gamma\rb{\tfrac{\mu_1+\mu_2}2+k_1-k_2-k_4}}{\vb{\Gamma\rb{\frac{1+i\Im\mu_1+i\Im\mu_2}2}\Gamma\rb{\frac{1+2i\Im\mu_1}2}\Gamma\rb{\frac{1+2i\Im\mu_2}2}\Gamma\rb{\frac{1-i\Im\mu_1+i\Im\mu_2}2}}}.
\end{multline}
In particular, (Ca2) can be obtained from (Ca1) by flipping $\mu_1 \leftrightarrow -\mu_1$ in the $\Gamma$-factors. We recall the rising and falling factorials:
\begin{align*}
(a)^{(n)} &= a(a+1)\cdots(a+n-1),\\
(a)_{(n)} &= a(a-1)\cdots(a-n+1),
\end{align*}
and rewrite (Ca1) as
\begin{multline*}
    \sum_{k_1,k_2,k_3,k_4 = 0}^\infty \frac{16C(-1)^{k_1+k_2+k_3+k_4}}{k_1!k_2!k_3!k_4!} \frac{\hat f(-2+it_1+2k_3) \hat f(-\tfrac 32 + it_2-\mu_2+2k_2+2k_4)}{\pi^{\mu_2-2k_2-2k_3-2k_4} X_1^{-2+it_1+2k_3} X_2^{-3/2+it_2-\mu_2+2k_2+2k_4}}\\
    \times \frac{\Gamma\rb{\mu_1}\Gamma\rb{\mu_2}\Gamma\rb{\tfrac{-\mu_1-\mu_2}2}\Gamma\rb{\tfrac{-\mu_1+\mu_2}2}}{\vb{\Gamma\rb{\frac{1+i\Im\mu_1+i\Im\mu_2}2}\Gamma\rb{\frac{1+2i\Im\mu_1}2}\Gamma\rb{\frac{1+2i\Im\mu_2}2}\Gamma\rb{\frac{1-i\Im\mu_1+i\Im\mu_2}2}}}\\
    \times \frac{1}{\rb{\mu_1-1}_{(k_1)}\rb{\mu_2-1}_{(k_2)}\rb{\tfrac{-\mu_1-\mu_2}2-1}_{(k_3-k_1-k_2)}\rb{\tfrac{-\mu_1+\mu_2}2-1}_{(k_2+k_4-k_1)}}.
\end{multline*}
Ignoring the factors which are common in (Ca1) and (Ca2), we extract the $k_1$-sum:
\begin{equation}\label{eq:k1_sum}
\sum\limits_{k_1=0}^\infty \frac{(-1)^{k_1}}{k_1! \rb{\mu_1-1}_{(k_1)} \rb{\frac{-\mu_1-\mu_2}2-1}_{(k_3-k_1-k_2)} \rb{\frac{-\mu_1+\mu_2}2-1}_{(k_2+k_4-k_1)}}.
\end{equation}
Using the formula $(a)_{(m-n)} = (a)_{(m)}/(a-m+1)^{(n)}$, we rewrite \eqref{eq:k1_sum} as
\begin{align}\nonumber
&\frac{1}{\rb{\frac{-\mu_1-\mu_2}2-1}_{(k_3-k_2)} \rb{\frac{-\mu_1+\mu_2}2-1}_{(k_2+k_4)}} \sum\limits_{k_1=0}^\infty \frac{(-1)^{k_1} \rb{\frac{-\mu_1-\mu_2}2-k_3+k_2}^{(k_1)} \rb{\frac{-\mu_1+\mu_2}2-k_2-k_4}^{(k_1)}}{k_1! \rb{\mu_1-1}_{(k_1)}}\\
&= \frac{1}{\rb{\frac{-\mu_1-\mu_2}2-1}_{(k_3-k_2)} \rb{\frac{-\mu_1+\mu_2}2-1}_{(k_2+k_4)}} \sum\limits_{k_1=0}^\infty \frac{\rb{\frac{-\mu_1-\mu_2}2-k_3+k_2}^{(k_1)} \rb{\frac{-\mu_1+\mu_2}2-k_2-k_4}^{(k_1)}}{k_1! \rb{-\mu_1+1}^{(k_1)}}\nonumber\\
&= \frac{_2 F_1\rb{\frac{-\mu_1-\mu_2}2-k_3+k_2,\frac{-\mu_1+\mu_2}2-k_2-k_4;-\mu_1+1;1}}{\rb{\frac{-\mu_1-\mu_2}2-1}_{(k_3-k_2)} \rb{\frac{-\mu_1+\mu_2}2-1}_{(k_2+k_4)}}.\label{eq:k1_sum2}
\end{align}
Using Gauß's formula 
\begin{equation}
_2F_1(a,b;c;1) = \frac{\Gamma(c)\Gamma(c-a-b)}{\Gamma(c-a)\Gamma(c-b)}, \quad \text{ for } \quad \Re(c) > \Re(a+b),
\end{equation}
we evaluate \eqref{eq:k1_sum2} as
\begin{align*}
&\frac{\Gamma\rb{-\mu_1+1} \Gamma\rb{k_3+k_4+1}}{\rb{\frac{-\mu_1-\mu_2}2-1}_{(k_3-k_2)} \rb{\frac{-\mu_1+\mu_2}2-1}_{(k_2+k_4)} \Gamma\rb{\frac{-\mu_1+\mu_2}2+k_3-k_2+1}\Gamma\rb{\frac{-\mu_1-\mu_2}2+k_2+k_4+1}}\\
&= \frac{\Gamma\rb{-\mu_1+1}\Gamma\rb{k_3+k_4+1}\Gamma\rb{\tfrac{-\mu_1+\mu_2}2+1}^{-1}\Gamma\rb{\tfrac{-\mu_1-\mu_2}2+1}^{-1}}{\rb{\tfrac{-\mu_1-\mu_2}2-1}_{(k_3-k_2)} \rb{\tfrac{-\mu_1+\mu_2}2-1}_{(k_2+k_4)} \rb{\tfrac{-\mu_1+\mu_2}2+1}^{(k_3-k_2)} \rb{\tfrac{-\mu_1-\mu_2}2+1}^{(k_2+k_4)}}.
\end{align*}
Hence \eqref{eq:Ca1} is equal to
\begin{multline*}
\sum_{k_2,k_3,k_4 = 0}^\infty \frac{16C(-1)^{k_2+k_3+k_4}}{k_2!k_3!k_4!} \frac{\hat f(-2+it_1+2k_3) \hat f(-\tfrac 32 + it_2-\mu_2+2k_2+2k_4)}{\pi^{\mu_2-2k_2-2k_3-2k_4} X_1^{-2+it_1+2k_3} X_2^{-3/2+it_2-\mu_2+2k_2+2k_4}}\\
\times \frac{\Gamma\rb{\mu_1}\Gamma\rb{\mu_2}\Gamma\rb{-\mu_1+1}\Gamma\rb{k_3+k_4+1}}{\vb{\Gamma\rb{\frac{1+i\Im\mu_1+i\Im\mu_2}2}\Gamma\rb{\frac{1+2i\Im\mu_1}2}\Gamma\rb{\frac{1+2i\Im\mu_2}2}\Gamma\rb{\frac{1-i\Im\mu_1+i\Im\mu_2}2}}\rb{\frac{-\mu_1+\mu_2}2}\rb{\frac{-\mu_1-\mu_2}2}}\\
\times \frac{1}{\rb{\mu_2-1}_{(k_2)} \rb{\frac{-\mu_1-\mu_2}2-1}_{(k_3-k_2)} \rb{\frac{-\mu_1+\mu_2}2-1}_{(k_2+k_4)} \rb{\frac{-\mu_1+\mu_2}2+1}^{(k_3-k_2)} \rb{\frac{-\mu_1-\mu_2}2+1}^{(k_2+k_4)}}.
\end{multline*}
Similarly, \eqref{eq:Ca2} is equal to
\begin{multline*}
\sum_{k_2,k_3,k_4 = 0}^\infty \frac{16C(-1)^{k_2+k_3+k_4}}{k_2!k_3!k_4!} \frac{\hat f(-2+it_1+2k_3) \hat f(-\tfrac 32 + it_2-\mu_2+2k_2+2k_4)}{\pi^{\mu_2-2k_2-2k_3-2k_4} X_1^{-2+it_1+2k_3} X_2^{-3/2+it_2-\mu_2+2k_2+2k_4}}\\
\times \frac{\Gamma\rb{-\mu_1}\Gamma\rb{\mu_2}\Gamma\rb{\mu_1+1}\Gamma\rb{k_3+k_4+1}}{\vb{\Gamma\rb{\frac{1+i\Im\mu_1+i\Im\mu_2}2}\Gamma\rb{\frac{1+2i\Im\mu_1}2}\Gamma\rb{\frac{1+2i\Im\mu_2}2}\Gamma\rb{\frac{1-i\Im\mu_1+i\Im\mu_2}2}}\rb{\frac{\mu_1+\mu_2}2}\rb{\frac{\mu_1-\mu_2}2}}\\
\times \frac{1}{\rb{\mu_2-1}_{(k_2)} \rb{\frac{\mu_1-\mu_2}2-1}_{(k_3-k_2)} \rb{\frac{\mu_1+\mu_2}2-1}_{(k_2+k_4)} \rb{\frac{\mu_1+\mu_2}2+1}^{(k_3-k_2)} \rb{\frac{\mu_1-\mu_2}2+1}^{(k_2+k_4)}}.
\end{multline*}
Using Euler's reflection formula, we find
\[
\Gamma(\mu_1)\Gamma(-\mu_1+1) = \frac{\pi}{\sin(\pi\mu_1)} = - \frac{\pi}{\sin(-\pi\mu_1)} = -\Gamma(-\mu_1)\Gamma(\mu_1+1).
\]
Using this and the observation that $(a)_{(n)} = (-1)^n (-a)^{(n)}$, we find that the series of residues \eqref{eq:Ca1} and \eqref{eq:Ca2} cancel. 
\end{proof}

The remaining $8$ series of residues have the form
\begin{multline}\label{eq:8series}
\sum_{k_1,k_2,k_3,k_4=0}^\infty \frac{16C(-1)^{k_1+k_2+k_3+k_4}}{k_1!k_2!k_3!k_4!} \frac{\hat f(-2+it_1-\alpha_1) \hat f(-\frac 32+it_2-\alpha_2)}{\pi^{\alpha_1+\alpha_2} X_1^{-2+it_1-\alpha_1} X_2^{-3/2+it_2-\alpha_2}}\\
\times \frac{\Gamma\rb{\psi_1} \Gamma\rb{\psi_2} \Gamma\rb{\psi_3} \Gamma\rb{\psi_4} }{\vb{\Gamma\rb{\frac{1+i\Im\mu_1+i\Im\mu_2}2}\Gamma\rb{\frac{1+2i\Im\mu_1}2}\Gamma\rb{\frac{1+2i\Im\mu_2}2}\Gamma\rb{\frac{1-i\Im\mu_1+i\Im\mu_2}2}}},
\end{multline}
where the tuples $\alpha = (\alpha_1,\alpha_2)$, $\psi = (\psi_1,\psi_2,\psi_3,\psi_4)$ are given as follows:
\[
\begin{array}{c|l}
\text{Series} & \ \alpha,\psi\\\hline
\text{(Aa1)} & \begin{array}{l} \alpha = \rb{\mu_1+\mu_2-2k_1-2k_2-2k_3,\mu_2-2k_2-2k_4}\\ \psi = \rb{\mu_1-k_1,\mu_2-k_2,\tfrac{\mu_1+\mu_2}2-k_1-k_2-k_3,\tfrac{-\mu_1+\mu_2}2+k_1-k_2-k_4}\end{array}\\\hline
\text{(Aa2)} & \begin{array}{l} \alpha = \rb{-\mu_1+\mu_2-2k_1-2k_2-2k_3,\mu_2-2k_2-2k_4}\\ \psi = \rb{-\mu_1-k_1,\mu_2-k_2,\tfrac{-\mu_1+\mu_2}2-k_1-k_2-k_3,\tfrac{\mu_1+\mu_2}2+k_1-k_2-k_4}\end{array}\\\hline
\text{(Ab1)} & \begin{array}{l} \alpha = \rb{\mu_1-\mu_2-2k_1-2k_2-2k_3,-\mu_2-2k_2-2k_4}\\ \psi = \rb{\mu_1-k_1,-\mu_2-k_2,\tfrac{\mu_1-\mu_2}2-k_1-k_2-k_3,\tfrac{-\mu_1-\mu_2}2+k_1-k_2-k_4}\end{array}\\\hline
\text{(Ab2)} & \begin{array}{l} \alpha = \rb{-\mu_1-\mu_2-2k_1-2k_2-2k_3,-\mu_2-2k_2-2k_4}\\ \psi = \rb{-\mu_1-k_1,-\mu_2-k_2,\tfrac{-\mu_1-\mu_2}2-k_1-k_2-k_3,\tfrac{\mu_1-\mu_2}2+k_1-k_2-k_4}\end{array}\\\hline
\text{(Ba1)} & \begin{array}{l} \alpha = \rb{\mu_1+\mu_2-2k_1-2k_2-2k_3,\mu_1-2k_1-2k_4}\\ \psi = \rb{\mu_1-k_1,\mu_2-k_2,\tfrac{\mu_1+\mu_2}2-k_1-k_2-k_3,\tfrac{\mu_1-\mu_2}2-k_1+k_2-k_4}\end{array}\\\hline
\text{(Ba2)} & \begin{array}{l} \alpha = \rb{-\mu_1+\mu_2-2k_1-2k_2-2k_3,-\mu_1-2k_1-2k_4}\\ \psi = \rb{-\mu_1-k_1,\mu_2-k_2,\tfrac{-\mu_1+\mu_2}2-k_1-k_2-k_3,\tfrac{-\mu_1-\mu_2}2-k_1+k_2-k_4}\end{array}\\\hline
\text{(Bb1)} & \begin{array}{l} \alpha = \rb{\mu_1-\mu_2-2k_1-2k_2-2k_3,\mu_1-2k_1-2k_4}\\ \psi = \rb{\mu_1-k_1,-\mu_2-k_2,\tfrac{\mu_1-\mu_2}2-k_1-k_2-k_3,\tfrac{\mu_1+\mu_2}2-k_1+k_2-k_4}\end{array}\\\hline
\text{(Bb2)} & \begin{array}{l} \alpha = \rb{-\mu_1-\mu_2-2k_1-2k_2-2k_3,-\mu_1-2k_1-2k_4}\\ \psi = \rb{-\mu_1-k_1,-\mu_2-k_2,\tfrac{-\mu_1-\mu_2}2-k_1-k_2-k_3,\tfrac{-\mu_1+\mu_2}2-k_1+k_2-k_4}\end{array}
\end{array}
\]

We now analyse the analytic behaviour of these series of residues, using the decay \eqref{eq:Mellin_decay} of $\hat{f}$, and prove the statements in the theorem. 

\subsubsection{Well-spaced parameters} Suppose further that the spectral parameters $\mu_1,\mu_2$ are well separated, in the sense that
\[
\vb{\mu_1}, \vb{\mu_2-\mu_1} \ge \varepsilon
\]
for some (small) constant $\varepsilon>0$. These assumptions ensure that the $\Gamma$-functions we analyse are bounded away from the poles, so we may use Stirling's approximation.

First we assume $X_1=X_2=1$, and verify Statement (A). Using the assumption that the spectral parameters are well-spaced, we deduce from \eqref{eq:Mellin_decay} that each series of residues \eqref{eq:8series} is bounded above by the term $(k_1,k_2,k_3,k_4) = (0,0,0,0)$ up to a constant factor depending only on $\varepsilon$. In other words, this says \eqref{eq:8series} is bounded above by
\begin{equation}\label{eq:fhat_Gamma}
\ll_\varepsilon \frac{\hat f(-2+it_1-\alpha'_1) \hat f(-\frac 32+it_2-\alpha'_2) \Gamma\rb{\pm\mu_1}\Gamma\rb{\pm\mu_2}\Gamma\rb{\pm\tfrac{\mu_1+\mu_2}2}\Gamma\rb{\pm\tfrac{\mu_1-\mu_2}2}}{\vb{\Gamma\rb{\frac{1+i\Im\mu_1+i\Im\mu_2}2}\Gamma\rb{\frac{1+2i\Im\mu_1}2}\Gamma\rb{\frac{1+2i\Im\mu_2}2}\Gamma\rb{\frac{1-i\Im\mu_1+i\Im\mu_2}2}}},
\end{equation}
with appropriate choices of signs, and $\alpha' = (\alpha'_1,\alpha'_2)$ is given as below:
\[
\begin{array}{c|cc}
\text{(Series)} & \alpha'_1 & \alpha'_2\\\hline
\text{(Aa1)} & \mu_1+\mu_2 & \mu_2\\
\text{(Aa2)} & -\mu_1+\mu_2 & \mu_2\\
\text{(Ab1)} & \mu_1-\mu_2 & -\mu_2\\
\text{(Ab2)} & -\mu_1-\mu_2 & -\mu_2\\
\end{array}
\quad
\begin{array}{c|cc}
\text{(Series)} & \alpha'_1 & \alpha'_2\\\hline
\text{(Ba1)} & \mu_1+\mu_2 & \mu_1\\
\text{(Ba2)} & -\mu_1+\mu_2 & -\mu_1\\
\text{(Bb1)} & \mu_1-\mu_2 & \mu_1\\
\text{(Bb2)} & -\mu_1-\mu_2 & -\mu_1\\
\end{array}
\]

Using \eqref{eq:ma}, and either \eqref{eq:mat} or \eqref{eq:man}, we apply Stirling's approximation and conclude that the $\Gamma$-quotients in \eqref{eq:fhat_Gamma} have size\newpage
\begin{align*}
&\asymp (1+\vb{\mu_1})^{\pm\Re(\mu_1)-\frac 12} (1+\vb{\mu_2})^{\pm\Re(\mu_2)-\frac 12} (1+\vb{\mu_1+\mu_2})^{\pm\Re(\frac{\mu_1+\mu_2}2)-\frac 12} (1+\vb{-\mu_1+\mu_2})^{\pm\Re(\frac{-\mu_1+\mu_2}2)-\frac 12}\\
&\asymp \mathcal C(\mu).
\end{align*}
Statement (A) then follows from the rapid decay of $\hat f$, using the observation that
\[
2\rb{\vb{\Im\alpha'_1-t_1}+\vb{\Im\alpha'_2-t_2}} \gg \vb{\Im\mu_1-\tau_1} + \vb{\Im\mu_2-\tau_2}.
\]

Next we verify Statement (B). Under the assumptions $\tau_1, \tau_2-\tau_1 \gg_A 1$ and $|\Im\mu_j - \tau_j| \le c$, we have
\begin{equation}\label{eq:spectral_boundaway}
    \vb{\tau_1-\Im\alpha'_1} + \vb{\tau_2-\Im\alpha'_2} \gg_A 1
\end{equation}
for all the series except (Aa1). It follows that the other $7$ series have negligible contributions. Under the same assumptions, we see that the sum over $k_i$ is dominated by the term $(k_1,k_2,k_3,k_4) = (0,0,0,0)$, and the total contribution of the terms $(k_1,k_2,k_3,k_4) \ne (0,0,0,0)$ has a smaller order of magnitude. Using the fact that $\hat f$ is non-vanishing in the strip $|\Im(s)|\le c$, we apply Stirling's approximation and conclude that the term $(k_1,k_2,k_3,k_4)$ has size $\asymp \mathcal C(\mu)$. This verifies Statement~(B).

Finally we verify Statement (C). Under the stated assumptions, \eqref{eq:spectral_boundaway} holds for all the series except~(Aa1) and (Ba1). It follows that the other $6$ series have negligible contributions. By similar arguments, the series of residues (Aa1) and (Ba1) are dominated by the term $(k_1,k_2,k_3,k_4) = (0,0,0,0)$. Because of the factor $X_1^{2-it_1+\alpha'_1} X_2^{3/2-it_2+\alpha'_2}$, we see that the residue from (Aa1) (resp.~(Ba1)) dominates when $\Re\mu_2 \ge \varepsilon$ (resp. $\Re\mu_2 \le -\varepsilon$), and the dominating term has size
\[
\asymp X_1^2 X_2^{3/2+|\Re\mu_2|} \mathcal C(\mu)
\]
by Stirling's approximation and the nonvanishing of $\hat f$ in the strip $|\Im(s)|\le c$. This verifies Statement (C).

\subsubsection{Closely spaced spectral parameters} Now we assume 
\[
\vb{\mu_1} < \varepsilon, \quad \text{ or } \quad \vb{\mu_2-\mu_1} < \varepsilon.
\]
Since $\vb{\mu_1}+\vb{\mu_2}$ is large, we see that the ``or'' above is exclusive. In this case it suffices to prove Statement (A); the other statements are void.

First suppose $0 < \vb{\mu_1} < \varepsilon$. In this case, we pair up the series
\begin{center}\begin{tabular}{cccc}
\{(Aa1), (Aa2)\}, & \{(Ab1), (Ab2)\}, & \{(Ba1), (Ba2)\}, & \{(Bb1), (Bb2)\}.
\end{tabular}\end{center}
Each of the series of residues in the pair can be obtained from the other term by flipping $\mu_1 \leftrightarrow -\mu_1$ in the $\Gamma$-factors. Since the sum $\Gamma(s)+\Gamma(-s)$ remains bounded as $s\to 0$, each pair's combined contribution can be bounded using Stirling's approximation, as in the well-spaced case.

Next suppose $0 < \vb{\mu_2-\mu_1}\ < \varepsilon$. In this case we pair up the series
\begin{center}\begin{tabular}{cccc}
\{(Aa1), (Ba1)\}, & \{(Aa2), (Ba1)\}, & \{(Ab1), (Ba2)\}, & \{(Ab2), (Bb2)\}.
\end{tabular}\end{center}
Each of the series of residues in the pair can be obtained from the other term by flipping $\mu_1 \leftrightarrow \mu_2$ in the $\Gamma$-factors. By the same argument as above, each pair's combined contribution can be bounded using Stirling's approximation, as in the well-spaced case.

\subsection{Degenerate spectral parameters} Now we assume that the spectral parameters are degenerate, that is,
\[
\mu_1 = 0, \quad \mu_2 = 0, \quad \text{ or } \quad \mu_1 = \pm\mu_2.
\]
In this case, the statements follow from the cases above, invoking the continuity of the Whittaker function $\widetilde W_\mu$. This finishes the proof of \Cref{thm:sl}. \qed

\section{Properties of the arithmetic transform}
\label{sec:arithmetic-transform}

This section is devoted to studying the integral transforms $\mathcal J_{w,F}$ appearing on the arithmetic side of the Kuznetsov trace formula \eqref{eq:ktf}.

\subsection{Trivial estimates}

In the rest of the section, we pick $f:\R_+ \to [0,1]$ to be a fixed smooth, nonzero, nonnegative function with support in [1,2]. Let $0\le \tau_1 \le \tau_2$, $1\le R_1\le R_2$ and $X_1,X_2\ge 1$ be parameters, 
\begin{equation}
\mathcal R := R_1R_2(R_1+R_2)(R_2-R_1+1),
\end{equation}
and fix
\begin{equation}\label{eq:ntf} 
F(y_1,y_2) = F_{\tau_1,\tau_2,X_1,X_2,R_1,R_2}(y_1,y_2) := \mathcal R^{1/2}f(X_1y_1)f(X_2y_2)y_1^{i(\tau_1+\tau_2)}y_2^{i\tau_2}.
\end{equation}
Compared to the test function \eqref{eq:tf_sl}, the function \eqref{eq:ntf} has an extra normalizing factor of $\mathcal R^{1/2}$. Through Statement (B) of \Cref{thm:sl}, this extra factor ensures that when $\tau_1\asymp R_1$ and $\tau_2\asymp R_2$, the spectral transform $|\langle W_\mu,F\rangle|$ has constant size when $\mu \approx (\tau_1,\tau_2)$. 
For the test function \eqref{eq:ntf}, it is straightforward to conclude that
\begin{equation}\label{eq:ntf_trivial}
F \ll \mathcal R^{1/2} \quad \text{ and } \quad \|F\|^2 \asymp X_1^4X_2^3\mathcal R,
\end{equation}
where the norm is taken with respect to the measure $d^\times y$.

\medskip 

We give two kinds of estimates:
\begin{enumerate}[label=(\roman*)]
\item $\mathcal J_{w,F}(A)$ vanishes when $A$ is small;
\item $\mathcal J_{w,F}(A)$ decays rapidly with respect to $\tau_1,\tau_2$.
\end{enumerate}

Recalling that $A$ contains the arithmetic parameters $(c, M, N)$, the estimate (i) allows us to truncate the sum over $c$ in terms of $M$ and $N$, while (ii) essentially allows us to truncate the sum over~$c$ in terms of $M, N$ \textit{and} the spectral parameter $\tau$, in particular the sum is shorter when the spectral parameters grow. 

\begin{prop}\label{prop:Jw_trivial}
    Let $F = F_{\tau_1,\tau_2,X_1,X_2,R_1,R_2}$ be as in \eqref{eq:ntf} and $C_1,C_2 \ge 0$, $\varepsilon > 0$. 
    \begin{enumerate}[label=(\roman*)]
    \item We have the following vanishing statements: 
    \begin{align}
        \mathcal J_{\alpha\beta\alpha,F}(A) &= 0 \quad \text{ if } A \le (2X_1X_2^{1/2})^{-1},\label{eq:Jaba_vanish}\\
        \mathcal J_{\beta\alpha\beta,F}(A) &= 0 \quad \text{ if } A \le (4X_1^2X_2^2)^{-1},\label{eq:Jbab_vanish}\\
        \mathcal J_{w_0,F}(A) &= 0 \quad \text{ if } A_1^2 A_2 \le (2X_1^2X_2)^{-1} \: \: \text{ or } \: \:  A_1A_2 \le (2X_1X_2)^{-1}.\label{eq:Jw0_vanish}
    \end{align}
    \item We have the following decay statements: 
    \begin{align}
        \mathcal J_{\alpha\beta\alpha,F}(A) &\ll_\varepsilon \mathcal R A^\varepsilon (X_1^4X_2^3)^{1+\varepsilon} \rb{\frac{1+AX_2^{1/2}}{\tau_2}}^{C_1},\label{eq:Jaba_trivial}\\
        \mathcal J_{\beta\alpha\beta, F}(A) &\ll_\varepsilon \mathcal R A^{-3/2+\varepsilon} (X_1^4X_2^3)^{1+\varepsilon} \rb{\frac{1+A^{1/2}X_1}{\tau_2}}^{C_1},\label{eq:Jbab_trivial}\\
        \mathcal J_{w_0,F}(A) &\ll_\varepsilon \mathcal R (A_1A_2)^\varepsilon (X_1^4X_2^3)^{1+\varepsilon} \rb{\frac{1+A_1^2A_2X_1X_2^2}{\tau_2}}^{C_1} \rb{\frac{1+A_1^2A_2^2X_1^2X_2}{\tau_2}}^{C_2}.\label{eq:Jw0_trivial}
    \end{align}
    \end{enumerate}
\end{prop}
\begin{proof}
    Recall that the transforms $\mathcal{J}_{w, F}$ are defined in \Cref{prop:kuznetsov_integral}. We address the three Weyl elements separately.

    \medskip 
    
    For $\mathcal J_{\alpha\beta\alpha,F}$, the support of $f$ restricts the support of the integral into the following compact region:
    \begin{align*}
        (AX_1)^{-1} &\le y_1 \le 2(AX_1)^{-1}, & X_2^{-1} &\le y_2 \le 2X_2^{-1},
    \end{align*}
    and
    \begin{align*}
        (A^2X_1^2X_2)^{-1} &\le \sqrt{\xi_{\alpha\beta\alpha,1}}/\xi_{\alpha\beta\alpha,2} \le 8(A^2X_1^2X_2)^{-1}, & \frac 12 &\le \xi_{\alpha\beta\alpha,2}/\xi_{\alpha\beta\alpha,1} \le 2,
    \end{align*}
    which implies
    \begin{align}\label{eq:xaba_condition}
        \frac{A^4X_1^4X_2^2}{256} &\le \xi_{\alpha\beta\alpha,1} \le 4A^4X_1^4X_2^2 & \frac{A^4X_1^4X_2^2}{128} &\le \xi_{\alpha\beta\alpha,2} \le 2A^4X_1^4X_2^2.
    \end{align}
    Since $\xi_{\alpha\beta\alpha,1},\xi_{\alpha\beta\alpha,2} \ge 1$, \eqref{eq:xaba_condition} cannot be satisfied when $A\le (2X_1X_2^{1/2})^{-1}$, proving \eqref{eq:Jaba_vanish}. Next we give a trivial bound for $\mathcal J_{\alpha\beta\alpha,F}(A)$ by bounding the size of the support of the integral. The condition $\xi_{\alpha\beta\alpha,1} \ll A^4X_1^4X_2^2$ says
    \[
    |x_1| \ll AX_1X_2^{1/2}, \quad |x_1x_4+x_2| \ll A^2X_1^2X_2.
    \]
    and $\xi_{\alpha\beta\alpha,2} \ll A^4X_1^4X_2^2$ says
    \[
    |x_1|, |x_2|, |x_4| \ll A^2X_1^2X_2,
    \]
    So the volume of the set of tuples $(x_1,x_2,x_4)$ satisfying these conditions has size $\ll (A^4X_1^4X_2^2)^{1+\varepsilon}$. The integrals over $y_1,y_2$ give factors of size $1$ and $X_2$ respectively. Using \eqref{eq:ntf_trivial}, we conclude that
    \[
    \mathcal J_{\alpha\beta\alpha,F}(A) \ll \mathcal R A^\varepsilon (X_1^4X_2^3)^{1+\varepsilon}.
    \]
    When $\tau_2$ is large, the bound can be improved by partial integration. Write $\mathcal J_{\alpha\beta\alpha,F}(A)$ in terms of the fixed test function $f$:
    \begin{multline*}
        \mathcal J_{\alpha\beta\alpha,F}(A) = \mathcal R A^{-4} \int_{\R_+^2} \int_{\R^3} e\rb{\frac{A\zeta_{\alpha\beta\alpha,1}}{\xi_{\alpha\beta\alpha,2}y_1y_2}} e\rb{\frac{\zeta_{\alpha\beta\alpha,2}y_2}{\xi_{\alpha\beta\alpha,1}}} e\rb{-Ax_1y_1} f\rb{\frac{X_1A\sqrt{\xi_{\alpha\beta\alpha,1}}}{\xi_{\alpha\beta\alpha,2}y_1y_2}}\\
        \times f\rb{\frac{X_2\xi_{\alpha\beta\alpha,2}y_2}{\xi_{\alpha\beta\alpha,1}}} \ol{f(X_1Ay_1)f(X_2y_2)} \xi_{\alpha\beta\alpha,1}^{i(\tau_1-\tau_2)/2} \xi_{\alpha\beta\alpha,2}^{-i\tau_1} y_1^{-2i(\tau_1+\tau_2)} y_2^{-i(\tau_1+\tau_2)} dx_1dx_2dx_4 \frac{dy_1dy_2}{y_1y_2^2}.
    \end{multline*}
    Then $C_1$ successive integrations by parts with respect to $y_1$ give an additional factor of
    \[
    \ll_{C_1} \rb{\rb{\frac{y_1}{\tau_1+\tau_2}} \rb{\frac{A|\zeta_{\alpha\beta\alpha,1}|}{\xi_{\alpha\beta\alpha,2}y_1^2y_2} + A|x_1| + y_1^{-1}}}^{C_1} \ll \rb{\frac{1+AX_2^{1/2}}{\tau_2}}^{C_1}
    \]
    in the support of the integral, proving \eqref{eq:Jaba_trivial}.

    \medskip 

    For $\mathcal J_{\beta\alpha\beta,F}$, the support of $f$ restricts the support of the integral into the following compact region:
    \begin{align*}
        X_1^{-1} &\le y_1 \le 2X_1^{-1}, & (AX_2)^{-1} &\le y_2 \le 2(AX_2)^{-1},
    \end{align*}
    and
    \begin{align*}
        \frac 12 &\le \sqrt{\xi_{\beta\alpha\beta,2}}/\xi_{\beta\alpha\beta,1} \le 2, & (AX_1^2X_2^2)^{-1} &\le \xi_{\beta\alpha\beta,1}/\xi_{\beta\alpha\beta,2} \le 16(AX_1^2X_2^2)^{-1},
    \end{align*}
    which implies
    \begin{align*}
        \frac{AX_1^2X_2^2}{64} &\le \xi_{\beta\alpha\beta,1} \le 4AX_1^2X_2^2, & \frac{A_2X_1^4X_2^4}{1024} \le \xi_{\beta\alpha\beta,2} \le 4A^2X_1^4X_2^4.
    \end{align*}
    Again, since $\xi_{\beta\alpha\beta,1},\xi_{\beta\alpha\beta,2}\ge 1$, this cannot be satisfied when $A\le (4X_1^2X_2^2)^{-1}$, proving \eqref{eq:Jbab_vanish}. Next we give a trivial bound for $\mathcal J_{\beta\alpha\beta,F}(A)$ by bounding the size of the support of the integral. The condition $\xi_{\beta\alpha\beta,1}\ll AX_1^2X_2^2$ says
    \[
    |x_4|, |x_5| \ll A^{1/2} X_1X_2,
    \]
    and $\xi_{\beta\alpha\beta,2} \ll A^2X_1^4X_2^4$ says
    \[
    |x_2| \ll AX_1^2X_2^2, \quad |x_4^2-x_2x_5| \ll AX_1^2X_2^2.
    \]
    So the volume of the set of tuples $(x_2,x_4,x_5)$ satisfying these conditions has size $(A^{3/2}X_1^3X_2^3)^{1+\varepsilon}$. The integrals over $y_1,y_2$ give factors of size $X_1$ and $1$ respectively. Using \eqref{eq:ntf_trivial}, we conclude that
    \[
    \mathcal J_{\beta\alpha\beta,F}(A) \ll \mathcal R A^{-3/2+\varepsilon} (X_1^4X_2^3)^{1+\varepsilon}.
    \]
    Now write $\mathcal J_{\beta\alpha\beta,F}(A)$ in terms of the fixed test function $f$:
    \begin{multline*}
        \mathcal J_{\beta\alpha\beta,F}(A) = \mathcal R A^{-3-i\tau_2} \int_{\R_+^2} \int_{\R^3} e\rb{\frac{\zeta_{\beta\alpha\beta,1}y_1}{\xi_{\beta\alpha\beta,1}}} e\rb{\frac{\zeta_{\beta\alpha\beta,2}}{\xi_{\beta\alpha\beta,2}y_1^2y_2}} e\rb{-Ax_5y_2} f\rb{\frac{X_1\sqrt{\xi_{\beta\alpha\beta,2}}y_1}{\xi_{\beta\alpha\beta,1}}}\\
        \times f\rb{\frac{X_2\xi_{\beta\alpha\beta,1}}{\xi_{\beta\alpha\beta,2}y_1^2y_2}} \ol{f(X_1y_1)f(X_2Ay_2)} \xi_{\beta\alpha\beta,1}^{-i\tau_1} \xi_{\beta\alpha\beta,2}^{i(\tau_1-\tau_2)/2} y_1^{-2i\tau_2} y_2^{-2i\tau_2} dx_2dx_4dx_5 \frac{dy_1dy_2}{y_1^2y_2}.
    \end{multline*}
    Then $C_1$ successive integrations by parts with respect to $y_2$ give an additional factor of
    \[
    \ll_{C_1} \rb{\rb{\frac{y_2}{\tau_2}}\rb{\frac{|\zeta_{\beta\alpha\beta,2}|}{\xi_{\beta\alpha\beta,2}^2y_1^2y_2^2}+A|x_5|+y_2^{-1}}}^{C_1} \ll \rb{\frac{1+A^{1/2}X_1}{\tau_2}}^{C_1}
    \]
    in the support of the integral, proving \eqref{eq:Jbab_trivial}.

    \medskip 

    Finally, for $\mathcal J_{w_0,F}$, the support of $f$ restricts the support of the integral into the following compact region:
    \begin{align*}
        (A_1X_1)^{-1} &\le y_1 \le 2(A_1X_1)^{-1}, & (A_2X_2)^{-1} &\le y_2 \le 2(A_2X_2)^{-1},
    \end{align*}
    and
    \begin{align*}
        (A_1^2X_1^2)^{-1} &\le \sqrt{\xi_{w_0,1}}/\xi_{w_0,2} \le 4(A_1^2X_1^2)^{-1}, & (A_2^2X_2^2)^{-1} &\le \xi_{w_0,2}/\xi_{w_0,1} \le 4(A_2^2X_2^2)^{-1},
    \end{align*}
    which implies
    \begin{align*}
        \frac{A_1^2A_2^4X_1^4X_2^4}{256} &\le \xi_{w_0,1} \le A_1^4A_2^4X_1^4X_2^4, & \frac{A_1^4A_2^2X_1^4X_2^2}{64} &\le \xi_{w_0,2} \le A_1^4A_2^2X_1^4X_2^2.
    \end{align*}
    Again, since $\xi_{w_0,1},\xi_{w_0,2}\ge 1$, this cannot be satisfied when $A_1^2A_2 \le (2X_1^2X_2)^{-1}$ or $A_1A_2 \le (2X_1X_2)^{-1}$, proving \eqref{eq:Jw0_vanish}.  Next we give a trivial bound for $\mathcal J_{w_0,F}(A)$ by bounding the size of the support of the integral. The condition $\xi_{w_0,1}\ll A_1^4A_2^4X_1^4X_2^4$ says
    \[
    |x_4|, |x_5|, |x_1x_4-x_2|, |x_4^2+x_1x_4x_5-x_2x_5| \ll A_1^2A_2^2X_1^2X_2^2,
    \]
    and $\xi_{w_0,2}\ll A_1^4A_2^4X_1^4X_2^2$ says
    \[
    |x_1|,|x_2|,|x_1x_5+x_4| \ll A_1^2A_2X_1^2X_2.
    \]
    To compute the volume of the support, we observe that for fixed $x_4$, the volume of the set of tuples $(x_1,x_2)$ satisfying $C\le x_1x_4-x_2\le C+dC$ is $\ll \min\{x_4^{-1},1\}A_1^2A_2X_1^2X_2 dC$. Meanwhile, for fixed $x_4$ and $C := x_1x_4-x_2$, the volume of the set of $x_5$ satisfying $|x_4^2+x_5C|\ll A_1^2A_2^2X_1^2X_2^2$ is~$\ll\min\{C^{-1},1\} A_1^2A_2^2X_1^2X_2^2$. So the volume of the set of tuples $(x_1,x_2,x_4,x_5)$ satisfying the conditions above is bounded by the integral
    \[
    \int_0^{A_1^2A_2^2X_1^2X_2^2} \int_0^{A_1^2A_2^2X_1^2X_2^2} A_1^4A_2^3X_1^4X_2^3 \min\{C^{-1},1\}\min\{x_4^{-1},1\} dC dx_4 \ll \rb{A_1^4A_2^3X_1^4X_2^3}^{1+\varepsilon}.
    \]
    Since the integrals over $y_1,y_2$ have constant size, we conclude using \eqref{eq:ntf_trivial} that
    \[
    \mathcal J_{w_0,F} \ll \mathcal R (A_1A_2)^{\varepsilon} (X_1^4X_2^3)^{1+\varepsilon}.
    \]
    Now write $\mathcal J_{w_0,F}(A_1,A_2)$ in terms of the fixed test function $f$:
    \begin{multline}\label{eq:Jw0_expand}
    \mathcal J_{w_0,F}(A) = \mathcal R A_1^{-4} A_2^{-3} \int_{\R_+^2} \int_{\R^4} e\rb{\frac{A_1\zeta_{w_0,1}}{\xi_{w_0,2}y_1}} e\rb{\frac{A_2\zeta_{w_0,2}}{\xi_{w_0,1}y_2}} e\rb{-A_1x_1y_1-A_2x_5y_2} f\rb{\frac{X_1A_1\sqrt{\xi_{w_0,1}}}{\xi_{w_0,2}y_1}}\\
    \times f\rb{\frac{X_2A_2\xi_{w_0,2}}{\xi_{w_0,1}y_2}} \ol{f(X_1A_1y_1)f(X_2A_2y_2)} \xi_{w_0,1}^{i(\tau_1-\tau_2)/2} \xi_{w_0,2}^{-i\tau_1} y_1^{-2i(\tau_1+\tau_2)} y_2^{-2i\tau_2} dx_1dx_2dx_4dx_5 \frac{dy_1dy_2}{y_1y_2}.
    \end{multline}
    Then, $C_1$ successive integrations by parts with respect to $y_1$ give an additional factor of
    \[
    \ll_{C_1} \rb{\rb{\frac{y_1}{\tau_1+\tau_2}}\rb{\frac{A_1|\zeta_{w_0,1}|}{\xi_{w_0,2}y_1^2}+A_1|x_1|+y_1^{-1}}}^{C_1} \ll \rb{\frac{1+A_1^2A_2X_1X_2^2}{\tau_2}}^{C_1},
    \]
    and $C_2$ successive integrations by parts with respect to $y_2$ give an additional factor of
    \[
    \ll_{C_2} \rb{\rb{\frac{y_2}{\tau_2}}\rb{\frac{A_2|\zeta_{w_0,2}|}{\xi_{w_0,1}y_2^2}+A_2|x_5|+y_2^{-1}}}^{C_2} \ll \rb{\frac{1+A_1^2A_2^2X_1^2X_2}{\tau_2}}^{C_2},
    \]
    where the last inequality follows by rewriting
    \[
    \zeta_{w_0,2} = (x_1x_5+x_4)^2(x_1x_4-x_2)+2(x_4^2+x_1x_4x_5-x_2x_5)x_2+x_1^2(x_1x_4-x_2)+x_2^2x_5+2x_1x_4-x_5,
    \]
    from which we deduce that $\zeta_{w_0,2} \ll A_1^6A_2^4X_1^6X_2^4$ using the bounds above. This proves \eqref{eq:Jw0_trivial}.
\end{proof}

\subsection{A nontrivial estimate for $\mathcal J_{w_0,F}$}

Here we prove under some extra assumptions (mostly that $\tau_1 \asymp \tau_2$) an improved estimate of the integral transform $\mathcal J_{w_0,F}$, which is useful in bounding the non-tempered spectrum.

\begin{prop}\label{prop:Jw0_nontriv}
Let $A_1,A_2\ge 0$, $X_1=1$, $X_2 = X\ge 1$, $R_1 = R_2 = R \asymp \tau_1 \asymp \tau_2$, and let $F = F_{\tau_1,\tau_2,X_1,X_2,R_1,R_2}$ be as in \eqref{eq:ntf}. Then we have
\begin{equation}\label{eq:Jw0_nontriv}
\mathcal J_{w_0,F}(A_1,A_2) \ll R^{7/3} (A_1A_2)^{\varepsilon} X^{3+\varepsilon} \rb{\frac{1+A_1^2A_2X^2}{\tau_2}}^{C_1} \rb{\frac{1+A_1^2A_2^2X}{\tau_2}}^{C_2}.
\end{equation}
\end{prop}
Compared to \eqref{eq:Jw0_trivial}, \Cref{prop:Jw0_nontriv} gives extra saving by a factor of $R^{2/3}$.
\begin{proof}
    The extra saving comes from a more careful estimate of the $y_1$ and $y_2$ integrals. We extract the $y_1$-integral from \eqref{eq:Jw0_expand}:
    \begin{multline*}
    \int_{\R_+} e\rb{\frac{A_1\zeta_{w_0,1}}{\xi_{w_0,2}y_1}} e\rb{-A_1x_1y_1} f\rb{\frac{X_1A_1\sqrt{\xi_{w_0,1}}}{\xi_{w_0,2}y_1}} \ol{f(X_1A_1y_1)} y_1^{-2i(\tau_1+\tau_2)} \frac{dy_1}{y_1}\\
     = \int_{\R^+} e\rb{-\frac{\tau_1+\tau_2}\pi \log y_1} e\rb{\frac{A_1\zeta_{w_0,1}}{\xi_{w_0,2}y_1}} e\rb{-A_1x_1y_1} f\rb{\frac{X_1A_1\sqrt{\xi_{w_0,1}}}{\xi_{w_0,2}y_1}} \ol{f(X_1A_1y_1)} \frac{dy_1}{y_1}.\label{eq:Jw0_sp} 
    \end{multline*}
    Define
    \[
    w(y_1) := \ol{f(y_1)}f\rb{\frac{X_1^2A_1^2\sqrt{\xi_{w_0,1}}}{\xi_{w_0,2}y_1}}. 
    \]
    Since $X_1^2A_1^2\sqrt{\xi_{w_0,1}}/\xi_{w_0,2}$ has bounded size in the support of the integral, we have $w^{(j)}(y_1)\ll 1$ uniformly in other variables for all $j\in\N_0$. Writing $T := \tau_1+\tau_2$ and through a change of variables $X_1A_1y_1\mapsto y_1$, we get an integral of the shape
    \begin{equation}\label{eq:sp_core} 
    \int_{\R_+} e\rb{T\rb{\kappa_1y_1-\frac{\log y_1}\pi + \frac{\kappa_2}{y_1}}} w(y_1) \frac{dy_1}{y_1}.
    \end{equation}
    This is a standard stationary phase integral, with phase function
    \[
    g(y_1) = \kappa_1y_1-\frac{\log y_1}\pi + \frac{\kappa_2}{y_1}.
    \]
    Its first few derivatives are:
    \begin{align*}
        g'(y_1) &= \kappa_1-\frac{1}{\pi y_1}-\frac{\kappa_2}{y_1^2}, & g''(y_1) &= \frac{1}{\pi y_1^2}+\frac{2\kappa_2}{y_1^3}, & g'''(y_1) &= -\frac{2}{\pi y_1^3}-\frac{6\kappa_2}{y_1^4}.
    \end{align*}
    Now we split into cases, depending on the relative size of $\kappa_1,\kappa_2$.
    \begin{enumerate}
    \item Suppose $\kappa_1, \kappa_2 = o(1)$. Then the oscillation is dominated by the $-\log y_1/\pi$ term over the support of the integral. In particular, there are always oscillations and the integral is vanishingly small.
    \item Suppose $\kappa_1 = o(1)$, and $\kappa_2 \asymp 1$. Then we have a critical point around $y_1 \approx -\pi \kappa_2$. At $y_1 = -\pi \kappa_2$, the second derivative is given by $-1/\pi^3\kappa_2^2$, which is bounded away from zero.
    \item Suppose $\kappa_1 \asymp 1$, and $\kappa_2 = o(1)$. Then we have a critical point around $y_1 \approx 1/\pi\kappa_1$. At $y_1 = 1/\pi\kappa_1$, the second derivative is given by $\pi\kappa_1^2$, which is bounded away from zero.
    \item Suppose $\kappa_1 \gg 1$ is large, and $\kappa_2 \ll 1$. Then the oscillation is dominated by the $\kappa_1 y_1$ term over the support of the integral.
    \item Suppose $\kappa_1 \ll 1$, and $\kappa_2 \gg 1$ is large. Then the oscillation is dominated by the $\kappa_2/y_1$ term over the support of the integral.
    \item Suppose $\kappa_1 \asymp \kappa_2 \gg 1$ are large. Then we have a critical point around $y_1 \approx \pm\sqrt{|\kappa_2/\kappa_1|}$. At $y_1 = \pm\sqrt{|\kappa_2/\kappa_1|}$, the second derivative is given by $2|\kappa_1|^{3/2}/|\kappa_2|^{1/2}$, which is bounded away from zero.
    \item Finally, suppose $\kappa_1 \asymp \kappa_2 \asymp 1$. Then the critical points are at 
    \[
    y_1 = \frac{\pi^{-1} \pm \sqrt{\pi^{-2} + 4\kappa_1\kappa_2}}{2\kappa_1}. 
    \]
    At the critical points, the second derivative vanishes when
    \[
    \frac{1}{\pi y_1^2} + \frac{2\kappa_2}{y_1^3} = 0 \implies \frac{y_1}{\pi} + 2\kappa_2 = 0,
    \]
    which happens when $\kappa_1 \kappa_2 = -1/4\pi^2$. In this case, the third derivative is given by $8\pi^2\kappa_1^3$, which is bounded away from zero.
    \end{enumerate}
    In conclusion, in the worst case the phase function at critical points has vanishing second derivative and non-vanishing third derivative. The standard stationary phase argument thus gives a saving of $T^{1/3}$ compared to the trivial bound. 
    
    The stationary phase argument for the $y_2$-integral is completely similar. We extract the $y_2$-integral from \eqref{eq:Jw0_expand}:
    \[
    \int_0^\infty e\rb{\frac{A_2\zeta_{w_0,2}}{\xi_{w_0,1}y_2}} e(-A_2x_5y_2) f\rb{\frac{X_2A_2\xi_{w_0,2}}{\xi_{w_0,1}y_2}} \ol{f(X_2A_2y_2)} y_2^{-2i\tau_2} \frac{dy_2}{y_2}.
    \]
    Define
    \[
    w(y_2) := \ol{f(y_2)} f\rb{\frac{X_2^2A_2^2\xi_{w_0,2}}{\xi_{w_0,1}y_2}},
    \]
    and again we have $w^{(j)}(y_2) \ll 1$ for all $j\in\N_0$. Writing $T:=\tau_2$ and through a change of variables $X_2A_2y_2\mapsto y_2$, we obtain an integral of the same shape as \eqref{eq:sp_core}. By the same argument, we get a saving of $T^{1/3}$ compared to the trivial bound as well. Combining the savings from the $y_1$ and $y_2$ integrals yields the result.
\end{proof}

\subsection{Integrating over spectral parameters}

By \Cref{thm:sl}, the use of the test function \eqref{eq:ntf} essentially picks a narrow part of the spectrum with $\mu_1\approx \tau_1$, $\mu_2\approx \tau_2$. To study a larger part of the spectrum, such as $\mu_1 \asymp T_1$, $\mu_2 \asymp T_2$, we integrate the Kuznetsov formula \eqref{eq:ktf} over $\tau_1,\tau_2$. A convenient way to do this is to pick a smooth, non-negative function $g:\R_+\to\R$ with compact support, and consider the integral over $\tau_1,\tau_2$ weighted by $g$ as follows:
\[
\int_{\R_+^2} g\rb{\frac{\tau_1}{T_1}} g\rb{\frac{\tau_2}{T_2}} \cdots d\tau_1 d\tau_2.
\]
Here we give estimates to the arithmetic side integrated over $\tau_1,\tau_2$.

\begin{prop}\label{prop:Jw_ls} 
Fix a smooth function $g:\R_+\to\R$ with compact support, and let $1\ll R_1 \le R_2$ be sufficiently large parameters, $C_1,C_2\ge 0$ and $\varepsilon > 0$. Then we have
\begin{align}
    \int_{\R^2} g\rb{\frac{\tau_1}{R_1}} g\rb{\frac{\tau_2}{R_2}} \mathcal J_{\alpha\beta\alpha,F}(A) d\tau_1 d\tau_2 &\ll \mathcal R (AR_1R_2)^\varepsilon (X_1^4X_2^3)^{1+\varepsilon} \rb{\frac{1+AX_2^{1/2}}{R_2}}^{C_1},\label{eq:Jaba_ls}\\
    \int_{\R^2} g\rb{\frac{\tau_1}{R_1}} g\rb{\frac{\tau_2}{R_2}} \mathcal J_{\beta\alpha\beta,F}(A) d\tau_1 d\tau_2 &\ll \mathcal R A^{-3/2+\varepsilon} (R_1R_2)^\varepsilon (X_1^4X_2^3)^{1+\varepsilon} \rb{\frac{1+A^{1/2}X_1}{R_2}}^{C_2},\label{eq:Jbab_ls}\\
    \int_{\R^2} g\rb{\frac{\tau_1}{R_1}} g\rb{\frac{\tau_2}{R_2}} \mathcal J_{w_0,F}(A_1,A_2) d\tau_1 d\tau_2 &\ll \mathcal R (A_1A_2R_1R_2)^\varepsilon (X_1^4X_2^3)^{1+\varepsilon} \nonumber\\
    &\hspace{1cm}\times \rb{\frac{1+A_1^2A_2X_1X_2^2}{R_2}}^{C_1} \rb{\frac{1+A_1^2A_2^2X_1^2X_2}{R_2}}^{C_2}.\label{eq:Jw0_ls}
\end{align}
\end{prop}
\Cref{prop:Jw_ls} shows that integration over $\tau_1,\tau_2$ can be done at almost no cost. Essentially, this gives a saving of $(R_1R_2)^{1-\varepsilon}$ over the trivial estimate in \Cref{prop:Jw_trivial}.

To prove \Cref{prop:Jw_ls}, we need a technical lemma, which can be viewed as an instance of a generalised version of Weyl's tube formula.

\begin{lem}\label{lem:shell_volume}
    Let $w \in \cb{\alpha\beta\alpha,\beta\alpha\beta,w_0}$. Let $Z_1,Z_2,R_1,R_2\gg 1$ be sufficiently large parameters, satisfying
    \begin{align}\label{eq:shell_Zassumption}
        Z_1 &\asymp Z_2 \quad \text{ if } w=\alpha\beta\alpha, & Z_1^2&\asymp Z_2 \quad \text{ if } w=\beta\alpha\beta.
    \end{align}
    Let $\xi_{w,1}$ and $\xi_{w,2}$ be defined as in \Cref{prop:Fourier_integral}. Consider the subset $\mathcal X_w = \mathcal X_w(Z_1,Z_2,R_1,R_2)$ of elements~$x\in U_w(\R)$ satisfying
    \begin{align}\label{eq:shell_condition}
    \xi_{w,1} &= Z_1 \rb{1+O(R_1^{-1})}, & \xi_{w,2} &= Z_2 \rb{1+O(R_2^{-1})}.
    \end{align}
    Then for any $\varepsilon>0$ we have
    \[
    \vol(\mathcal X_w) \ll \frac{(Z_1Z_2)^{1/2}}{R_1R_2} (Z_1Z_2R_1R_2)^\varepsilon.
    \]
\end{lem}
We defer the proof of \Cref{lem:shell_volume} to \Cref{subsec:shell_volume_proof}.
\begin{proof}[Proof of \Cref{prop:Jw_ls}]
We only give a proof for \eqref{eq:Jaba_ls}, as the proofs for \eqref{eq:Jbab_ls} and \eqref{eq:Jw0_ls} are completely analogous. If $1+AX_2^{1/2} \le R_2^{1-\delta}$ for some $\delta>0$, then we may use $C_1 \mapsto C_1+2/\delta$ together with \eqref{eq:Jaba_trivial}, saving a factor of
\[
\rb{\frac{1+AX_2^{1/2}}{R_2}}^{2/\delta} \ge R_1R_2.
\]
So we may assume $1+AX_2^{1/2} \ge R_2^{1-\delta}$. Let $Z:= R_1R_2X_1X_2(1+A)$. We extract the $\tau_1,\tau_2$-integral from \eqref{eq:Jaba_ls}:
\begin{multline*}
    \int_{\R^2} g\rb{\frac{\tau_1}{R_1}} g\rb{\frac{\tau_2}{R_2}} \xi_{\alpha\beta\alpha,1}^{i(\tau_1-\tau_2)/2} \xi_{\alpha\beta\alpha,2}^{-i\tau_1} y_1^{-2i(\tau_1+\tau_2)} y_2^{-i(\tau_1+\tau_2)} d\tau_1 d\tau_2\\
    = R_1R_2 \widehat g\rb{\frac{R_1}{2\pi}\log\frac{\xi_{\alpha\beta\alpha,2}y_1^2y_2}{\sqrt{\xi_{\alpha\beta\alpha,1}}}} \widehat g\rb{\frac{R_2}{2\pi}\log \sqrt{\xi_{\alpha\beta\alpha,1}}y_1^2y_2},
\end{multline*}
where $\widehat g$ denotes the Fourier transform of $g$. Since $g$ is smooth, $\widehat g$ has rapid decay. Hence, up to a negligible error of $Z^{-C}$, we may restrict $\xi_{\alpha\beta\alpha,1}, \xi_{\alpha\beta\alpha,2}$ to the range
\begin{align}\label{eq:Jaba_shell}
    \xi_{\alpha\beta\alpha,2}/\sqrt{\xi_{\alpha\beta\alpha,1}} &= y_1^{-2}y_2^{-1} \rb{1+O(R_1^{-1}Z^\varepsilon)}, & \sqrt{\xi_{\alpha\beta\alpha,1}} &= y_1^{-2}y_2^{-1} \rb{1+O(R_2^{-1}Z^\varepsilon)}.
\end{align}
From \Cref{lem:shell_volume} the volume of the set satisfying \eqref{eq:Jaba_shell} is bounded by
\[
\ll A^4X_1^2X_2^2R_1^{-1}R_2^{-1}Z^\varepsilon,
\]
providing saving of a factor $(R_1R_2)^{1-\varepsilon}$. Combining this with the estimate \eqref{eq:Jaba_trivial} yields \eqref{eq:Jaba_ls}.
\end{proof}

\begin{prop}\label{prop:arith_integrated}
Fix a smooth function $g:\R_+\to\R$ with compact support. Let $X_1=X_2=1$, and let $1\ll R_1 \le R_2$ be sufficiently large parameters. Let $M,N\in\N^2$, and for $i\in\{1,2\}$ we write~$d_i = (m_i,n_i)$, and $m_i=d_im'_i$, $n_i=d_in'_i$. Then we have
\begin{multline}\label{eq:arith_aba}
    \sum_{\substack{c_2\mid c_1^2\\ m_2c_1^2=n_2c_2^2}} \frac{|\Kl_{\alpha\beta\alpha}(c,M,N)|}{c_1c_2} \vb{\int_{\R_+^2} g\rb{\frac{\tau_1}{R_1}} g\rb{\frac{\tau_2}{R_2}} \mathcal J_{\alpha\beta\alpha,F}\rb{\sqrt{\frac{m_1m_2n_1}{c_2}}} d\tau_1 d\tau_2}\\
    \ll \mathcal R^{1+\varepsilon} R_2^{-4/3+\varepsilon} d_1^{7/3+\varepsilon} d_2^{2/3+\varepsilon} {m'_1}^{2/3+\varepsilon} {n'_1}^{2/3+\varepsilon} {m'_2}^{1/2+\varepsilon} {n'_2}^{-1/3+\varepsilon}, 
\end{multline}
\begin{multline}\label{eq:arith_bab}
    \sum_{\substack{c_1\mid c_2\\ m_1c_2=n_1c_1^2}} \frac{|\Kl_{\beta\alpha\beta}(c,M,N)|}{c_1c_2} \vb{\int_{\R_+^2} g\rb{\frac{\tau_1}{R_1}} g\rb{\frac{\tau_2}{R_2}} \mathcal J_{\beta\alpha\beta,F} \rb{\frac{m_1\sqrt{m_2n_2}}{c_1}} d\tau_1 d\tau_2}\\
    \ll \mathcal R^{1+\varepsilon} R_2^{-2+\varepsilon} d_1^{1/2+\varepsilon} d_2^{3/2+\varepsilon} {m'_1}^\varepsilon {n'_1}^{-1/2+\varepsilon} {m'_2}^{1/4+\varepsilon} {n'_2}^{1/4+\varepsilon},
\end{multline}
\begin{multline}\label{eq:arith_w0}
    \sum_{c_1,c_2} \frac{|\Kl_{w_0}(c,M,N)|}{c_1c_2} \vb{\int_{\R_+^2} g\rb{\frac{\tau_1}{R_1}} g\rb{\frac{\tau_2}{R_2}} \mathcal J_{w_0,F} \rb{\frac{\sqrt{m_1n_1c_2}}{c_1}, \frac{\sqrt{m_2n_2}c_1}{c_2}} d\tau_1 d\tau_2}\\
    \ll \mathcal R^{1+\varepsilon} R_2^{-5/4+\varepsilon} (m_1n_1)^{5/4+\varepsilon} (m_2n_2)^{1+\varepsilon}.
\end{multline}
\end{prop}
{\begin{rk}
    It is possible to slightly improve the bounds in \Cref{prop:arith_integrated}, by taking into account common factors of $m_1n_1$ and $m_2n_2$. But we do not pursue this here, so that the resulting expressions remain relatively simple.
\end{rk}}
\begin{proof}
    First we prove \eqref{eq:arith_aba}. By \eqref{eq:Jaba_ls}, the integral in \eqref{eq:arith_aba} has size $\mathcal R^{1+\varepsilon}$, and the $c$-sum is truncated to $c_2 \ll m_1m_2n_1/R_2^2$. We may assume $m_2/n_2$ is a rational square, otherwise the sum vanishes. The conditions on the $c$-sum imply $m_2 \mid n_2c_2$. So we may write $m_2 = d_2m'_2 = d_2{m''_2}^2$, $n_2 = d_2n'_2 = d_2{n''_2}^2$, $c_2 = c'_2{m''_2}^2$, $c_1 = c'_2m''_2n''_2$, $d'_2 = (d_2,c'_2)$, and $c'_2 = d'_2 c''_2$. Next we let $d'_1 = (m_1,n_1,c_1)$. Using \eqref{eq:Kaba_bound}, the left hand side of \eqref{eq:arith_aba} is bounded by
    \begin{align*}
        &\ll \mathcal R^{1+\varepsilon} \sum_{\substack{c_2\ll m_1m_2n_1/R_2^2\\ c_2\mid c_1^2,\ m_2c_1^2=n_2c_2^2}} (m_1,n_1,c_1)(m_2,c_2)(c_1,c_2)(c_1c_2)^{-2/3+\varepsilon}\\
        &\ll \sum_{d'_2 \mid d_2} \sum_{c''_2\ll m_1n_1d_2/d'_2R_2^2} d'_1 {d'_2}^{2/3+\varepsilon} {c''_2}^{-1/3+\varepsilon} {m''_2}^{1+\varepsilon} {n''_2}^{-2/3+\varepsilon}\\
        &\ll \mathcal R^{1+\varepsilon} R_2^{-4/3+\varepsilon} d_1^{7/3+\varepsilon} d_2^{2/3+\varepsilon} {m'_1}^{2/3+\varepsilon} {n'_1}^{2/3+\varepsilon} {m'_2}^{1/2+\varepsilon} {n'_2}^{-1/3+\varepsilon}.
    \end{align*}

    Next we prove \eqref{eq:arith_bab}. By \eqref{eq:Jbab_ls}, the integral in \eqref{eq:arith_bab} has size
    \[
    \ll \mathcal R^{1+\varepsilon} \rb{\frac{m_1\sqrt{m_2n_2}}{c_1}}^{-3/2+\varepsilon},
    \]
    and the $c$-sum is truncated to $c_1\ll m_1\sqrt{m_2n_2}/R_2$. The conditions on the $c$-sum imply $m'_1 \mid c_1$. Write $c_1 = m'_1c'_1$. Then we have $c_2 = m'_1n'_1{c'_1}^2$. Write $d'_1 = (d_1,c'_1)$, $c'_1 = d'_1c''_1$. Next we let $d'_2 = (m_2,n_2,c_2)$. Using \eqref{eq:Kbab_bound}, the left hand side of \eqref{eq:arith_bab} is bounded by
    \begin{align*}
        &\ll \mathcal R^{1+\varepsilon} \sum_{\substack{c_1\ll m_1\sqrt{m_2n_2}/R_2\\ c_1\mid c_2,\ m_1c_2=n_1c_1^2}} (m_1,c_1)(m_2,n_2,c_2)(c_1^2,c_2)c_1^{-3/2+\varepsilon}c_2^{-1/2+\varepsilon} \rb{\frac{m_1\sqrt{m_2n_2}}{c_1}}^{-3/2+\varepsilon}\\
        &\ll \mathcal R^{1+\varepsilon} \sum_{d'_1 \mid d_1} \sum_{c''_1\ll d_1\sqrt{m_2n_2}/d'_1R_2} d_1^{-3/2+\varepsilon} {d'_1}^{2+\varepsilon} d'_2 {c''_1}^{1+\varepsilon} {m'_1}^\varepsilon {n'_1}^{-1/2+\varepsilon} (m_2n_2)^{-3/4+\varepsilon}\\
        &\ll \mathcal R^{1+\varepsilon} R_2^{-2+\varepsilon} d_1^{1/2+\varepsilon} d_2^{3/2+\varepsilon} {m'_1}^\varepsilon {n'_1}^{-1/2+\varepsilon} {m'_2}^{1/4+\varepsilon} {n'_2}^{1/4+\varepsilon}.
    \end{align*}

    Finally we prove \eqref{eq:arith_w0}. By \eqref{eq:Jw0_ls}, the integral in \eqref{eq:arith_w0} has size $\mathcal R^{1+\varepsilon}$, and the $c$-sum is truncated to $c_1\ll m_1n_1\sqrt{m_2n_2}/R_2=:Z_1$, $c_2\ll m_1n_2m_2n_2/R_2=:Z_2$. Write $D=(c_1,c_2)$, $c_1=Dc'_1$, $c_2=Dc'_2$. Write also $d=(m_1m_2,n_1n_2,c_1c_2)$, so that $d\mid (m_1m_2,n_1n_2)$ and $d\mid c_1c_2 = D^2c'_1c'_2$. Using \eqref{eq:Kw0_bound}, the left hand side of \eqref{eq:arith_w0} is bounded by
    \begin{align*}
        &\ll \mathcal R^{1+\varepsilon} \sum_{c_i\ll Z_i} (m_1m_2,n_1n_2,c_1c_2)^{1/2} (c_1,c_2)^{1/2} c_1^{-1/2+\varepsilon} c_2^{-1/4+\varepsilon}\\
        &\ll \mathcal R^{1+\varepsilon} \sum_{D\ll Z_1} D^{-1/4+\varepsilon} \sum_{c'_i \ll Z_i/D} \sum_{\substack{d\mid(m_1m_2,n_1n_2)\\ d\mid D^2c'_1c'_2}} d^{1/2}{c'_1}^{-1/2+\varepsilon} {c'_2}^{-1/4+\varepsilon}\\
        &\ll \mathcal R^{1+\varepsilon} d^\varepsilon Z_1^{1/2+\varepsilon} Z_2^{1/2+\varepsilon}\\
        &\ll \mathcal R^{1+\varepsilon} R_2^{-5/4+\varepsilon} (m_1n_1)^{5/4+\varepsilon} (m_2n_2)^{1+\varepsilon}. \qedhere
    \end{align*}
\end{proof}

\subsection{Proof of \Cref{lem:shell_volume}}\label{subsec:shell_volume_proof}

Now we prove \Cref{lem:shell_volume}. The proof is very elementary, and only involves dividing the set $\mathcal X_w$ into pieces for which the volume can be conveniently estimated. However, we do not have a unified argument which establishes the lemma for all Weyl elements~$w$, and each case needs to be treated separately. A potential unified argument would likely entail a better understanding of the shape of the polynomials $\xi_{w,i}$ arising from explicit Iwasawa decomposition.

Throughout the proof, we often make use of the following inequality, which is easy to verify:
\begin{equation}\label{eq:root_formula}
\sqrt{A}-\sqrt{B} \le \frac{A-B}{\sqrt{A}} \quad \text{ for } A\ge B\ge 0.
\end{equation}

\subsubsection{Proof for $w=\alpha\beta\alpha$}
For $w=\alpha\beta\alpha$, the tube domain \eqref{eq:shell_condition} reads
\begin{align*}
    Z_1(1-c_1/R_1) &\le (x_1+1)^2+(x_1x_4+x_2)^2 \le Z_1(1+c_1/R_1),\\
    Z_2(1-c_2/R_2) &\le x_1^2+x_2^2+x_4^2+1 \le Z_2^2(1+c_2/R_2).
\end{align*}
Using an orthogonal change of variables
\[
\bp x'_2\\x'_4\ep = (x_1^2+1)^{-1/2} \bp 1&x_1\\-x_1&1\ep \bp x_2\\x_4\ep,
\]
we may rewrite the conditions above as follows:
\begin{align}
    Z_1(1-c_1/R_1) &\le (x_1^2+1)^2 + (x_1^2+1){x'_2}^2 \le Z_1(1+c_1/R_1),\label{eq:shell_aba1}\\
    Z_2(1-c_2/R_2) &\le x_1^2+{x'_2}^2+{x'_4}^2+1 \le Z_2^2(1+c_2/R_2).\label{eq:shell_aba2}
\end{align}
From \eqref{eq:shell_aba1} we get
\begin{align}
    (x_1^2+1)^2 & \le Z_1(1+c_1/R_1),\label{eq:shell_aba3}\\
    \frac{Z_1(1-c_1/R_1)}{x_1^2+1}-x_1^2-1 & \le {x'_2}^2 \le \frac{Z_1(1+c_1/R_1)}{x_1^2+1}-x_1^2-1.\label{eq:shell_aba4}
\end{align}

Suppose $x_1\ll 1$. By \eqref{eq:root_formula}, the volume of $x'_2$ verifying \eqref{eq:shell_aba4} is $\le 2Z_1^{1/2}c_1/R_1$, and the volume of~$x'_4$ verifying \eqref{eq:shell_aba2} is bounded by
\[
\le \min\cb{\frac{2Z_2c_2}{R_2V(x_1,x'_2)^{1/2}}, V(x_1,x'_2)^{1/2}},
\]
where $V(x_1,x'_2) := Z_2(1+c_2/R_2)-{x'_2}^2-x_1^2-1$. Note that we may assume $V(x_1,x'_2)\ge 0$, otherwise~\eqref{eq:shell_aba2} is not solvable.
\begin{enumerate}[label=(\roman*)]
    \item Suppose $R_1\ge R_2$. First assume $V(x_1,x'_2) \ll 2Z_2c_2/R_2$. Using \eqref{eq:shell_aba4}, we find
    \[
    0 \le V(x_1,x'_2) \le Z_2(1+c_2/R_2)-\frac{Z_1(1-c_1/R_1)}{x_1^2+1} \le \frac{CZ_2}{R_2}
    \]
    for some constant $C>0$, noting that the use of \eqref{eq:shell_aba4} overestimates $V(x_1,x'_2)$ by at most $2Z_1c_1/R_1 \ll Z_2/R_2$, invoking the assumption $Z_1\asymp Z_2$. It follows that
    \begin{equation}\label{eq:shell_aba5}
        \frac{Z_1(1-c_1/R_1)}{Z_2(1+c_2/R_2)}\le x_1^2+1 \le \frac{Z_1(1-c_1/R_1)}{Z_2(1+(c_2-C)/R_2)}.
    \end{equation}
    In this case, the volume of $x'_4$ is $\ll V(x_1,x'_2)^{1/2} \ll (Z_2/R_2)^{1/2}$, and the volume of $x_1$ verifying \eqref{eq:shell_aba5} is $\ll 1/R_2^{1/2}$. So the contribution to $\vol(\mathcal X_{\alpha\beta\alpha})$ from this case is $\ll (Z_1Z_2)^{1/2}/R_1R_2$.

    Next we assume $V(x_1,x'_2) \gg 2Z_2c_2/R_2$ for some sufficiently large constant, such that
    \[
    V(x_1,x'_2) \ge Z_2(1+c_2/R_2) - \frac{Z_1(1+c_1/R_1)}{x_1^2+1} \ge \frac{CZ_2}{R_2}
    \]
    for some constant $C>0$. Let $U_1 \le 2/3$ be a parameter, and consider dyadic intervals
    \begin{equation}\label{eq:shell_aba6}
        Z_2U_1 \le Z_2(1+c_2/R_2) - \frac{Z_1(1+c_1/R_1)}{x_1^2+1} \le 2Z_2U_1.
    \end{equation}
    In this case, the volume of $x'_4$ is $\ll Z_2^{1/2}/R_2U_1^{1/2}$, and the volume of $x_1$ verifying \eqref{eq:shell_aba6} is~$\ll U_1$. Summing over the dyadic intervals, the contribution to $\vol(\mathcal X_{\alpha\beta\alpha})$ from this case is $\ll (Z_1Z_2)^{1/2}/R_1R_2$.
    \item Suppose $R_2\ge R_1$. The same argument applies; the only difference being that the cutoff is taken to be $2Z_1c_1/R_1$ instead. The cases $V(x_1,x'_2) \ll 2Z_1c_1/R_1$ and $V(x_1,x'_2) \gg 2Z_1c_1/R_1$ contribute $\ll (Z_1Z_2)^{1/2}/R_1R_2$ to $\vol(\mathcal X_{\alpha\beta\alpha})$.
\end{enumerate}
It remains to consider the case where $x_1\gg 1$ is sufficiently large such that
\begin{equation}\label{eq:shell_aba7}
\frac{Z_1(1+c_1/R_1)}{x_1^2+1} \ll \frac{Z_2(1-c_2/R_2)}{2}.
\end{equation}
Using \eqref{eq:root_formula} and \eqref{eq:shell_aba4}, for fixed $x_1$ and $x'_2$ verifying \eqref{eq:shell_aba7} the volume of the elements $x'_4$ verifying~\eqref{eq:shell_aba2} is~$\ll c_2Z_2^{1/2}/R_2$. Let $U_2\le 2/3$ be a parameter, and consider the dyadic intervals
\begin{equation}\label{eq:shell_aba8}
    Z_1U_2 \le Z_1(1+c_1/R_1) - (x_1^2+1)^2 \le 2Z_1U_2.
\end{equation}
For $x_1$ verifying \eqref{eq:shell_aba8}, the volume of $x'_2$ verifying \eqref{eq:shell_aba4} is bounded by
\[
\ll \min \cb{\frac{2c_1Z_1^{1/2}}{R_1(U_2(x_1^2+1))^{1/2}}, \frac{2(Z_1U_2)^{1/2}}{(x_1^2+1)^{1/2}}}.
\]
First suppose $U_2\ge 1/100$. Summing over the dyadic intervals with $1/100 \le U_2 \ll 1$, it follows that the volume of $(x_1,x'_2,x'_4)$ satisfying the aforementioned conditions is bounded above by
    \[
    \frac{(Z_1Z_2)^{1/2}}{R_1R_2} \int_{x_1 \text{ verifying \eqref{eq:shell_aba1}}} \frac{dx_1}{(x_1^2+1)^{1/2}} \ll \frac{(Z_1Z_2)^{1/2}}{R_1R_2} \log Z_1.
    \]
    
When $U_2\le 1/100$, we have $x_1^2+1 \asymp Z_1^{1/2}$, and the volume of $x_1$ verifying \eqref{eq:shell_aba8} is $\ll Z_1^{1/4}U_2$. Summing over the dyadic intervals with $U_2\le 1/100$, it follows that the volume of $(x_1,x'_2,x'_4)$ satisfying the aforementioned conditions is $\ll (Z_1Z_2)^{1/2}/R_1R_2$. This finishes the proof for the Weyl element $w=\alpha\beta\alpha$.

\subsubsection{Proof for $w=\beta\alpha\beta$} For $w=\beta\alpha\beta$, \eqref{eq:shell_condition} reads, after a slight rewriting,
\begin{align}
    Z_1(1-c_1/R_1) &\le 1+x_4^2+x_5^2 \le Z_1(1+c_1/R_1),\label{eq:shell_bab1}\\
    \frac{Z_2(1-c_2/R_2)}{x_5^2+1} &\le \rb{x_2-\frac{x_4^2x_5}{x_5^2+1}}^2 + \rb{1+\frac{x_4^2}{x_5^2+1}}^2 \le \frac{Z_2(1+c_2/R_2)}{x_5^2+1}.\label{eq:shell_bab2}
\end{align}
By \eqref{eq:root_formula}, the volume of $x_2$ verifying \eqref{eq:shell_bab2} is bounded by
\[
\ll \min\cb{\frac{2c_2Z_2}{R_2(x_5^2+1)V(x_4,x_5)^{1/2}}, V(x_4,x_5)^{1/2}},
\]
where
\[
V(x_4,x_5) := \frac{Z_2(1+c_2/R_2)}{x_5^2+1} - \rb{1+\frac{x_4^2}{x_5^2+1}}^2.
\]

Suppose $x_5\gg 1$ for a sufficiently large constant (depending on $Z_2/Z_1^2, c_1,c_2$). If $x_5^2+1 \ge Z_1(1-c_1/R_1)/100$, then $1+x_4^2/(x_5^2+1)\ll 1$. In this case, the volume of $x_2$ verifying~\eqref{eq:shell_bab2} is~$\ll 2c_2Z_2^{1/2}/Z_1^{1/2}R_2$, and the volume of $(x_4,x_5)$ verifying \eqref{eq:shell_bab1} is $\ll 2c_1Z_1/R_1$. So the contribution to $\vol(\mathcal X_{\beta\alpha\beta})$ from this case is $\ll (Z_1Z_2)^{1/2}/(R_1R_2)$.

On the other hand, if $x_5^2+1 \le Z_1(1-c_1/R_1)/100$, then $x_4^2 \asymp Z_1$, and thus
\[
\rb{1+\frac{x_4^2}{x_5^2+1}}^2 \asymp \frac{Z_1^2}{(x_5^2+1)^2} \asymp \frac{Z_2}{(x_5^2+1)^2}.
\]
As $x_5\gg 1$ is large, we have $V(x_4,x_5) \asymp Z_2/(x_5^2+1)$. Hence for fixed $x_5$, the volume of $x_2$ verifying~\eqref{eq:shell_bab2} is $\ll 2c_2Z_2^{1/2}/R_2(x_5^2+1)^{1/2}$. Meanwhile, the volume of the elements $x_4$ verifying~\eqref{eq:shell_bab1} is $\ll 2c_1Z_1^{1/2}/R_1$. So the contribution to $\vol(\mathcal X_{\beta\alpha\beta})$ from this case is 
\[
\frac{(Z_1Z_2)^{1/2}}{R_1R_2} \int_{x_5\ll Z_1} \frac{dx_5}{(x_5^2+1)^{1/2}} \ll \frac{(Z_1Z_2)^{1/2}}{R_1R_2} \log Z_1.
\]

Now we suppose $x_5\ll 1$. Let $U\le 2/3$ be a parameter, and consider the dyadic intervals
\begin{equation}\label{eq:shell_bab3}
\frac{Z_2U}{x_5^2+1} \le \frac{Z_2(1+c_2/R_2)}{x_5^2+1} - \rb{1+\frac{x_4^2}{x_5^2+1}}^2 \le \frac{2Z_2U}{x_5^2+1}.
\end{equation}
For $(x_4,x_5)$ verifying \eqref{eq:shell_bab3}, the volume of $x_2$ verifying \eqref{eq:shell_bab2} is bounded by
\[
\ll \min\cb{\frac{2c_2Z_2^{1/2}}{R_2U^{1/2}(x_5^2+1)^{1/2}}, \frac{Z_2U^{1/2}}{(x_5^2+1)^{1/2}}}.
\]
Meanwhile, we solve $x_4$ from \eqref{eq:shell_bab3}:
\begin{equation}\label{eq:shell_bab4}
Z_2^{1/2}\sqrt{1+c_2/R_2-2U}(x_5^2+1)^{1/2} - (x_5^2+1) \le x_4^2 \le Z_2^{1/2}\sqrt{1+c_2/R_2-U}(x_5^2+1)^{1/2}-(x_5^2+1).
\end{equation}
Together with \eqref{eq:shell_bab1}, this says $x_4$ is solvable only if
\begin{align*}
    Z_1(1+c_1/R_1)&\ge Z_2^{1/2}\sqrt{1+c_2/R_2-2U}\sqrt{x_5^2+1}, & Z_1(1-c_1/R_1)&\le Z_2^{1/2}\sqrt{1+c_2/R_2-U}\sqrt{x_5^2+1},
\end{align*}
that is,
\begin{equation}\label{eq:shell_bab5}
\frac{Z_1(1-c_1/R_1)}{Z_2^{1/2}(1+c_2/R_2-U)^{1/2}} \le (x_5^2+1)^{1/2} \le \frac{Z_1(1+c_1/R_1)}{Z_2^{1/2}(1+c_2/R_2-2U)^{1/2}}.
\end{equation}
The volume of $x_5\ll 1$ verifying \eqref{eq:shell_bab5} is $\ll \max\{R_1^{-1/2}, U^{1/2}\}$, and the volume of $x_4$ verifying \eqref{eq:shell_bab1} and \eqref{eq:shell_bab4} is $\ll\min\{c_1Z_1^{1/2}/R_1, Z_2^{1/4}U\}$. Summing over the dyadic intervals, the contribution to~$\vol(\mathcal X_{\beta\alpha\beta})$ from this case is $\ll \min\{\log R_1,\log R_2\}(Z_1Z_2)^{1/2}/R_1R_2$. This ends the proof for the Weyl element  $w=\beta\alpha\beta$.

\subsubsection{Proof for $w=w_0$} For $w=w_0$, \eqref{eq:shell_condition} reads, setting $C:= x_1x_4-x_2$,
\begin{align}
    Z_1(1-c_1/R_1)&\le 1+2x_4^2+x_5^2+C^2+(x_4^2+x_5C)^2 \le Z_1(1+c_1/R_1),\label{eq:shell_w01}\\
    Z_2(1-c_2/R_2)&\le 1+x_1^2+(x_1x_4-C)^2+(x_1x_5+x_4)^2 \le Z_2(1+c_2/R_2).\label{eq:shell_w02}
\end{align}
From \eqref{eq:shell_w02} we get
\begin{align}
    2x_4^2+x_5^2 &\le Z_1(1+c_1/R_1)-1,\label{eq:shell_w03}\\
    Z_1(1-c_1/R_1)-x_5^2-2x_4^2-1 &\le C^2+(x_5C+x_4^2)^2\le Z_1(1+c_1/R_1)-x_5^2-2x_4-1.\label{eq:shell_w04}
\end{align}
By completing square with respect to $C$, we may rewrite \eqref{eq:shell_w04} as 
\begin{equation}\label{eq:shell_w05}
    \frac{Z_1(1-c_1/R_1)}{x_5^2+1} - \frac{(x_4^2+x_5^2+1)^2}{(x_5^2+1)^2} \le \rb{C+\frac{x_4^2x_5}{x_5^2+1}}^2 \le \frac{Z_1(1+c_1/R_1)}{x_5^2+1} - \frac{(x_4^2+x_5^2+1)^2}{(x_5^2+1)^2},
\end{equation}
and by completing square with respect to $x_1$, we may rewrite \eqref{eq:shell_w02} as
\begin{multline}\label{eq:shell_w06}
    \frac{Z_2(1-c_2/R_2)}{x_4^2+x_5^2+1} - \frac{(x_5^2+1)\rb{C+\tfrac{x_4^2x_5}{x_5^2+1}}^2}{(x_4^2+x_5^2+1)^2} - \frac{1}{x_5^2+1} \le \rb{x_1+\frac{x_4x_5-x_4C}{x_4^2+x_5^2+1}}^2\\
    \le \frac{Z_2(1+c_2/R_2)}{x_4^2+x_5^2+1} - \frac{(x_5^2+1)\rb{C+\tfrac{x_4^2x_5}{x_5^2+1}}^2}{(x_4^2+x_5^2+1)^2} - \frac{1}{x_5^2+1}
\end{multline}
Let $U_1 \le 2/3$ be a parameter, and consider the dyadic interval
\begin{equation}\label{eq:shell_w07}
Z_1U_1 \le Z_1(1+c_1/R_1) - \frac{(x_4^2+x_5^2+1)^2}{x_5^2+1} \le 2Z_1U_1.
\end{equation}
Using \eqref{eq:root_formula}, for fixed $x_4,x_5$ verifying \eqref{eq:shell_w07}, the volume of $C$ verifying \eqref{eq:shell_w05} is bounded by
\[
\le \frac{2c_1Z_1^{1/2}}{R_1U_1^{1/2}(x_5^2+1)^{1/2}}.
\]
Meanwhile, plugging \eqref{eq:shell_w05} into \eqref{eq:shell_w06} gives
\begin{equation}\label{eq:shell_w08}
    \frac{Z_2(1-c_2/R_2)}{x_4^2+x_5^2+1}-\frac{Z_1(1+c_1/R_1)}{(x_4^2+x_5^2+1)^2} \le \rb{x_1+\frac{x_4x_5-x_4C}{x_4^2+x_5^2+1}}^2 \le \frac{Z_2(1+c_2/R_2)}{x_4^2+x_5^2+1}-\frac{Z_1(1-c_1/R_1)}{(x_4^2+x_5^2+1)^2}.
\end{equation}
Let $U_2\le 2/3$ be a parameter, and consider the dyadic interval
\begin{equation}\label{eq:shell_w09}
Z_2U_2\le Z_2(1+c_2/R_2) - \frac{Z_1(1-c_1/R_1)}{x_4^2+x_5^2+1} \le 2Z_2U_2.
\end{equation}
Using \eqref{eq:root_formula}, for fixed $x_4,x_5$ verifying \eqref{eq:shell_w09}, the volume of $x_1$ verifying \eqref{eq:shell_w08} is bounded by
\[
\le \frac{2c_2Z_2^{1/2}}{R_2U_2^{1/2}(x_4^2+x_5^2+1)^{1/2}}.
\]
Now we consider the following cases.
\begin{enumerate}[label=(\roman*)]
    \item Suppose $U_1,U_2\ge 1/100$. Summing over these dyadic intervals, the volume of $(x_1,x_2,x_4,x_5)$ satisfying the aforementioned conditions is bounded by
    \[
    \ll \frac{(Z_1Z_2)^{1/2}}{R_1R_2} \int_{(x_4,x_5) \text{ verifying \eqref{eq:shell_w03}}} \frac{dx_4dx_5}{(x_4^2+x_5^2+1)^{1/2}(x_5^2+1)^{1/2}} \ll \frac{(Z_1Z_2)^{1/2}}{R_1R_2}(\log Z_1)^2.
    \]
    \item Suppose $U_1\ge 1/100$, $U_2\le 1/100$. It follows from \eqref{eq:shell_w09} that $x_4^2+x_5^2+1 \asymp Z_1/Z_2$. Introduce a parameter $U_3\le 2/3$ and consider the dyadic interval
    \[
    Z_1U_3/Z_2 \le x_5^2 + 1 \le 2Z_1U_3/Z_2.
    \]
    For each dyadic interval, the volume of $(x_1,x_2,x_4,x_5)$ satisfying the aforementioned conditions is bounded by $(Z_1Z_2U_2)^{1/2}/R_1R_2U_1^{1/2}$. Summing over these dyadic intervals, we find that the contribution to $\vol(\mathcal X_{w_0})$ from this case is $\ll (\log Z_1)(Z_1Z_2)^{1/2}/R_1R_2$.
    \item Suppose $U_1\le 1/100$ and $U_2\ge 1/100$. Let $U_4\le 2/3$ be a parameter and consider the dyadic interval
    \[
    Z_1U_4 \le x_5^2+1 \le 2Z_1U_4.
    \]
    For each dyadic interval, the volume of $(x_1,x_2,x_4,x_5)$ satisfying the aforementioned conditions is bounded by $(Z_1Z_2U_2)^{1/2}/R_1R_2U_1^{1/2}$. Summing over these dyadic intervals, we find that the contribution to $\vol(\mathcal X_{w_0})$ from this case is $\ll (\log Z_1)(Z_1Z_2)^{1/2}/R_1R_2$ as well.
    \item Finally, suppose $U_1,U_2\le 1/100$. We use \eqref{eq:shell_w07} and \eqref{eq:shell_w09} to solve $x_4$ in terms of $x_5$:
    \begin{align*}
        Z_1^{1/2}\sqrt{1+c_1/R_1-2U_1}\sqrt{x_5^2+1} - x_5^2-1 &\le x_4^2 \le Z_1^{1/2}\sqrt{1+c_1/R_1-U_1}\sqrt{x_5^2-1}-x_5^2-1,\\
        \frac{Z_1(1-c_1/R_1)}{Z_2(1+c_2/R_2-U_2)} - x_5^2-1 &\le x_4^2 \le \frac{Z_1(1-c_1/R_1)}{Z_2(1+c_2/R_2-2U_2)} - x_5^2-1.
    \end{align*}
    For these inequality to have nonempty intersection, we need
    \[
    \frac{Z_1^{1/2}(1-c_1/R_1)}{Z_2(1+c_2/R_2-U_2)\sqrt{1+c_1/R_1-U_1}} \le \sqrt{x_5^2+1} \le \frac{Z_1^{1/2}(1-c_1/R_1)}{Z_2(1+c_2/R_2-2U_2)\sqrt{1+c_1/R_1-2U_1}}.
    \]
    For each dyadic interval, we use \eqref{eq:root_formula} and find that the volume of $(x_1,x_2,x_4,x_5)$ satisfying the aforementioned conditions is $\ll (Z_1Z_2U_1U_2)^{1/2}/R_1R_2$. Summing over these dyadic intervals, we find that the contribution to $\vol(\mathcal X_{w_0})$ from this case is $\ll (Z_1Z_2)^{1/2}/R_1R_2$. This finishes the proof for $w=w_0$.
\end{enumerate}

This completes the proof of \Cref{lem:shell_volume}. \qed

\section{Consequences}
\label{sec:consequences}

All the tools are now at hand to give the proof of all the theorems announced in \Cref{subsec:consequences}.

\subsection{Weyl law}

To prove \Cref{thm:Weyl_law}, we apply \eqref{eq:ktf} with $M=N=(1,1)$, $\tau_1 = R_1 = T_1$, $\tau_2=R_2=T_2$, $X_1=X_2=1$, and pick the normalized test function $F$ as in \eqref{eq:ntf}.

On the arithmetic side, the main contribution is given by $\mathcal K_{\id}$, which has size $\|F\|^2 \asymp \mathcal T$. For the terms $\mathcal K_{\alpha \beta \alpha }, \mathcal K_{\beta \alpha \beta }, \mathcal K_{w_0}$, by the vanishing statements in \Cref{prop:Jw_trivial} the sums over $c_1,c_2$ are finite, and by the decay statements in \Cref{prop:Jw_trivial} we have the bounds
\[
\mathcal K_{\alpha \beta \alpha }, \mathcal K_{\beta \alpha \beta }, \mathcal K_{w_0} \ll \mathcal T^{-100}.
\]

On the spectral side, we drop the continuous spectrum by positivity and apply \Cref{thm:sl}. The spectral transform $|\langle W_\mu, F\rangle|$ has constant size when $\mu \approx (\tau_1, \tau_2)$, so that we obtain a bound
\[
\sum_{\substack{|\mu_1(\varpi) - i\tau_1| \le c\\ |\mu_2(\varpi) - i\tau_2| \le c}} L(1,\varpi,\Ad)^{-1} \ll \mathcal T
\]
for some sufficiently small constant $c$ which depends only on $f$, using \Cref{lem:adlv}. By varying the target parameters $\tau_1,\tau_2$ around $T_1,T_2$ and combining their contributions, we find that
\[
\sum_{\substack{|\mu_1(\varpi) - iT_1| \le K\\ |\mu_2(\varpi) - iT_2| \le K}} L(1,\varpi,\Ad)^{-1} \ll_K \mathcal T
\]
for any fixed constant $K\ge 1$. This gives us the upper bound.

For the lower bound, we choose $K$ sufficiently large such that
\[
\sum_{\max\{|\mu_1(\varpi) - iT_1|,|\mu_2(\varpi) - iT_2|\}\ge K} \frac{|\langle \widetilde W_{\mu_1,\mu_2}, F \rangle|^2}{\|\varpi\|^2} \le \frac{\|F\|^2}2,
\]
which is possible by \Cref{thm:sl}. We bound the Eisenstein series trivially, use the Weyl law for~$\GL(2)$ (which says there are $\asymp T$ cusps forms $\phi$ with $\nu_\phi - iT = O(1)$), and the bounds \eqref{eq:zlb} as well as \eqref{eq:lflb}; we find that
\[
\mathcal S_{0} \ll T_2^\varepsilon, \quad \mathcal S_{S}, \mathcal S_{K} \ll T_2^{1+\varepsilon}.
\]
This gives
\[
\sum_{\substack{|\mu_1(\varpi) - iT_1| \le K\\ |\mu_2(\varpi) - iT_2| \le K}} L(1,\varpi,\Ad)^{-1} \ge \frac 12 \|F\|^2 + O\rb{T_2^{1+\varepsilon}}. 
\]
As $\|F\|^2 \asymp \mathcal T$, the lower bound is established, finishing the proof of \Cref{thm:Weyl_law}. \qed

\subsection{Non-tempered spectrum}

To prove \Cref{thm:nts}, we apply \eqref{eq:ktf} with $M=N=(1,1)$, $\tau_1=\tau_2=R_1=R_2=T$, $X_1 = 1$, $X_2=X=T^\delta$ for some $\delta > 0$ and pick the test function $F$ as in \eqref{eq:ntf}. The parameter $X$ amplifies the contribution of the non-tempered spectrum. 

We collect the contributions from the arithmetic side. By \eqref{eq:ntf_trivial}, the contribution from $\mathcal K_{\id}$ has size~$\asymp T^3X^3$. For $\mathcal K_{\alpha\beta\alpha}$, we use \eqref{eq:Jaba_vanish} to truncate the $c$-sum and obtain
\[
\mathcal K_{\alpha\beta\alpha} = \sum_{c\ll X} \frac{\Kl_{\alpha\beta\alpha}((c,c),(1,1),(1,1))}{c^2} \mathcal J_{\alpha\beta\alpha, F}(c^{-1/2}). 
\]
When $\delta \le 2-\varepsilon$ for some fixed $\varepsilon > 0$, we can use \eqref{eq:Jaba_trivial} to save as many powers of $T$ as we want. By the same arguments, we have
\[
\mathcal K_{\beta\alpha\beta} = \sum_{c\ll X^2} \frac{\Kl_{\beta\alpha\beta}((c,c^2),(1,1),(1,1))}{c^3} \mathcal J_{\beta\alpha\beta,F}(c^{-1}),
\]
and \eqref{eq:Jbab_trivial} says we can save as many powers of $T$ as we want.

Finally, for $\mathcal K_{w_0}$, we use \eqref{eq:Jw0_vanish} and \eqref{eq:Jw0_trivial} to truncate the sum to $c_1 \ll \min\{X,X^2/T\}$, $c_2\ll X/T$, and apply the bounds \eqref{eq:Kw0_bound} and \eqref{eq:Jw0_nontriv} to the remaining sum, obtaining the bound 
\begin{equation}\label{eq:nts_w0}
\mathcal K_{w_0} \ll T^{7/3+\varepsilon} X^{3+\varepsilon} \sum_{c_1\ll \min\{X,X^2/T\}} \sum_{c_2\ll X/T} c_1^{-1/2+\varepsilon} c_2^{-1/4+\varepsilon} (c_1,c_2)^{1/2}.
\end{equation}
We consider the following cases:
\begin{enumerate}[label=(\roman*)]
    \item Suppose $\delta \le 1$. In this case, the sum runs over $c_1\ll X^2/T$ and $c_2\ll X/T\ll 1$. So we may bound \eqref{eq:nts_w0} by
    \begin{align*}
    \mathcal K_{w_0} &\ll T^{7/3+\varepsilon} X^{3+\varepsilon} \sum_{c_1\ll X^2/T} \sum_{c_2\ll 1} c_1^{-1/2+\varepsilon} c_2^{-1/4+\varepsilon} (c_1,c_2)^{1/2}\\
    &\ll T^{11/6+\varepsilon} X^{4+\varepsilon}.
    \end{align*}
    \item Suppose $\delta \ge 1$. In this case, the sum runs over $c_1\ll X$, and $c_2\ll X/T$. Writing $d=(c_1,c_2)$, and $c_1=dc'_1$, $c_2=dc'_2$, we may bound \eqref{eq:nts_w0} by
    \begin{align*}
        \mathcal K_{w_0} &\ll T^{7/3+\varepsilon} X^{3+\varepsilon} \sum_{d\ll X/T} d^{-1/4+\varepsilon} \sum_{c'_1\ll X/d} \sum_{c'_2\ll X/Td} {c'_1}^{-1/2+\varepsilon} {c'_2}^{-1/4+\varepsilon}\\
        &\ll T^{19/12+\varepsilon} X^{17/4+\varepsilon}.
    \end{align*}
\end{enumerate}
Combining the estimates above, for $\delta \le 2-\varepsilon$ the arithmetic side has size bounded by
\[
\ll (TX)^\varepsilon \rb{T^3X^3 + T^{19/12} X^{17/4}}.
\]

Meanwhile, on the spectral side we keep only the non-tempered spectrum by positivity and use \Cref{thm:sl}. This says the spectral side is bounded below by
\[
\gg T^{-\varepsilon} X^3 \sum_{\substack{|\Im\mu_2(\varpi) - T| \le K \\ |R(\varpi)| \ge \varepsilon}} X^{2|R(\varpi)|}.
\]
Substituting $\delta = \frac{17}{15}$ then yields \Cref{thm:nts}. \qed

\subsection{Large sieve inequalities}

Now we prove \Cref{thm:ls}. First we prove \eqref{eq:ls1}. This time we apply \eqref{eq:ktf} with $R_1=T_1$, $R_2=T_2$, $X_1=X_2=1$, and pick $F=F_\tau$ as in \eqref{eq:ntf}. In view of \Cref{thm:sl}, we may cover the concerned part of the spectrum by integrating \eqref{eq:ktf} over $\tau_1,\tau_2$ with weight $g(\tau_1/T_1)g(\tau_2/T_2)$, where $g:\R_+\to\R$ is a non-negative smooth function with support in $[1/2,3]$. Together with \Cref{lem:adlv}, the left hand side of \eqref{eq:ls1} satisfies the bound
\[
\sum_{\substack{T_1\le|\mu_1(\varpi)|\le 2T_1\\ T_2\le|\mu_2(\varpi)|\le 2T_2}} \hspace{-0.1cm} \Big| \sum_{n\le N} \alpha(n) A_{\varpi}(1,n)\Big|^2 \hspace{-0.1cm} \ll T_2^\varepsilon \sum_{\varpi} \int_{\R_+^2} g\rb{\frac{\tau_1}{T_1}}g\rb{\frac{\tau_2}{T_2}} \frac{|\langle \widetilde W_\mu, F\rangle|^2}{\|\varpi\|^2} d\tau_1 d\tau_2 \Big| \sum_{n\le N} \alpha(n) A_{\varpi}(1,n)\Big|^2.
\]
Next we cut the $n$-sum into dyadic intervals and expand the square. As there are $\ll \log N$ such intervals, we may bound the preceding expression by
\begin{align*}
\ll (NT_2)^\varepsilon \max_{M\le N} \sum_{\varpi} \int_{\R_+^2} g\rb{\frac{\tau_1}{T_1}} g\rb{\frac{\tau_2}{T_2}} \frac{|\langle \widetilde W_\mu, F\rangle|^2}{\|\varpi\|^2} d\tau_1 d\tau_2 \Big| \sum_{M\le n\le 2M} \alpha(n) A_{\varpi}(1,n)\Big|^2.
\end{align*}
We add in the contribution from the continuous spectrum by positivity, then open the square and apply the Kuznetsov formula. Now we collect the terms on the arithmetic side. For $\mathcal K_{\id}$ the contribution is bounded by
\[
\ll (NT_2)^\varepsilon \max_{M\le N} \delta_{m=n} \mathcal T \vb{\alpha(m)\alpha(n)} \int_{\R_+^2} g\rb{\frac{\tau_1}{T_1}} g\rb{\frac{\tau_1}{T_2}} d\tau_1 d\tau_2 \ll (NT_2)^\varepsilon T_1^2T_2^4 \|\alpha\|_2^2.
\]
The contribution from $\mathcal K_{\alpha\beta\alpha}$ is bounded by
\begin{multline*}
\ll (NT_2)^\varepsilon \max_{M\le N} \sum_{M\le m,n \le 2M} \vb{\alpha(m)\alpha(n)} \sum_{\substack{c_2\mid c_1^2\\ mc_1^2=nc_2^2}} \frac{\vb{\Kl_{\alpha\beta\alpha}((c_1,c_2),(1,m),(1,n))}}{c_1c_2}\\
\times \vb{\int_{\R_+^2} g\rb{\frac{\tau_1}{T_1}} g\rb{\frac{\tau_1}{T_2}} \mathcal J_{\alpha\beta\alpha,F} \rb{\sqrt{m/c_2}} d\tau_1 d\tau_2}.
\end{multline*}
Applying \eqref{eq:arith_aba} yields, writing $d = (m,n)$, $m=dm'$, $n=dn'$,
\begin{align*}
&\ll (NT_2)^\varepsilon \max_{M\le N} \sum_{\substack{M\le m,n\le 2M\\ m/n \in (\Q^\times)^2}} \vb{\alpha(m)\alpha(n)} \mathcal T^{1+\varepsilon} T_2^{-4/3+\varepsilon} d^{2/3+\varepsilon} {m'}^{1/2+\varepsilon} {n'}^{-1/3+\varepsilon}\\
&\ll N^{5/3+\varepsilon} T_1^{1+\varepsilon}T_2^{5/3+\varepsilon}\|\alpha\|_2^2.
\end{align*}

The contribution from $\mathcal K_{\beta\alpha\beta}$ is bounded by
\begin{multline*}
\ll (NT_2)^\varepsilon \max_{M\le N} \sum_{M\le m,n \le 2M} \vb{\alpha(m)\alpha(n)} \sum_{c_2=c_1^2} \frac{\vb{\Kl_{\beta\alpha\beta}((c_1,c_2),(1,m),(1,n))}}{c_1c_2}\\
\times \vb{\int_{\R_+^2} g\rb{\frac{\tau_1}{T_1}} g\rb{\frac{\tau_1}{T_2}} \mathcal J_{\beta\alpha\beta,F} \rb{\sqrt{mn}/c_1} d\tau_1 d\tau_2}.
\end{multline*}
Applying \eqref{eq:arith_bab} yields, writing $d=(m,n)$, $m=dm'$, $n=dn'$,
\begin{align*}
&\ll (NT_2)^\varepsilon \max_{M\le N} \sum_{M\le m,n \le 2M} \vb{\alpha(m)\alpha(n)} \mathcal T^{1+\varepsilon} T_2^{-2+\varepsilon} d^{3/2+\varepsilon} (m'n')^{1/4+\varepsilon}\\
&\ll N^{5/2+\varepsilon} T_1^{1+\varepsilon} T_2^{1+\varepsilon} \|\alpha\|_2^2.
\end{align*}

Finally, the contribution from $\mathcal K_{w_0}$ is bounded by
\begin{multline*}
\ll(NT_2)^\varepsilon \max_{M\le N} \sum_{M\le m,n\le 2M} \vb{\alpha(m)\alpha(n)} \sum_{c_1,c_2} \frac{\vb{\Kl_{w_0}((c_1,c_2),(1,m),(1,n))}}{c_1c_2}\\
\times \vb{\int_{\R_+^2} g\rb{\frac{\tau_1}{T_1}} g\rb{\frac{\tau_1}{T_2}} \mathcal J_{w_0,F} \rb{\frac{\sqrt{c_2}}{c_1}, \frac{\sqrt{mn}c_1}{c_2}} d\tau_1 d\tau_2}.
\end{multline*}
Applying \eqref{eq:arith_w0} yields
\begin{align*}
&\ll(NT_2)^\varepsilon \max_{M\le N} \sum_{M\le m,n\le 2M} \vb{\alpha(m)\alpha(n)} \mathcal T^{1+\varepsilon} T_2^{-5/4+\varepsilon} (mn)^{1+\varepsilon}\\
&\ll N^{3+\varepsilon} T_1^{1+\varepsilon} T_2^{7/4+\varepsilon} \|\alpha\|_2^2.
\end{align*}
Combining the contributions above yields \eqref{eq:ls1}, noting that $\mathcal K_{\beta\alpha\beta}$ is dominated by $\mathcal K_{w_0}$. 

The inequality \eqref{eq:ls2} can be proven similarly. Analogously to the above, the contribution from $\mathcal K_{\id}$ is bounded by~$(NT_2)^\varepsilon T_1^2T_2^4\|\alpha\|_2^2$. The contribution from $\mathcal K_{\alpha\beta\alpha}$ is bounded by
\begin{multline}\label{eq:ls2_aba}
\ll (NT_2)^\varepsilon \max_{M\le N} \sum_{M\le m,n\le 2M} \vb{\alpha(m)\alpha(n)} \sum_{c_1=c_2} \frac{\vb{\Kl_{\alpha\beta\alpha}((c_1,c_2),(m,1),(n,1)}}{c_1c_2}\\
\times \vb{\int_{\R_+^2} g\rb{\frac{\tau_1}{T_1}} g\rb{\frac{\tau_1}{T_2}} \mathcal J_{\alpha\beta\alpha,F} \rb{\sqrt{mn/c_2}} d\tau_1 d\tau_2}.
\end{multline}
Applying \eqref{eq:arith_aba} yields, writing $d=(m,n)$, $m=dm'$, $n=dn'$, 
\begin{align*}
&\ll (NT_2)^\varepsilon \max_{M\le N} \sum_{M\le m,n\le 2M} \vb{\alpha(m)\alpha(n)} \mathcal T^{1+\varepsilon} T_2^{-4/3+\varepsilon} d^{7/3+\varepsilon} (m'n')^{2/3+\varepsilon}\\
&\ll N^{10/3+\varepsilon} T_1^{1+\varepsilon} T_2^{5/3+\varepsilon} \|\alpha\|_2^2.
\end{align*}

The contribution from $\mathcal K_{\beta\alpha\beta}$ is bounded by
\begin{multline*}
\ll (NT_2)^\varepsilon \max_{M\le N} \sum_{M\le m,n\le 2M} \vb{\alpha(m)\alpha(n)} \sum_{\substack{c_1\mid c_2\\ mc_2=nc_1^2}} \frac{\vb{\Kl_{\beta\alpha\beta}((c_1,c_2),(m,1),(n,1))}}{c_1c_2}\\
\times \vb{\int_{\R_+^2} g\rb{\frac{\tau_1}{T_1}} g\rb{\frac{\tau_1}{T_2}} \mathcal J_{\beta\alpha\beta,F} \rb{\sqrt{m/c_1}} d\tau_1 d\tau_2}.
\end{multline*}
Applying \eqref{eq:arith_bab} yields,
\begin{align*}
&\ll (NT_2)^\varepsilon \max_{M\le N} \sum_{M\le m,n\le 2M} \vb{\alpha(m)\alpha(n)} \mathcal T^{1+\varepsilon} T_2^{-2+\varepsilon} d^{1/2+\varepsilon} {m'}^{\varepsilon} {n'}^{1/2+\varepsilon}\\
&\ll N^{3/2+\varepsilon} T_1^{1+\varepsilon} T_2^{1+\varepsilon} \|\alpha\|_2^2.
\end{align*}

Finally, the contribution from $\mathcal K_{w_0}$ is bounded by
\begin{multline*}
\ll (NT_2)^\varepsilon \max_{M\le N} \sum_{M\le m,n\le 2M} \vb{\alpha(m)\alpha(n)} \sum_{c_1,c_2} \frac{\vb{\Kl_{w_0}((c_1,c_2),(m,1),(n,1))}}{c_1c_2}\\
\times \vb{\int_{\R_+^2} g\rb{\frac{\tau_1}{T_1}} g\rb{\frac{\tau_1}{T_2}} \mathcal J_{w_0,F} \rb{\frac{\sqrt{mnc_2}}{c_1},\frac{c_1}{c_2}} d\tau_1 d\tau_2}.
\end{multline*}
By \eqref{eq:arith_w0} this is bounded by
\begin{align*}
&\ll (NT_2)^\varepsilon \max_{M\le N} \sum_{M\le m,n\le 2M} \vb{\alpha(m)\alpha(n)} \mathcal T^{1+\varepsilon} T_2^{-5/4+\varepsilon} (mn)^{5/4+\varepsilon}\\
&\ll N^{7/2+\varepsilon} T_1^{1+\varepsilon} T_2^{7/4+\varepsilon} \|\alpha\|_2^2.
\end{align*}
Combining the contributions above yields \eqref{eq:ls2}, noting that $\mathcal K_{\alpha\beta\alpha}, \mathcal K_{\beta\alpha\beta}$ are dominated by $\mathcal K_{w_0}$. This finishes the proof of \Cref{thm:ls}. \qed

\subsection{Second moments of $L$-functions}

We can classically deduce an application to bounds of the second moment in the spectral aspect. The proof of \Cref{thm:2mspin} follows from using an approximate functional equation to express $L(\tfrac 12,\varpi,\Spin)$ as an essentially finite Dirichlet series, expanding the square and invoking the large sieve inequality proven above. In order to do so, it is crucial to have a version of the approximate functional equation which is uniform in the spectral parameter. This has been obtained in \cite{bh} for any entire $L$-function and in particular applies to the present setting. 

\begin{lem}[{\cite[Proposition 1]{bh}}]\label{lem:bh}
Let $G_0 : (0, \infty) \to \mathbf{R}$ be a smooth function with functional equation $G_0(x) + G_0(1/x) = 1$ and derivatives decaying faster than any negative power of $x$ as $x \to \infty$. Let $M \in \mathbf{N}$ and fix a cuspidal automorphic form $\varpi$ of $\GL(m)$. Let $L(\varpi,s) = \sum_{n\ge 1} a_n n^{-s}$ denote the associated $L$-function. There are explicitly computable rational constants $c_{n,\ell} \in \mathbf{Q}$ depending only on $n,\ell,M,m$ such that for
\begin{equation}
G(x) := G_0(x) + \sum_{\substack{0 < |n| < M \\ 0 < \ell < |n| + M}} c_{n, \ell} \eta_j^{-n} ( x \partial_x)^\ell G_0(x), 
\end{equation}
we have, for any $\varepsilon >0$, 
\begin{equation}\label{eq:bh}
L(\varpi, \tfrac12) = \sum_{n \geqslant 1} \frac{a_n}{n^{1/2}} G\left( \frac{n}{\sqrt{C}}\right) + \kappa \overline{\sum_{n \geqslant 1} \frac{a_n}{n^{1/2}} G\left( \frac{n}{\sqrt{C}}\right)} + O\rb{\eta^{-M} C^{1/4+\varepsilon}}
\end{equation}
where $\eta, C,$ and $\kappa$ only depends upon the $L$-function, as made precise in \cite{bh}. Here, $|\kappa|=1$ and the implied constant in the error term depends at most on $\varepsilon$, $M$ and $G_0$, but not on~$\varpi$.
\end{lem}

\begin{proof}[Proof of \Cref{thm:2mspin}]

For a form $\varpi$ on $\GSp(4)$, or rather its lift to $\GL(4)$ by \cite{gt}, $\eta$ and $C$ are given as follows. Let 
\[
\cb{\eta_1, \eta_2, \eta_3, \eta_4} = \cb{\tfrac 14+\mu_1, \tfrac 14+\mu_2, \tfrac 14-\mu_1, \tfrac 14-\mu_2}.
\]
Then $\eta$ is given by $\eta = \min_{1\le i \le 4} |\eta_i|$, and $C$ is given by $C = \prod_{i=1}^4 \eta_i$. In particular, the assumptions on the spectral parameters imply $\eta \asymp T_1$, and $C \asymp T_1^2T_2^2$. Inserting the approximate functional equation \eqref{eq:bh}, and writing $G_\ell = (x\partial_x)^\ell G_0$, we have for all $M, A\geqslant 1$, 
\[
\sum_{\substack{T_1 \le |\mu_1(\varpi)| \le 2T_1\\ T_2 \le |\mu_2(\varpi)| \le 2T_2}} \vb{L(\tfrac12,\varpi,\Spin)}^2 \ll \sum_{\ell \leqslant M} \sum_{\substack{T_1 \le |\mu_1(\varpi)| \le 2T_1\\ T_2 \le |\mu_2(\varpi)| \le 2T_2}} \rb{\vb{\sum_{n \leqslant C^{1/2+\varepsilon}} \frac{A_{\varpi}(1,n)}{n^{1/2}} G_\ell\rb{\frac{n}{\sqrt{C}}}} + O\rb{\eta^{-M} C^{1/4+\varepsilon}}}^2.
\]
We used here that $G_\ell$ is essentially supported in $(0,n^\varepsilon)$ to truncate the summation over $n$. Using Mellin inversion to express $G_\ell$ as a vertical integral on the line $\Re(s) = \varepsilon$, and appealing to the rapid decay of the Mellin transform $\hat{G}_\ell$ to truncate that integral to $(-T_2^\varepsilon, T_2^\varepsilon)$, this is bounded by
\begin{align*}
&\ll T_2^\varepsilon \sum_{\substack{T_1 \le |\mu_1(\varpi)| \le 2T_1\\ T_2 \le |\mu_2(\varpi)| \le 2T_2}} \rb{\int_{-T_2^\varepsilon}^{T_2^\varepsilon} \vb{\sum_{n \leqslant (T_1T_2)^{1+\varepsilon}} \frac{A_{\varpi}(1,n)}{n^{1/2 + \varepsilon + it}}} dt  + O\rb{T_1^{-M+1/2}T_2^{1/2}}}^2\\
&\ll T_2^\varepsilon \rb{\max_{|t|\le T_2^\varepsilon} \sum_{\substack{T_1 \le |\mu_1(\varpi)| \le 2T_1\\ T_2 \le |\mu_2(\varpi)| \le 2T_2}} \vb{\sum_{n \leqslant (T_1T_2)^{1+\varepsilon}} \frac{A_{\varpi}(1,n)}{n^{1/2+\varepsilon+it}}}^2 + \sum_{\substack{T_1 \le |\mu_1(\varpi)| \le 2T_1\\ T_2 \le |\mu_2(\varpi)| \le 2T_2}} T_1^{-2M+1} T_2}.
\end{align*}

We bound the first term using the large sieve inequality (\Cref{thm:ls}) with $N = (T_1T_2)^{1+\varepsilon}$ and~$\alpha(n) = n^{-1/2-\varepsilon-it}$, which gives
\begin{align*}
\max_{|t|\le T_2^\varepsilon} \sum_{\substack{T_1 \le |\mu_1(\varpi)| \le 2T_1\\ T_2 \le |\mu_2(\varpi)| \le 2T_2}} \vb{\sum_{n \leqslant (T_1T_2)^{1+\varepsilon}} \frac{A_{\varpi}(1,n)}{n^{1/2+\varepsilon+it}}}^2 &\ll (T_1T_2)^\varepsilon \rb{T_1^2T_2^4 + (T_1T_2)^{5/3}T_1 T_2^{5/3} + (T_1T_2)^3T_1T_2^{7/4}}\\
&\ll T_1^{4+\varepsilon} T_2^{19/4+\varepsilon}.
\end{align*}

On the other hand, we use the Weyl law (\Cref{thm:Weyl_law}) and \Cref{lem:adlv} to bound the second term:
\[
\sum_{\substack{T_1 \le |\mu_1(\varpi)| \le 2T_1\\ T_2 \le |\mu_2(\varpi)| \le 2T_2}} T_1^{-2M+1} T_2 \ll T_1^{-2M+3} T_2^5.
\]
This finishes the proof of \Cref{thm:2mspin}.
\end{proof}

\begin{proof}[Proof of \Cref{thm:2mstd}] The standard $L$-function $L(s,\varpi,\Std)$ is not known to correspond to the $L$-function of an automorphic representation of $\GL(m)$, so we cannot directly use \Cref{lem:bh}, unless we assume the functoriality conjecture of Langlands. Nevertheless, a slightly weaker version of \Cref{lem:bh} still holds for our family of $L$-function $L(s,\varpi,\Std)$, which we describe below. 

We note that the proof of \Cref{lem:bh} in \cite{bh, harcos2, harcos1} relies on an analytic argument, and the automorphicity assumption is used for \begin{enumerate*}[label=(\roman*)] \item the functional equation, \item the archimedean part of the $L$-function is holomorphic for $\Re s\ge \frac 12$, and \item the bound on the average size of the coefficients~$a_n$ by Molteni \cite{molteni}: $\sum_{n\le X} |a_n| \ll X^{1+\varepsilon}$.\end{enumerate*} For $L(s,\varpi,\Std)$, (i) and (ii) are satisfied (noting that our selection of the spectral parameters $(\mu_1,\mu_2)$ forces temperedness by \eqref{eq:man}), and (iii) is replaced by a weaker bound $|a_n| \ll n^{9/11+\varepsilon}$ \cite{bb,gt}. With this weaker bound we obtain \eqref{eq:bh} but with a larger error term of size $O(\eta^{-M} C^{29/44+\varepsilon})$. 

From this point on the proof of \Cref{thm:2mstd} is completely analogous to that of \Cref{thm:2mspin}. In this case we have $\eta = \frac 14$, and $C\asymp T_2^4$. By \eqref{eq:Fourier_L} we have $a_n(\varpi) = \sum_{k^2 \mid n} A_\varpi(n/k^2,1)$. By the same analysis, we get
\[
\sum_{\substack{T_1 \le |\mu_1(\varpi)| \le 2T_1\\ T_2 \le |\mu_2(\varpi)|\le 2T_2}} \vb{L(\tfrac 12,\varpi,\Std)}^2 \ll T_2^\varepsilon \rb{\max_{|t|\le T_2^\varepsilon} \sum_{\substack{T_1 \le |\mu_1(\varpi)| \le 2T_1\\ T_2 \le |\mu_2(\varpi)|\le 2T_2}} \vb{\sum_{n\le T_2^{2+\varepsilon}} \frac{a_n(\varpi)}{n^{1/2+\varepsilon+it}}}^2 + \sum_{\substack{T_1 \le |\mu_1(\varpi)| \le 2T_1\\ T_2 \le |\mu_2(\varpi)|\le 2T_2}} \eta_j^{-2M} C^{2\kappa+\varepsilon}},
\]
where $\kappa = \tfrac 14$ assuming Langlands conjecture, and $\kappa = \frac{29}{44}$ otherwise. We bound the first term using the large sieve inequality (\Cref{thm:ls}) with $N = T_2^{2+\varepsilon}$ and $\alpha(n) = \sum_{k^2n \le T_2^{2+\varepsilon}} (k^2n)^{-1/2-\varepsilon-it}$, which gives
\[
\max_{|t|\le T_2^\varepsilon} \sum_{\substack{T_1 \le |\mu_1(\varpi)| \le 2T_1\\ T_2 \le |\mu_2(\varpi)|\le 2T_2}} \vb{\sum_{n\le T_2^{2+\varepsilon}} \frac{a_n(\varpi)}{n^{1/2+\varepsilon+it}}}^2 \ll (T_1T_2)^\varepsilon \rb{T_1^2T_2^4 + (T_2^2)^{7/2}T_1T_2^{7/4}} \ll T_1^{2+\varepsilon}T_2^{4+\varepsilon} + T_1^{1+\varepsilon}T_2^{35/4+\varepsilon}.
\]
On the other hand, we use the Weyl law (\Cref{thm:Weyl_law}) and \Cref{lem:adlv} to bound the second term:
\[
\sum_{\substack{T_1 \le |\mu_1(\varpi)| \le 2T_1\\ T_2 \le |\mu_2(\varpi)|\le 2T_2}} \eta_j^{-2M} C^{2\kappa+\varepsilon} \ll \begin{cases} T_1^{2+\varepsilon} T_2^{6+\varepsilon} & \text{ assuming Langlands conjecture},\\ T_1^2 T_2^{102/11} & \text{ otherwise.}\end{cases}
\]
Combining the bound finishes the proof of \Cref{thm:2mstd}.
\end{proof}

\subsection{Quantitative quasi-orthogonality}

To prove \Cref{thm:quasi_orthogonality}, we pick a smooth, nonnegative function $g:\R_+\to\R$ with support $[1/2,3]$, and set
\begin{equation}\label{eq:hT_def}
h_{T_1,T_2} := \int_{\R_+^2} g\rb{\frac{\tau_1}{T_1}} g\rb{\frac{\tau_2}{T_2}} F_{\tau_1,\tau_2,1,1,T_1,T_2} d\tau_1 d\tau_2,
\end{equation}
with $F_{\tau_1,\tau_2,1,1,T_1,T_2}$ as in \eqref{eq:ntf}. 

First we examine the contribution from the continuous spectrum. Using the lower bounds \eqref{eq:zlb} and~\eqref{eq:lflb} for the $L$-functions on the line $\Re s=1$, the Weyl law on $\GL(2)$ and known bounds towards the Ramanujan conjecture \cite{ks} for $\GL(2)$, we conclude that the contribution from the continuous spectrum is $\ll T_1^{1+\varepsilon}T_2^{2+\varepsilon}(m_1m_2n_1n_2)^{\theta+\varepsilon}$.

On the arithmetic side, It follows from \Cref{prop:arith_integrated} that the term $\mathcal K_{\alpha\beta\alpha}$ contributes
\[
\ll T_1^{1+\varepsilon}T_2^{5/3+\varepsilon}d_1^{7/3+\varepsilon}d_2^{2/3+\varepsilon}(m'_1n'_1)^{2/3+\varepsilon}{m'_2}^{1/2+\varepsilon}{n'_2}^{-1/3+\varepsilon},
\]
the term $\mathcal K_{\beta\alpha\beta}$ contributes
\[
T_1^{1+\varepsilon}T_2^{1+\varepsilon}d_1^{1/2+\varepsilon}d_2^{3/2+\varepsilon}{m'_1}^\varepsilon{n'_1}^{-1/2+\varepsilon}(m'_2n'_2)^{1/4+\varepsilon},
\]
and the term $\mathcal K_{w_0}$ contributes
\[
T_1^{1+\varepsilon}T_2^{7/4+\varepsilon}(m_1n_1)^{5/4+\varepsilon}(m_2n_2)^{1+\varepsilon}.
\]
\Cref{thm:quasi_orthogonality} then follows, noting that the contributions from $\mathcal K_{\alpha\beta\alpha}$ and $\mathcal K_{\beta\alpha\beta}$ are always dominated by that of $\mathcal K_{w_0}$. \qed

\subsection{Low-lying zeros}

We prove here \Cref{thm:llz} on the low-lying zeros and their types of symmetry, towards the Rudnick-Sarnak density conjecture. We start by appealing to the classical explicit formula of Weyl \cite[(4.11)]{ils} that rephrases the sum over zeros into a sum of spectral parameters over primes. 

\begin{prop} 
Let $\phi$ be a Scwhartz function with compactly supported Fourier transform, and $\varpi$ a Hecke--Maaß cusp form of $\GSp(4)$. Then we have 
\begin{equation}
\label{eq:explicit-formula}
D_\bullet(\varpi, \phi) = \widehat{\phi}(0) - \frac{2}{\log(c_\bullet)} \sum_p \sum_{k \geq 1} \left( \sum_i u_{\varpi,\bullet,p,i}^k  \right) \widehat{\phi}\left( \frac{k\log p}{\log c_\bullet}\right) \frac{\log p}{p^{k/2}} + O\rb{\frac{1}{\log c_\bullet}},
\end{equation}
where $\bullet\in\cb{\Spin,\Std}$, $c_\bullet$ is given as in \Cref{thm:llz}, and $u_{\varpi,\Spin,p,i}$, $u_{\varpi,\Std,p,i}$ are given as in \eqref{eq:up_spin} and~\eqref{eq:up_std} respectively.
\end{prop}
Since $\widehat\phi$ has compact support, the sums over both $p$ and $k$ in \eqref{eq:explicit-formula} are actually finite. We can split the sum and study for each $k\in\N$ the corresponding sum over $p$ and $i$, that we will denote by~$P_{\bullet,k}(\varpi, \phi)$, so that we have
\begin{equation}\label{eq:Pk_def}
D_\bullet(\varpi, \phi) = \widehat{\phi}(0) - \sum_{k \geq 1} P_{\bullet,k}(\varpi, \phi) + O(1/\log c_\bullet).
\end{equation}

\begin{rk}
The sums of spectral parameters $P_{\bullet,k}(\varpi, \phi)$ over primes are therefore critical in understanding the distribution of low-lying zeros. Using the Hecke relations, these relate to automorphic coefficients, whose analogous sums over primes have to be bounded. The standard approach in the literature is to bound the sum termwise for large $k$. For this to be possible, we need that the~$p^k$-th Fourier coefficient contributes less than the factor~$p^{-k/2}$, so that the geometric sum converges. This in turn relies on either the assumption of the Ramanujan conjecture (as in~\cite{ils}), or the known bound of Luo, Rudnick, and Sarnak \cite{lrs} for the $\GL(n)$ case. For the $\GSp(4)$ $L$-functions, we know through the functorial lift of Gan and Takeda \cite{gt} that this argument also applies to $L(s, \varpi, \Spin)$, but remains conjectural for $L(s, \varpi, \Std)$. In order to obtain unconditional results for the standard $L$-functions, it is obligatory to exploit the harmonic averages and use the Kuznetsov trace formula for all $k\in\N$, which we do below.
\end{rk}

To apply the Kuznetsov formula, we make use of the following lemma which relates the sum over spectral parameters $\sum_i u_{\varpi,\bullet,p,i}^k$ with Fourier coefficients $A_\varpi(M)$. 
\begin{lem}
Let $\varpi$ be an arithmetically normalised Hecke--Maaß cusp form of $\GSp(4)$. Then we have
\begin{align*}
\sum_{i=1}^4 u_{\varpi,\Spin,p,i} &= A_\varpi(1,p),\\
\sum_{i=1}^4 u_{\varpi,\Spin,p,i}^2 &= A_\varpi(1,p^2) - A_\varpi(p,1) - 1,\\
\sum_{i=1}^4 u_{\varpi,\Spin,p,i}^3 &= A_\varpi(1,p^3) - A_\varpi(p,p),\\
\sum_{i=1}^4 u_{\varpi,\Spin,p,i}^k &= A_\varpi(1,p^k) - A_\varpi(p,p^{k-2}) + A_\varpi(p,p^{k-4}) - A_\varpi(1,p^{k-4}), & k\ge 4,
\end{align*}
and 
\begin{align*}
\sum_{i=1}^5 u_{\varpi,\Std,p,i} &= A_\varpi(p,1),\\
\sum_{i=1}^5 u_{\varpi,\Std,p,i}^2 &= A_\varpi(p^2,1) - A_\varpi(1,p^2)+1\\
\sum_{i=1}^5 u_{\varpi,\Std,p,i}^k &= A_\varpi(p^k,1) - A_\varpi(p^{k-2},p^2) + A_\varpi(p^{k-3},p^2) - A_\varpi(p^{k-3},1) + 1, & k\ge 3.
\end{align*}
\end{lem}
\begin{proof}
Using the generating function \eqref{eq:Fourier_coefficient}, and induction.
\end{proof}

\begin{proof}[Proof of \Cref{thm:llz}] We compute the weighted average
\begin{equation}\label{eq:llz_weighted}
H^{-1}\sum_{\varpi}  D_\bullet(\varpi,\phi) \frac{h_{T_1,T_2}(\mu(\varpi))}{\|\varpi\|^2},
\end{equation}
where $h_{T_1,T_2}$ is as in \eqref{eq:hT_def}, and $H = H_{T_1,T_2} = \sum_{\varpi} h_{T_1,T_2}(\mu(\varpi))/\|\varpi\|^2$. Using the Weyl law (\Cref{thm:Weyl_law}), we find $H \asymp T_1^2 T_2^4$. Next we use \eqref{eq:Pk_def} and rewrite \eqref{eq:llz_weighted} as
\[
\widehat\phi(0) - H^{-1} \sum_{k \ge 1} \sum_{\varpi} P_{\bullet,k}(\varpi,\phi) \frac{h_{T_1,T_2}(\mu(\varpi))}{\|\varpi\|^2} + O\rb{\frac{1}{\log c_\bullet}}.
\]
For $k=1$ we find
\begin{align*}
P_{\Spin,1}(\varpi, \phi) &= \frac{2}{\log c_{\Spin,T_1,T_2}} \sum_p A_{\varpi}(1,p) \frac{\log p}{p^{1/2}} \widehat{\phi}\left( \frac{\log p}{\log c_{\Spin,T_1,T_2}}\right).
\end{align*}
As we have $c_{\Spin,T_1,T_2} \asymp T_1^2T_2^2$, we get
\begin{equation}\label{eq:Spink1}
    H^{-1} \sum_{\varpi} P_{\Spin,k}(\varpi,\phi) \frac{h_{T_1,T_2}(\mu(\varpi))}{\|\varpi\|^2} \ll \frac 2{2H\log(T_1T_2)} \sum_{p\le (T_1T_2)^{2\delta}} \sum_{\varpi} A_{\varpi}(1,p) \frac{h_{T_1,T_2}(\mu(\varpi))}{\|\varpi\|^2} \frac{\log p}{p^{1/2}}. 
\end{equation}
Using \Cref{thm:quasi_orthogonality}, we bound \eqref{eq:Spink1} by
\begin{align*}
&\ll (T_1T_2)^\varepsilon T_1^{-2}T_2^{-4} \sum_{p\le(T_1T_2)^{2\delta}} \rb{T_1T_2^2p^\theta + T_1T_2^{7/4}p} \frac{\log p}{p^{1/2}}\\
&\ll (T_1T_2)^\varepsilon T_1^{-2}T_2^{-4} \rb{T_1T_2^2 (T_1T_2)^{2(\theta+1/2)\delta} + T_1T_2^{7/4}(T_1T_2)^{6\delta}}.
\end{align*}
Using that $T_2\asymp T_1^t$, the expression above is bounded by
\begin{align*}
    \ll T_1^\varepsilon T_1^{-2-4t} \rb{T_1^{1+2t + 2(1+t)(\theta+1/2)\delta} + T_1^{1+7t/4+(1+t)6\delta}}.
\end{align*}
For this to be of size $o(1)$, we need $\delta < \frac{4+9t}{12(1+t)}$.

Meanwhile, for the standard function we have
\begin{align*}
    P_{\Std,1}(\varpi, \phi) &= \frac{2}{\log c_{\Std,T_1,T_2}} \sum_p A_{\varpi}(p,1) \frac{\log p}{p^{1/2}} \widehat{\phi}\left( \frac{\log p}{\log c_{\Std,T_1,T_2}}\right).
\end{align*}
As we have $c_{\Std,T_1,T_2} \asymp T_2^4$, we get
\begin{equation}\label{eq:Stdk1}
    H^{-1} \sum_{\varpi} P_{\Std,k}(\varpi,\phi) \frac{h_{T_1,T_2}(\mu(\varpi))}{\|\varpi\|^2} \ll \frac 2{4H\log(T_2)} \sum_{p\le T_2^{4\delta}} \sum_{\varpi} A_{\varpi}(p,1) \frac{h_{T_1,T_2}(\mu(\varpi))}{\|\varpi\|^2} \frac{\log p}{p^{1/2}}. 
\end{equation}
Using \Cref{thm:quasi_orthogonality}, we bound \eqref{eq:Stdk1} by
\begin{align*}
&\ll (T_1T_2)^\varepsilon T_1^{-2} T_2^{-4} \sum_{p\le T_2^{4\delta}} \rb{T_1T_2^2p^\theta + T_1T_2^{7/4}p^{5/4}} \frac{\log p}{p^{1/2}}\\
&\ll (T_1T_2)^\varepsilon T_1^{-2} T_2^{-4} \rb{T_1T_2^{2+4(\theta+1/2)\delta} + T_1T_2^{7/4+7\delta}}.
\end{align*}
Using that $T_2\asymp T_1^t$, the expression above is bounded by
\begin{align*}
    \ll T_1^\varepsilon T_1^{-2-4t} \rb{T_1^{1+2t + 4t(\theta+1/2)\delta} + T_1^{1+7t/4+7t\delta}}.
\end{align*}
For this to be of size $o(1)$, we need $\delta < \frac{4+9t}{28t}$.

For $k\ge 2$, the expressions for $\sum_i u_{\varpi,\bullet,p,i}^k$ are linear combinations of Fourier coefficients $A_\varpi(p^i,p^j)$, and possibly a constant $1=A_\varpi(1,1)$. For instance, for $k=2$ we have
\begin{align*}
P_{\Spin,2}(\varpi, \phi) & =\frac{2}{\log c_{\Spin,T_1,T_2}} \sum_p (A_\varpi(1,p^2) - A_\varpi(p,1) - 1) \frac{\log p}{p} \widehat{\phi}\rb{\frac{2\log p}{\log c_{\Spin,T_1,T_2}}} ,\\
P_{\Std,2}(\varpi, \phi) & = \frac{2}{\log c_{\Std,T_1,T_2}} \sum_p (A_\varpi(p^2,1) - A_\varpi(1,p^2) + 1) \frac{\log p}{p} \widehat{\phi}\rb{\frac{2\log p}{\log c_{\Std,T_1,T_2}}}.
\end{align*}
The constants appearing in the parentheses of $P_{\bullet,2}(\varpi, \phi)$ are of critical importance, since they contribute non-trivially to the limiting behaviour of the one-level density of the zeros, hence to the final type of symmetry; see \cite{sst} for a discussion about the invariants that determine the type of symmetry of a family. Precisely, from \cite[(4.20)]{ils}, we find
\begin{equation}\label{eq:symmetry}
\frac{2}{\log c_\bullet} \sum_p \frac{\log p}{p} \widehat{\phi}\left( \frac{2\log p}{\log c_\bullet}\right) = \frac{1}{2} \phi(0) + O\rb{\frac{1}{\log c_\bullet}}.
\end{equation}

To prove \Cref{thm:llz}, it remains to show that other terms contribute negligibly. The contribution from the constants arising from $\sum_i u_{\varpi,\bullet,p,i}^k$ for $k\ge 3$ is bounded by
\[
\frac{2}{\log c_\bullet} \sum_{k\ge 3} \sum_p \frac{\log p}{p^{k/2}} \widehat\phi \rb{\frac{k\log p}{\log c_\bullet}} = O\rb{\frac 1{\log c_\bullet}}.
\]
To compute the contribution from the other Fourier coefficients, we need to evaluate the expressions of the form
\begin{equation}\label{eq:higher_k}
\frac{2}{H\log c_\bullet} \sum_p \sum_{\varpi} A_{\varpi}(p^i,p^j) \frac{h_{T_1,T_2}(\mu(\varpi))}{\|\varpi\|^2} \frac{\log p}{p^{k/2}}\widehat\phi \rb{\frac{k\log p}{\log c_\bullet}}
\end{equation}
with $1\le i+j\le k$, using \Cref{thm:quasi_orthogonality}. For the spinor $L$-function, \eqref{eq:higher_k} is bounded by
\begin{align*}
&\ll (T_1T_2)^\varepsilon T_1^{-2} T_2^{-4} \sum_{p\le (T_1T_2)^{2\delta/k}} \rb{T_1T_2^2p^{k\theta} + T_1T_2^{7/4}p^{5k/4}} \frac{\log p}{p^{k/2}}\\
&\ll (T_1T_2)^\varepsilon T_1^{-2} T_2^{-4} \rb{T_1T_2^2(T_1T_2)^{\max\{0,2(k\theta-k/2+1)\delta/k\}} T_1T_2^{7/4}(T_1T_2)^{2(3k/4+1)\delta/k}}.
\end{align*}
Using that $T_2\asymp T_1^t$, the expression above is bounded by
\[
\ll T_1^\varepsilon T_1^{-2-4t} \rb{T_1^{1+2t+(1+t)\max\{0,2(k\theta-k/2+1)\delta/k} + T_1^{1+7t/4+2(1+t)(3k/4+1)\delta/k}}.
\]
When $\delta < \frac{4+9t}{12(1+t)}$, the contribution is then bounded by
\[
\ll T_1^{(4+9t)(4-3k)/24k} = o(1)
\]
for $k\ge 2$. 

Similarly, for the standard $L$-function, \eqref{eq:higher_k} is bounded by
\begin{align*}
&\ll (T_1T_2)^\varepsilon T_1^{-2} T_2^{-4} \sum_{p\le T_2^{4\delta/k}} \rb{T_1T_2^2p^{k\theta}+T_1T_2^{7/4}p^{5/4}} \frac{\log p}{p^{k/2}}\\
&\ll (T_1T_2)^\varepsilon T_1^{-2} T_2^{-4} \rb{T_1T_2^{2+\max\{0,4(k\theta-k/2+1)\delta/k\}} + T_1T_2^{7/4+4(3k/4+1)\delta/k}}.
\end{align*}
Using that $T_2\asymp T_1^t$, the expression above is bounded by
\[
\ll T_1^\varepsilon T_1^{-2-4t} \rb{T_1^{1+2t+\max\{0,4(k\theta-k/2+1)\delta/k\}t} + T_1^{1+7t/4+4(3k/4+1)t\delta/k}}.
\]
When $\delta < \frac{4+9t}{28t}$, the contribution is then bounded by
\[
\ll T_1^{(4+9t)(1-k)/7k} = o(1)
\]
for $k\ge 2$. This finishes the proof of \Cref{thm:llz}.
\end{proof}

\subsection*{Acknowledgements} We are grateful to Edgar Assing, Valentin Blomer, Farrell Brumley and Ralf Schmidt for enlightening discussions. The first author is supported by the ERC Advanced Grant  101054336, and Germany's Excellence Strategy grant EXC-2047/1 - 390685813. The second author is supported by the Labex CEMPI (ANR-11-LABX-0007-01) and the CNRS (PEPS). The third author is supported by the Czech Science Foundation GAČR grant 21-00420M, and the Charles University programme PRIMUS/24/SCI/010 and PRIMUS/25/SCI/008. We thank Charles University, Université de Lille and Universität Bonn for providing very good working environments.

\textit{}\bibliographystyle{abbrv}
\bibliography{reference}

\end{document}